\documentclass[english,11pt,reqno,twoside]{amsart}

\usepackage{amssymb,amsmath,amsthm,soul,color}
\usepackage{graphicx}
\usepackage[noadjust]{cite}
\usepackage{t1enc}
\usepackage[utf8]{inputenc}
\usepackage{braket}
\usepackage[T1]{fontenc}
\usepackage{amsfonts}
\usepackage{a4,indentfirst,latexsym,graphics}
\usepackage[english]{babel}
\usepackage{mathtools}
\usepackage{fixmath}

\usepackage[left=2.1cm,right=2.1cm,top=2cm,bottom=2cm]{geometry}

\usepackage{url}

\usepackage{mathrsfs}

\usepackage{enumitem}
\usepackage[colorlinks=true,urlcolor=blue,
citecolor=red,linkcolor=blue,linktocpage,pdfpagelabels,
bookmarksnumbered,bookmarksopen]{hyperref}

\usepackage[hyperpageref]{backref}
\usepackage[colorinlistoftodos]{todonotes}

\usepackage[normalem]{ulem}

\usepackage{lmodern}

\newtheorem{thm}{Theorem}[section]

\newtheorem{lem}[thm]{Lemma}
\newtheorem{prop}[thm]{Proposition}

\theoremstyle{definition}

\theoremstyle{remark}
\newtheorem{rem}[thm]{Remark}

\numberwithin{equation}{section}

\DeclareMathSymbol{\C}{\mathalpha}{AMSb}{"43}

\newcommand{\eps}{\varepsilon}

\newcommand{\R}{{\mathbb{R}}}

\newcommand{\inte}{\int_{\mathbb{R}^N}}

\newcommand{\bsub}{\begin{subequations}}
	\newcommand{\esub}{\end{subequations}$\!$}

\begin{document}
	\title[Boosted ground states for pseudo-relativistic Schr\"odinger equations]{Boosted Ground States for a Pseudo-Relativistic Schr\"odinger Equation with a double power nonlinearity}

	\author[P. d'Avenia]{Pietro d'Avenia}
	\address{P. d'Avenia
		\newline\indent Dipartimento di Meccanica, Matematica e Management,\newline\indent
		Politecnico di Bari
		\newline\indent
		Via Orabona 4,  70125  Bari, Italy}
	\email{pietro.davenia@poliba.it}

	\author[A. Pomponio]{Alessio Pomponio}
	\address{A. Pomponio
		\newline\indent Dipartimento di Meccanica, Matematica e Management,\newline \indent
		Politecnico di Bari
		\newline\indent
		Via Orabona 4,  70125  Bari, Italy}
	\email{alessio.pomponio@poliba.it}

	\author[G. Siciliano]{Gaetano Siciliano}
	\address{G. Siciliano
		\newline\indent Dipartimento di Matematica,\newline \indent
		Univerist\`a degli Studi di Bari Aldo Moro, 
		\newline\indent
		Via Orabona 4,  70125  Bari, Italy}
	\email{gaetano.siciliano@uniba.it}

	\author[L. Yang]{Lianfeng Yang}
	\address{L. Yang
		\newline \indent School of Mathematics, 
		\newline \indent Guangxi University, 
		\newline \indent Nanning, Guangxi, P. R. China
	}
	\email{yanglianfeng2021@163.com}

	

	\begin{abstract}
		In this paper, we investigate the existence and limit behaviours of travelling solitary waves of the form $\psi(t,x)=e^{i\lambda t}\varphi\left(x-vt\right)$ to the nonlinear pseudo-relativistic Schr\"odinger equation
		$$i\partial_t \psi=(\sqrt{-\Delta+m^2})\psi - |\psi|^{\frac{2}{N}}\psi-\mu|\psi|^{q}\psi~~\text{ on }\mathbb{R}^N,$$ 
		for $m\ge 0$ and $|v|<1$.
		To this end, we introduce and analyse an associated 
		constrained variational problem,  whose minimizers are termed boosted ground states and the parameter $\lambda$ is obtained as a Lagrangian multiplier. We first provide a complete classification for the existence and nonexistence of such boosted ground states. Based on this classification, we then study several limiting profiles, for which the exact blow-up rate is also established.
	\end{abstract}
	
	\maketitle

	\begin{center}
		\begin{minipage}{8.5cm}
			\small
			\tableofcontents
		\end{minipage}
	\end{center}
	
	\bigskip

	\section{Introduction}
	In recent years, there has been a growing interest in fractional operators
	due to their important applications and to their nonlocal nature which also makes them very interesting from a mathematical point of view (see \cite{DS,DPV12,FF,BP24} together with their related literature).

	For example, in the study of pseudo-relativistic boson stars, the so-called pseudo-relativistic Hartree equation,
	the pseudo-relativistic operator 
	\[
	\sqrt{-\Delta + m^2}\, u := \mathcal{F}^{-1}\left( \sqrt{|\xi|^2 + m^2} \, \mathcal{F}(u)(\xi) \right), \quad u \in \mathcal{S}(\mathbb{R}^N),
	\]
	appears  (see e.g. \cite{A16,CS16,CZN11,S19,CHS18,CZN13,ES07,FJL07,FL07,GZ17,GZ20,L09,LL11,LT84,LY87,LY88,P14,LS10}). Here, $\mathcal{S}(\mathbb{R}^N)$ denotes the Schwartz space of rapidly decaying functions, while $\mathcal{F}$ and $ \mathcal{F}^{-1}$ represent the Fourier transform and its inverse, respectively.
	From the physical viewpoint, the operator $\sqrt{-\Delta + m^2}$ not only describes the kinetic and rest energy of a relativistic particle of mass $m> 0$, but is associated with the Hamiltonian $\mathcal{H}=\sqrt{p^2 c^2 + m^2 c^4}$ of a free relativistic particle of mass $m$ by the usual quantization $p \mapsto i\hbar \Delta$, after rescaling units so that $\hbar=1$ and $c=1$.

	Starting from \cite{BGLV19,FJL07,HS,KLR13}, the study of travelling solitary waves of the form
	\begin{equation}\label{eq1.02}
		\psi(t,x)=e^{i\lambda t}\varphi\left(x-vt\right)
	\end{equation}
	with $\lambda\in \R$ and a travelling velocity $v\in \R^3$ satisfying $|v|<1$ (i.e., below the speed of light in our units) attracted a lot of interest. In particular, motivated by the above works, the authors in \cite{HYZ24} considered travelling solitary waves
	for the  nonlinear pseudo-relativistic Schr\"odinger equation with focusing (i.e. $\nu>0$) power type nonlinearity
	$$i\partial _t \psi=(\sqrt{-\Delta+m^2}-m)\psi-\nu|\psi|^{p}\psi,~(t,x)\in\mathbb{R}^+\times\mathbb{R}^3.$$
	By studying a constrained minimization problem, they provided a complete classification of the existence and nonexistence of minimizers according to the constant 
	$\nu$ and  the mass critical exponent for $p=\frac{2}{3}$, and analysed the asymptotic behaviour of the minimizers. See also \cite{EHP}, for the study of travelling wave solutions in the presence of an anisotropic operator.
	
	Inspired by the significant role played by the mass critical exponent $2+\frac2N$ and the optimizer of Gagliardo-Nirenberg inequality in the field of normalized solution problems, 
	in this paper we study the following double power nonlinear pseudo-relativistic Schr\"odinger equation 
	\begin{equation}\label{eq1.1}
		i\partial _t \psi=(\sqrt{-\Delta+m^2})\psi - |\psi|^{\frac{2}{N}}\psi-\mu|\psi|^{q}\psi,~(t,x)\in\mathbb{R}^+\times\mathbb{R}^N,
	\end{equation}
	where $\psi (t,x)$ is a complex-valued wave function and the parameter $0<q<\frac2N$ with 
	$N\geq2$ and $\mu\in \R$.
	To the best of our knowledge, there seem to be very few results available on this topic. In particular, for the special case $m=0$, i.e., the so-called half-wave equation, we refer interested readers to \cite{BF25,BGLV19,BGV18,GL22,KLR13,LZW,ZL22} and the references therein.

	Specifically, we employ variational methods and the concentration-compactness principle to investigate travelling solitary waves for equation \eqref{eq1.1}. To this end, plugging the ansatz \eqref{eq1.02} into \eqref{eq1.1} yields
	\begin{equation}\label{eq1.2}
		(\sqrt{-\Delta+m^2})\varphi+i(v\cdot \nabla )\varphi+\lambda \varphi =|\varphi|^{\frac{2}{N}}\varphi+\mu|\varphi|^{q}\varphi,
	\end{equation}
	and we look for $(\lambda,\varphi)\in \mathbb R\times H^{\frac{1}{2}}(\mathbb R^N)$ that solve \eqref{eq1.2}. To be precise, we consider the following constrained minimization problem
	\begin{equation}\label{eq1.3}
		e(a):=\inf_{\varphi\in \mathcal{S}_a} E(\varphi ),
	\end{equation}
	where the energy functional $E$ is given by
	\begin{equation}\label{eq1.4}
		E(\varphi ):=\frac{1}{2}\int_{\mathbb{R}^N}\bar{{\varphi}}(\sqrt{-\Delta+m^2}+iv\cdot \nabla)  \varphi dx-\frac{N}{2N+2}\|\varphi\|_{2+\frac2N}^{2+\frac2N}-\frac{\mu}{q+2}\|\varphi\|_{q+2}^{q+2},
	\end{equation}
	and the constraint $\mathcal{S}_a$ is defined as
	\begin{equation*}
		\mathcal{S}_a:=\left\{\varphi \in H^{\frac{1}{2}}(\mathbb{R}^N):\|\varphi\|_2^2=a\right\}.
	\end{equation*}
	In this way $\lambda$ in \eqref{eq1.2} appears as Lagrangian multiplier.

	Before stating our main results, we first give some conventions and notations.
	
	Henceforth, for simplicity and convenience, unless otherwise mentioned, we will always assume that $0<q<\frac2N$ with $N\ge2$ and $v\in \mathbb{R}^N$ with $|v|<1$.
	
	As in \cite{FJL07,HYZ24}, any minimizer of problem (\ref{eq1.3}), if it exists, is called a boosted ground state (specifically, for the travelling velocity $v=0$, a minimizer of problem (\ref{eq1.3}), if it exists, is usually called a ground state). 
	
	Throughout the article, we set
	\begin{equation}\label{astar}
		a^\ast:=\|Q\|_2^2,
	\end{equation}
	where $Q$ is the unique ground state of
	\begin{equation}\label{eqv0}
		\sqrt{-\Delta}u+u=|u|^{\frac{2}{N}}u
		\quad \text{in }\R^N,
	\end{equation} see \cite{FLS}.
	When we deal with equation
	\begin{equation}\label{eq1.6}
		(\sqrt{-\Delta}+iv\cdot \nabla)u+u=|u|^{\frac{2}{N}}u \quad \text{in }\R^N,
	\end{equation}
	the uniqueness of the ground state is not known. Anyway, the $L^2$-norm of all possible ground states $Q_v$ is preserved and so
	\begin{equation}\label{astarv}
		a^\ast_v:=\|Q_v\|_2^2 
	\end{equation}
	is well-defined (see Proposition \ref{prop1.1} for details).
	
	For any $\tau>0$, we define
	\begin{equation}\label{h}
		h(\tau):=\frac{am^2}{4a^\ast_v\tau}\int_{\mathbb{R}^N}\frac{|\hat{Q}_v(k)|^2}{|k|}dk+\frac{aN\tau}{2}\left[1-\left(\frac{a}{a^\ast_v}\right)^{\frac{1}{N}} \right]-\frac{\mu\tau^{\frac{Nq}{2}}}{q+2} \left(\frac{a}{a^\ast_v}\right)^{\frac{q+2}{2}}\|Q_v\|_{q+2}^{q+2},
	\end{equation}
	where $\hat{Q}_v(k)$ denotes the Fourier transform of $Q_v(x)$. Observe that $h$ has a positive minimum and, in particular, 
	when $\mu=0$ we have
	\begin{equation}\label{minh}
		\min_{\tau>0}h(\tau)=am \left(1-\left(\frac{a}{a^\ast_v}\right)^{\frac{1}{N}} \right)^\frac{1}{2}\left(\frac{N}{2a^*_v}
		\int_{\mathbb{R}^N}\frac{|\hat{Q}_v(k)|^2}{|k|}dk\right)^{\frac{1}{2}}.
	\end{equation}
	Furthermore, let
	
	\begin{equation}\label{mu_1}
		\mu_m^*:= -m^{\frac{2-Nq}{2}} C(q,N,v)
		\sup_{\phi\in \mathcal{S}_a\cap H^{1}(\R^N)}\frac{\|\phi\|_{2+\frac2N}^{\frac{(N+1)(4-Nq)}N}}{\|\nabla \phi\|_2^{2-Nq}\|\phi\|_{q+2}^{q+2}},
	\end{equation}
	where
	$$C(q,N,v)= \frac{q+2}{(4-Nq)^{2-\frac{Nq}2}}\left(\frac{2(2-Nq)}{\sqrt{ 1-|v|^2}}\right)^{1-\frac{Nq}2}\left(\frac{N}{N+1}\right)^{2-\frac{Nq}{2}}>0.$$
	We observe that, as shown in Lemma \ref{lem1.6}, $\mu_m^*\in(-\infty,0)$.
	
	Our first result concerns the existence and non-existence of minimizers for problem \eqref{eq1.3} and we obtain also  some estimates on $e(a)$.
	Here $C_{v,N,q}$ denotes the sharp constant  of the Gagliardo-Nirenberg inequality \eqref{eq2.19} (see Remark \ref{re1} for its explicit expression).
	
	\begin{thm}\label{Thm1}
		\begin{enumerate}[label=(\arabic*),ref=\arabic*]
			\item \label{E1} Suppose that $a^*_v>a>0$, $m\ge0$ and $\mu>0$, then there exists a minimizer for problem \eqref{eq1.3}. Moreover, we have the following energy estimates:
			\begin{equation*}
				-\frac{2-Nq}{2}\left({Nq}\right)^{\frac{Nq}{2-Nq}}{\left[1-\left(\frac{a}{a^\ast_v}\right)^{\frac1N}\right]}^{-\frac{Nq}{2-Nq}}{\left(\frac{\mu C_{v,N,q}}{q+2}a^{\frac{q+2-Nq}{2}}\right)}^{\frac{2}{2-Nq}}
				\le e(a)< \frac{\sqrt{1-|v|^2}}{2}ma.
			\end{equation*}

			\item \label{E2} 
			For all $a^*_v\ge a>0$, $m>0$ and $\mu<0$, the following energy estimates hold:
			\begin{equation}\label{cacata}
				\frac{ma}{2}\sqrt{\left[1-\left(\frac{a}{a^\ast_v}\right)^{\frac1N}\right](1-|v|)} \sqrt{1+|v|+\left(\frac{a}{a^\ast_v}\right)^{\frac1N}(1-|v|)}
				\le e(a)\le \min_{\tau>0}h(\tau).
			\end{equation}
			Moreover, if $\mu_m^*<\mu<0$, then problem \eqref{eq1.3} has a minimizer and
			\begin{equation}\label{nacacata}
				e(a)< \frac{\sqrt{1-|v|^2}}{2}ma.
			\end{equation}

			\item \label{E3} Let $\mu=0$.
			\begin{enumerate}[label=(\roman{*}),ref=\roman{*}]
				\item \label{E3i} If $a^*_v>a>0, m>0$, then there exists a minimizer for problem \eqref{eq1.3}, and \eqref{cacata} and \eqref{nacacata} hold.
				\item \label{E3ii} If $a=a^*_v, m\ge 0$, then $e(a^*_v)=0$; moreover, if $m=0$ problem \eqref{eq1.3} has a minimizer $Q_v$.
			\end{enumerate}
			
			\item \label{NE4} Problem \eqref{eq1.3} has no minimizer in each of these cases:
			\begin{enumerate}[label=(\roman{*}),ref=\roman{*}]
				\item \label{eq:NE4i} $a=a^*_v, m\ge0, \mu>0$; \smallskip
				\item \label{eq:NE4iii}$a>a^*_v, m\ge0, \mu\in\R$; \smallskip
				\item \label{eq:NE4ii} $a^*_v\ge a>0, m=0, \mu<0$; \smallskip
				\item \label{eq:NE4iv}$a=a^\ast_v, m>0, \mu=0$; \smallskip
				\item \label{eq:NE4v}$a^\ast_v>a>0, m=0, \mu=0$.
			\end{enumerate}
			
			Furthermore, the following energy estimates hold:
			\begin{equation*}
				e(a)=
				\left\{ \begin{array}{lll}
					-\infty, & \text{in cases } \eqref{eq:NE4i} \text{ and } \eqref{eq:NE4iii},\\
					0, & \text{in cases } \eqref{eq:NE4ii}, \eqref{eq:NE4iv}, \text{ and } \eqref{eq:NE4v}.
				\end{array}\right.
			\end{equation*}
			
		\end{enumerate}
		
	\end{thm}

	Building upon Theorem \ref{Thm1}, our next goal is to characterize several limiting profiles. In view of Theorem \ref{Thm1}-\eqref{E1} and Theorem \ref{Thm1}-\eqref{NE4}-\eqref{eq:NE4i}, we have existence of minimizer  for $a_v^*>a$ and non-existence for $a_v^*=a$. So, it is natural to study what happens when $a_n \nearrow a^\ast_v$. We obtain a blow-up behaviour for the boosted ground states and that the associated energy goes to $-\infty$ with a precise rate, as expected.
	\begin{thm}\label{Thm3}
		Under the assumptions of  Theorem \ref{Thm1}-\eqref{E1}, let $u_{a}$ be a minimizer of $e(a)$.
		Then, if $a_n \nearrow a^\ast_v$, there exists a sequence $\{y_{n}\}\subset\mathbb{R}^N$ such that, 
		\begin{equation*}
			\eps_{n}^{\frac{N}{2} }u_{a_n}\left(\eps_{n}(\cdot +y_{n})\right)\rightarrow \gamma^{\frac N2}{G_0\left(\gamma \cdot\right)}
			\text{~in~} H^{\frac{1}{2}}(\mathbb{R}^N),
		\end{equation*}
		where $$\eps_{n}=\left[1-\left(\frac{a_n}{a^\ast_v}\right)^{\frac1N}\right]^{\frac{2}{2-Nq}}, \qquad\gamma=\left[\frac{q\mu\|G_0\|_{q+2}^{q+2}}{(q+2)a^\ast_v}\right]^{\frac{2}{2-Nq}},$$ and $G_0$ optimizes the inequality (\ref{eq1.5}) and satisfies equation (\ref{eq1.6}).
		
		Moreover, we also have that
		\begin{equation*}
			\begin{aligned}
				&\lim_{n\to\infty}
				\left[1-\left(\frac{a_n}{a^\ast_v}\right)^{\frac1N}\right]^{\frac{Nq}{2-Nq}}e(a_n)
				=- \frac{2-Nq}{2} \left(\frac{q}{a^*_v}\right)^{\frac{Nq}{2-Nq}}
				{\left(\frac{\mu}{q+2}\|G_0\|_{q+2}^{q+2}\right)}^{\frac{2}{2-Nq}}.
			\end{aligned}
		\end{equation*}
	\end{thm}

	Under the assumption of Theorem \ref{Thm1}-\eqref{E2} and Theorem \ref{Thm1}-\eqref{E3}-\eqref{E3i}, we have some estimates on the infimum energy level and and the existence of a minimizer for any $m>0$ and $\mu$ sufficiently small. On the other hand, by Theorem \ref{Thm1}-\eqref{NE4}-\eqref{eq:NE4ii} and Theorem \ref{Thm1}-\eqref{NE4}-\eqref{eq:NE4v}, in \eqref{eq1.3} the infimum is not achieved and $e(a)=0$.
	Therefore, in the following, we study the asymptotic behaviour of  $e_m(a):=e(a)$ and of the sequence of minimizers, whenever it exists, as $m$ goes to zero.

	\begin{thm}\label{Thm2}
		Under the hypotheses of  Theorem \ref{Thm1}-\eqref{E2} with $\mu\le 0$, we have that
		$$e_m(a)\to 0, \text{ as } m\to 0^+.$$
		In addition, if $a\in(0,a^*_v)$ and, for any $m>0$, we take $\mu_m^*<\mu_m\le 0$ and $u_m$ the corresponding minimizer of $e_m(a)$, then
		$$u_m\to 0 \text{  in } \dot{H}^\frac12(\R^N) \text{ and in }L^\kappa(\R^N),\,\ \forall 2<\kappa\le{2N}/{(N-1)}, \quad \text{as } m\to 0^+.$$

	\end{thm}
	
	Furthermore, compared Theorem \ref{Thm1}-\eqref{E2} with Theorem \ref{Thm1}-\eqref{E3}-\eqref{E3i} and Theorem \ref{Thm1}-\eqref{NE4}-\eqref{eq:NE4iv}, we study the asymptotic behaviour as $\mu\to 0^-$. 
	To make the dependence on the parameter 
	$\mu$ explicit, we rewrite the constrained minimization problem \eqref{eq1.3} as
	\begin{equation*}
		e_{\mu}(a):=\inf_{\varphi\in \mathcal{S}_a} E_{\mu}(\varphi ),
	\end{equation*}
	where the energy functional $E_{\mu}(\varphi)$ is defined as
	\begin{equation*}
		E_{\mu}(\varphi):=\frac{1}{2}\int_{\mathbb{R}^N}\bar{{\varphi}}(\sqrt{-\Delta+m^2}+iv\cdot \nabla)\varphi dx-\frac{N}{2N+2}\|\varphi\|_{2+\frac2N}^{2+\frac2N}-\frac{\mu}{q+2} \|\varphi\|_{q+2}^{q+2}.
	\end{equation*}
	Then we have that
	
	\begin{thm}\label{Thm6}
		Suppose that $a^*_v>a>0$, $m>0$. Let $\phi_{\mu_n}$ be a minimizer of $e_{\mu_n}(a)$ with $\{\mu_n\}\subset(\mu_m^*,0)$ satisfying $\mu_n\to0^-$. Then, up to a subsequence, as $n\to \infty$,
		$$\phi_{\mu_n}\to\phi_0 \text{~in} ~ H^{\frac{1}{2}}(\mathbb{R}^N),$$
		where $\phi_0$ is a minimizer of $e_{0}(a)$.
	\end{thm}
	
	\begin{thm}\label{Thm5}
		Assume that $m>0$, $0>\mu>\mu_m^*$, and let $\varphi_\mu$ be a minimizer of $e_\mu(a^\ast_v)$. Then, if $\mu_n \to0^-$, there exists a sequence $\{y_n\}\subset\mathbb{R}^N$ such that
		\begin{equation*}
			(-\mu_n)^{\frac{N}{2+Nq}}\varphi_{\mu_n}\left((-\mu_n)^{\frac{2}{2+Nq}}(\cdot+y_{n})\right) \rightarrow {\vartheta_0}^{\frac N2}{W_0(\vartheta_0 \cdot)}, \quad
			\text{~in~} H^{\frac{1}{2}}(\mathbb{R}^N),
		\end{equation*}
		where $$\vartheta_0=\left[\frac{q+2}{2Nq\|W_0\|_{q+2}^{q+2}}\inte\bar{W_0}\left(\frac{m^2}{\sqrt{-\Delta}}\right)W_0 dx\right]^{\frac{2}{2+Nq}},$$ and $W_0$ optimizes  the inequality (\ref{eq1.5}) and solves equation (\ref{eq1.6}).
		
		Moreover, the following energy estimate holds:
		\begin{equation*}
			\begin{aligned}
				&\lim_{n\to \infty}(-\mu_n)^{-\frac{2}{2+Nq}}e_{\mu_n}(a^\ast_v)=\frac{Nq+2}{2}
				\left[\frac{1}{2Nq}\inte\bar{W_0}\left(\frac{m^2}{\sqrt{-\Delta}}\right)W_0 dx\right]^{\frac{Nq}{2+Nq}}\left(\frac{\|W_0\|_{q+2}^{q+2}}{q+2}\right)^{\frac{2}{2+Nq}}.
			\end{aligned}
		\end{equation*}

	\end{thm}

	Finally, let us consider $\mu=0$ and $m>0$. 
	Motivated by \cite[Lemma B.1]{FJL07} and \cite[Lemma 2.3]{LZW},
	inspired by \cite{KLR13,W21}, and due to the nonexistence results in Theorem \ref{Thm1}-\eqref{NE4}-\eqref{eq:NE4iii} and Theorem \ref{Thm1}-\eqref{NE4}-\eqref{eq:NE4iv}, in the framework of Theorem \ref{Thm1}-\eqref{E3}-\eqref{E3i}, we investigate asymptotic behaviour at small velocity. For this purpose, without loss of generality (up to a rotation), we always assume $v=(\beta,0,\ldots,0)\in\R^N$ with $0<\beta =|v|<1$. Meanwhile, for simplicity and convenience, we further rewrite  the constrained minimization problem \eqref{eq1.3} as
	\begin{equation*}
		e_{\beta}(a):=\inf_{\varphi\in \mathcal{S}_a} E_{\beta}(\varphi ),
	\end{equation*}
	where the energy functional $E_{\beta}(\varphi)$ is given by
	\begin{equation*}
		\begin{aligned}
			E_{\beta}(\varphi):=&\frac{1}{2}\int_{\mathbb{R}^N}\bar{{\varphi}}(\sqrt{-\Delta+m^2}+i\beta\partial_{x_1})\varphi dx-\frac{N}{2N+2}\|\varphi\|_{2+\frac2N}^{2+\frac2N}.
		\end{aligned}
	\end{equation*}
	Set
	$a^\ast_\beta:=\|Q_\beta\|_2^2,$
	where $Q_\beta:=Q_{(\beta,0,\ldots,0)}$ with $\beta>0$. We plan to study the behaviour of minimizers $u_\beta$, associated with a mass $a_\beta$ \emph{slightly} smaller that $a_\beta^*$, as $\beta \to 0^+$. Then we have that
	\begin{thm}\label{Thm4}

		Under the assumptions of  Theorem \ref{Thm1}-\eqref{E3}-\eqref{E3i}, let $u_{\beta}$ be a minimizer of $e_{\beta}(a_\beta)$ with $a_\beta=(1-\beta)^Na^\ast_\beta$. Then, if $\beta_n\to 0^+$, 
		\begin{equation}\label{qualunquecazzata}
			\|u_{\beta_n}\|_{\dot{H}^{\frac{1}{2}}(\R^N)}^2 \sim \|u_{\beta_n}\|_{2+\frac2N}^{2+\frac2N} \sim\beta_n^{-\frac12}, \text{ as } n\to \infty,
		\end{equation} and there exists a sequence $\{y_n\}\subset\mathbb{R}^N$ such that
		$$\beta_n^{\frac{N}{4}}u_{\beta_n}({\beta_n}^{\frac12}(\cdot+y_{n}))\rightarrow \eta^{\frac{N}{2}}{Q}(\eta \cdot) \quad\text{~in} ~ H^{\frac{1}{2}}(\mathbb{R}^N),$$
		where ${Q}$ is the unique ground state solution, up to translation, of equation \eqref{eqv0} and
		$$\eta=\left[{\frac{m^2}{2Na^\ast}\int_{\mathbb{R}^N}\bar{{Q}}\left(\frac{1}{\sqrt{-\Delta}}\right){Q} dx}\right]^{\frac12}.$$

		Moreover,
		we also have that
		\begin{equation*}
			\lim_{n\to \infty}\beta_n^{-\frac12}e_{\beta_n}(a_{\beta_n})=\left[\frac{Na^\ast m^2}{2}\int_{\mathbb{R}^N}\bar{{Q}}\left(\frac{1}{\sqrt{-\Delta}}\right){Q} dx\right]^{\frac12}.
		\end{equation*}
		
	\end{thm}
	
	\begin{rem}
		In particular, for the case $\mu,m=0$, the results obtained as a by-product of Theorem \ref{Thm4} (namely, Lemmas \ref{lem2.5} and \ref{convforte}) generalize those in Section 2 of \cite{KLR13} from $N=1$ to $N\ge 2$.
	\end{rem}

	This paper is organized as follows. In Section \ref{se2}, we provide some preliminary results and prove a few technical statements, which will be used frequently in the sequel. The existence and non-existence of minimizers to problem \eqref{eq1.4} are established in Section \ref{se3}. In sections \ref{se4}--\ref{se7}, we focus on characterizing various limiting profiles of the minimizers obtained in Section \ref{se3}.
	
	Finally, we list some notations.
	\begin{itemize}
		\item The Fourier transform of functions $u$ will be written as $\mathcal{F}(u)$ or $\hat{u}$.
		\item The letters $C,C_0,C_1,C_2\ldots$ will stand for a generic positive constant that may vary from line to line.
		\item We write $A\lesssim B$  (resp., $A\gtrsim B$) to indicate that there is a constant $C>0$ such that $A\le CB$ (resp., $A\ge CB$), also denote $A\sim B$ when $A\lesssim B\lesssim A$.
		\item The abbreviation of the norm $\|u\|_p:=\|u\|_{L^p(\R^N)}$  is used for brevity.
		\item We also denote $v\cdot \nabla :=\sum_{k=1}^{N} v_k\partial _{x_k}$, where $v\in \mathbb{R}^N$ is some fixed vector.
	\end{itemize}

	\section{Preliminaries}\label{se2}
	For the readers' convenience, in this section, we begin by briefly recalling some basic definitions and important properties. Meanwhile, we refer the interested reader to \cite{BGLV19,DPV12,FJL07} for more details.
	
	The fractional Sobolev space $ H^{\frac{1}{2}}(\mathbb{R}^N)$ is defined by
	$$H^{\frac{1}{2}}(\mathbb{R}^N):=\left\{u\in L^2(\mathbb{R}^N): (-\Delta)^\frac{1}{4}u\in L^2(\mathbb{R}^N)\right\},$$
	and  endowed  with the norm
	$$\|u\|_{H^{\frac{1}{2}}(\mathbb{R}^N)}:=\left(\|u\|_2^2+\int_{\mathbb{R}^N}|k||\hat{u}(k)|^2dk\right)^\frac{1}{2}.$$
	In addition, for any velocity $v\in\R^N$ with  $|v|<1$, we define the quadratic form as follows:
	$$T_v:\dot{H}^{\frac{1}{2} }(\R^N)\to \R,~~ u\mapsto T_v(u):=\inte \bar{u}(\sqrt{-\Delta  }+iv\cdot \nabla )u dx,$$
	where ${\dot{H}^{\frac{1}{2} }(\R^N)}$ denotes the usual Sobolev space in $\R^N$  equipped with the norm
	$$\|u\|_{\dot{H}^{\frac{1}{2} }(\R^N)}:=\| (-\Delta)^\frac{1}{4}u\|_{2}.$$
	Observe that 
	$$T_v(u)=\inte \bar{u}(\sqrt{-\Delta  }+iv\cdot \nabla )u dx=\inte (|k|-v\cdot k)|\hat{u}|^2dk.$$
	Clearly, we have the bounds
	\begin{equation}\label{eq10280938}
		(1-|v|)\|u\|_{\dot{H}^{\frac{1}{2} }(\R^N)}^2\leq T_v(u)\leq (1+|v|)\|u\|_{\dot{H}^{\frac{1}{2}}(\R^N)}^2,
	\end{equation}
	and so $\sqrt{T_v(u)}$ is a Hilbertian norm equivalent to the standard norm $\|u\|_{\dot{H}^{\frac{1}{2}}(\R^N)}$ since $|v|<1$.
	On the other hand, we introduce 
	\begin{equation*}
		\mathcal{T}_{m,v}(u):=\inte \bar{u}(\sqrt{-\Delta+m^2}+iv\cdot \nabla )udx=\inte (\sqrt{|k|^2+m^2}-v\cdot k)|\hat{u}|^2dk.
	\end{equation*}
	We have that $\sqrt{\mathcal{T}_{m,v}(u)}$ is a Hilbertian norm equivalent to $\|u\|_{{H}^{\frac{1}{2}}(\R^N)}$ (see \cite[Lemma A.4]{FJL07}) and, moreover, since
	\[
	\sqrt{|k|^2+m^2}-v\cdot k\ge \sqrt{1-|v|^2}m,\qquad\text{for all }k\in \R^N,
	\]
	we get
	\begin{equation}\label{APP.C}
		\mathcal{T}_{m,v}(u)\ge  \sqrt{1-|v|^2}m\|u\|_2^2, \qquad\text{for all }u\in H^{\frac12}(\R^N).
	\end{equation}
	Observe that if $w=\lambda_1 u(\lambda_2 \cdot)$ with $\lambda_1,\lambda_2>0$, then, 
	\begin{equation}\label{riscalamento}
		\mathcal{T}_{m,v}(w)=\frac{\lambda_1^2}{\lambda_2^{N-1}}\mathcal{T}_{m/\lambda_2,v}(u), \qquad
		T_v(w)=\frac{\lambda_1^2}{\lambda_2^{N-1}}T_v(u).
	\end{equation}

	Now we recall a few classical results. 
	
	\begin{lem}\label{lem1.3}(\cite[Lemma 2.2]{FQT12})
		Let $R>0$ and $\{\psi_k\}$ be bounded in $H^\frac{1}{2}(\R^N).$ If
		$$\lim_{k \to \infty}\sup _{y\in \mathbb{R}^N}\int_{B_R(y)}|\psi _{{k} } |^2dx =0, $$
		then
		$\psi_k\to 0$ strongly in $L^\kappa(\R^N)$ for $2< \kappa<\frac{2N}{N-1}.$
	\end{lem}

	\begin{lem}\label{lem1.2}(\cite[Proposition 3.1]{BGLV19})
		Let $0<p<\frac{2}{N-1}$ with $N\ge2$ and $v\in \mathbb{R}^N$ with $|v|<1$. Then 
		there exists a sharp constant $C_{v,N,p}>0$ such that,
		for all $u\in H^{\frac{1}{2} }(\mathbb{R}^N)$,
		\begin{equation}\label{eq2.19}
			\begin{aligned}
				\|u\|_{p+2}^{p+2}
				&\le 
				C_{v,N,p}\left(T_v(u)\right)^{\frac{Np}{2}}\|u\|_2^{{p+2}-{Np}},
			\end{aligned}
		\end{equation}
		Moreover, there exists an optimizer $U_v\in H^{\frac{1}{2} }(\mathbb{R}^N)\setminus\{0\}$ for \eqref{eq2.19} and it satisfies  
		\begin{equation}\label{eq2.20}
			(\sqrt{-\Delta  }+iv\cdot \nabla ) u  +u  =|u|^{p}u.
		\end{equation}
	\end{lem}
	
	\begin{rem}\label{re1}
		In fact, repeating similar arguments as the proof of \cite[Lemma 2.2]{HYZ24} or \cite[Theorem 3.2]{Z17}, we can obtain
		\begin{equation}\label{eq2.7}
			C_{v,N,p}=\frac{p+2}{p+2-Np}\left[\left(\frac{p+2-Np}{Np}\right)^N\frac{1}{\|U_v\|_2^2}\right]^{\frac{p}{2}}.
		\end{equation}
	\end{rem}
	
	By \cite[Lemma A.4]{BGLV19}, we have
	\begin{lem}\label{lem2.7}
		Let $0<p<\frac{2}{N-1}$ with $N\ge2$, and $v\in \mathbb{R}^N$ with $|v|<1$. Suppose that $U_v\in H^{\frac{1}{2} }(\mathbb{R}^N)$ solves equation \eqref{eq2.20}. 
		Then $U_v\in H^1(\mathbb{R}^N)\cap C_0(\mathbb{R}^N)$ and we have the decay estimate
		$$|U_v(x)|+|\nabla U_v(x)|\le \frac{C}{|x|^{N+1}}$$
		with some constant $C>0$. In particular,  we also have that $U_v\in L^\tau(\R^N)$, for all $\tau\in [1,+\infty]$, and $x\cdot\nabla U_v\in L^2(\mathbb{R}^N)$.
		
	\end{lem}
	
	The next result is a classical consequence of the Pohozaev identity. 
	\begin{prop}\label{prop1.1}
		Let $|v|<1$ and $Q_v\in H^{\frac{1}{2}}(\mathbb{R}^N)$ be any optimizer for \eqref{eq2.19} for  $p=2/N$. Then it is a ground state solution of \eqref{eq1.6} (eventually with $v=0$).
		Moreover, if $v\neq 0$, any ground state has the same squared $L^2$-norm  $a^*_v$.
	\end{prop}
	
	\begin{proof}
		We define the associated energy functional of equation \eqref{eq1.6} by
		\begin{equation}\label{eq1.8}
			I(u):=\frac12T_v(u)+\frac{1}{2}\|u\|_2^2-\frac{N}{2N+2} \|u\|_{2+\frac2N}^{2+\frac2N},
		\end{equation}
		and  the non-empty set $\mathcal{G}$ by
		$$
		\mathcal{G}:=\left\{u\in H^{\frac{1}{2}}(\mathbb{R}^N) : u \text{ is a solution of  equation } \eqref{eq1.6}\right\}.
		$$
		Then, by \cite[Lemma A.3]{BGLV19}, we see that all solutions $u\in \mathcal{G}$ satisfy the following Pohozaev-type identity
		\begin{equation}\label{eq3.49}
			T_v(u)=\frac{N}{N+1} \|u\|_{2+\frac2N}^{2+\frac2N},~u\in \mathcal{G},
		\end{equation}
		which, combined with \eqref{eq1.6}, indicates that
		\begin{equation}\label{eq1.9}
			T_v(u)=N\|u\|_2^2=\frac{N}{N+1} \|u\|_{2+\frac2N}^{2+\frac2N},~u\in \mathcal{G}.
		\end{equation}
		Thus, it follows from \eqref{eq1.8} and \eqref{eq1.9} that
		\begin{equation}\label{eq1.10}
			I(u)=\frac{1}{2}\|u\|_2^2,~ \forall u\in \mathcal{G},
		\end{equation}
		which means that all ground state solutions of \eqref{eq1.6} have the same $L^2$-norm, even though ground state solution of \eqref{eq1.6} is not unique.
		
		On the other hand, using 
		\eqref{eq2.19}, \eqref{eq3.49}, and \eqref{eq2.7}, we note that
		\begin{equation}\label{eq3.16}
			\|Q_v\|_2^2\le\|u\|_2^2,~ \forall u\in \mathcal{G}.
		\end{equation}
		Thus, by \eqref{eq1.10},
		$Q_v\in \mathcal{G}$ is a ground state solution. So, we are done.
	\end{proof}

	Recalling the definitions in \eqref{astar} and \eqref{astarv}, we further have
	\begin{lem}(\cite[Lemma 2.3]{LZW})
		Assume that $N\ge2$, $v\in \mathbb{R}^N$ with $|v|<1$. Then, for any $u\in H^{\frac{1}{2} }(\mathbb{R}^N)$, we have
		\begin{equation}\label{eq1.5}
			\begin{aligned}
				\|u\|_{2+\frac{2}
					{N}}^{2+\frac{2}{N}}
				&\le 
				\frac{N+1}{N{(a^*_v)}^{\frac1N}}T_v(u)\|u\|_2^{\frac{2}{N}},
			\end{aligned}
		\end{equation}
		where $a^*_v$ is defined in \eqref{astarv}.
		
		In addition, 
		\begin{equation}\label{eq1.16}
			(1-|v|)^Na^\ast\le a^\ast_v \le a^\ast,
		\end{equation}
		so that 
		\begin{equation}\label{eq3.114}
			a_v^*\to a^*, \text{ as } |v|\to0.
		\end{equation}
	\end{lem}

	Now, we conclude this section by stating several technical results that will be essential in later analysis.

	\begin{lem}\label{lemA.7}
		Let
		$$
		g_1(t):=at-bt^{\frac{Nq}{2}},~~ 
		g_2(t):=at^{-1}+bt^{\frac{Ns}{2}},~~ 
		g_3(t):=at^{2-\frac{Nq}{2}}-bt^{1-\frac{Nq}{2}},~~
		g_4(t):=at+bt^{-\frac{Nq}{2}},
		$$
		where $a,b,t>0$ and $0<s\le \frac{2}{N}$. Then we have that
		\begin{equation}\label{eqming1}
			g_1(t_1^\ast)=\min_{t>0}g_1(t)=-\frac{2-Nq}{2}\left(\frac{Nq}{2}\right)^{\frac{Nq}{2-Nq}}a^{-\frac{Nq}{2-Nq}}b^{\frac{2}{2-Nq}}<0,
		\end{equation}
		\begin{equation}\label{eq5.2}
			g_2(t_2^\ast)=\min_{t>0}g_2(t)=\frac{2+Ns}{2}\left(\frac{2a}{Ns}\right)^{\frac{Ns}{2+Ns}}b^{\frac{2}{2+Ns}}>0,
		\end{equation}
		\begin{equation}\label{eqming3}
			g_3(t_3^*)=\min_{t>0}g_3(t)=-\frac{2}{4-Nq}a^{\frac{Nq}{2}-1}b^{2-\frac{Nq}{2}}\left(\frac{2-Nq}{4-Nq}\right)^{1-\frac{Nq}{2}}<0,
		\end{equation}
		\begin{equation}\label{eqming4}
			g_4(t_4^*)=\min_{t>0}g_4(t)
			=\left(1+\frac{2}{Nq}\right)a^{\frac{Nq}{2+Nq}}\left(\frac{Nq}{2}b\right)^{\frac{2}{2+Nq}}>0,    
		\end{equation}
		and the minima are achieved, respectively,  by
		$$t_1^\ast=\left(\frac{bNq}{2a}\right)^{\frac{2}{2-Nq}},~~ t_2^\ast=\left(\frac{2a}{bNs_1}\right)^{\frac{2}{2+Ns_1}},~~ t_3^*=\frac{b(2-{Nq})}{a(4-{Nq})},~~
		t_4^*=\left(\frac{Nqb}{2a}\right)^{\frac{2}{2+Nq}}.$$
	\end{lem}
	
	\begin{proof}
		The proof is standard, here we omit it.
	\end{proof}

	Let us recall the following well-known
	Gagliardo–Nirenberg inequality.
	Consider $1\leq r,s\leq +\infty$, $j$ and $m$ non-negative integers with $j<m$, $p\geq 1$, and $\theta \in [0,1]$ such that
	$$\frac{1}{p}=\frac{j}{N}+\theta \left(\frac{1}{r}-\frac{m}{N}\right)+\frac{1-\theta}{s},\qquad \frac{j}{m}\leq \theta \leq 1.$$
	Then,
	\begin{equation}\label{GN}
		\|D^{j}u\|_{p}\leq C\|D^{m}u\|_{r}^{\theta}\|u\|_{s}^{1-\theta},
	\end{equation}
	for any $u\in L^{s}(\mathbb{R}^{N})$ such that 
	$D^{m}u\in L^{r}(\mathbb{R}^{N})$.\\
	Hence, we have that 
	\begin{lem}\label{lem1.6}
		There holds 
		\begin{align*}
			\sup_{\{u\in H^{1}(\R^N): \|u\|_2^2=a\}}\frac{\|u\|_{2+\frac2N}^{\frac{(N+1)(4-Nq)}N}}{\|\nabla u\|_2^{2-Nq}\|u\|_{q+2}^{q+2}}<+\infty.
		\end{align*}
	\end{lem}
	
	\begin{proof}
		Let $p=2+\frac{2}{N}$, $j=0$, $m=1$, $r=2$ and $s=q+2$
		in \eqref{GN}
		Then it follows that
		\begin{equation*}
			\|u\|_{2+\frac{2}{N}}\leq C\|\nabla u\|_{2}^{\frac{N(Nq-2)}{(N+1)(Nq-2q-4)}}\|u\|_{q+2}^{1-\frac{N(Nq-2)}{(N+1)(Nq-2q-4)}}, \quad \forall u\in H^{1}(\R^N). 
		\end{equation*}
		Thus we deduce that
		\begin{equation}\label{eq5.1}
			\frac{\|u\|_{2+\frac2N}^{(2+\frac2N)(2-\frac{Nq}{2})}
			}{\|\nabla u\|_2^{2-Nq}\|u\|_{q+2}^{q+2}}\le C\frac{\|\nabla u\|_2^{\frac{(2-Nq)(4-Nq)}{2(q+2)-Nq}}\|u\|_{q+2}^{\frac{(N+2)(q+2)(4-Nq)}{N(2q+4-Nq)}}
			}{\|\nabla u\|_2^{2-Nq}\|u\|_{q+2}^{q+2}}=C\left(\frac{\|u\|_{q+2}^{\frac{2(q+2)}{N}}}{\|\nabla u\|_2^q}\right)^{\frac{2(2-Nq)}{2(q+2)-Nq}}.
		\end{equation}
		Setting now $p=q+2$, $j=0$, $m=1$, $r=2$ and $s=2$ in \eqref{GN},
		then one has
		\begin{equation*}
			\|u\|_{q+2}\leq C\|\nabla u\|_{2}^{\frac{Nq}{2(q+2)}}, \quad \forall u\in H^{1}(\R^N), \|u\|_2^2=a,
		\end{equation*}
		which, together with \eqref{eq5.1}, implies the desired result.
	\end{proof}

	\section{Existence and non-existence}\label{se3}
	In this section, we mainly investigate the existence of minimizers to problem \eqref{eq1.3} under some suitable assumptions with respect to the parameters $a,m,\mu$ proving Theorem \ref{Thm1}.
	
	\subsection{Proof of Theorem \ref{Thm1}-(\ref{E1})}
	\
	\\
	Here we want to apply the Concentration Compactness in order to prove our main result.
	The following lemma will be used to establish the strict subadditivity inequality and the boundedness of minimizing sequence  for $e(a)$, as well as to exclude {\em Vanishing}.
	\begin{lem}\label{lem2.1}
		Let $a^*_v>a>0$, $m\ge0$, $\mu>0$. Then we have that
		\begin{equation}\label{eq2.17}
			\frac{Nq-2}{2}\left({Nq}\right)^{\frac{Nq}{2-Nq}}{\left[1-\left(\frac{a}{a^\ast_v}\right)^{\frac1N}\right]}^{-\frac{Nq}{2-Nq}}{\left(\frac{\mu C_{v,N,q}}{q+2}a^{\frac{q+2-Nq}{2}}\right)}^{\frac{2}{2-Nq}}
			\le e(a)< \frac{\sqrt{1-|v|^2}}{2}ma.
		\end{equation}
	\end{lem}

	\begin{proof}
		To show the lower bound, by using the operator inequality
		\begin{equation}\label{eq2.5}
			\sqrt{-\Delta}\le \sqrt{-\Delta+m^2}\le \sqrt{-\Delta}+m,
		\end{equation}
		for any $u\in \mathcal{S}_a$, we firstly deduce from \eqref{eq2.19} and \eqref{eq1.5} that
		\begin{equation}\label{eq2.18}
			\begin{aligned}
				E(u)
				&\ge
				\frac{1}{2}\left[1-\left(\frac{a}{a^\ast_v}\right)^{\frac1N}\right]T_v(u)-\frac{\mu C_{v,N,q}}{q+2} \left(T_v(u)\right)^{\frac{Nq}{2} }a^{\frac{q+2}{2}-\frac{Nq}{2}},
			\end{aligned}
		\end{equation}
		which, combined with \eqref{eqming1} in Lemma \ref{lemA.7}, a direct calculation gives
		\begin{equation}\label{eq4.46}
			\begin{aligned}
				e(a)\ge \frac{Nq-2}{2}\left({Nq}\right)^{\frac{Nq}{2-Nq}}{\left[1-\left(\frac{a}{a^\ast_v}\right)^{\frac1N}\right]}^{-\frac{Nq}{2-Nq}}{\left(\frac{\mu C_{v,N,q}}{q+2}a^{\frac{q+2}{2}-\frac{Nq}{2}}\right)}^{\frac{2}{2-Nq}}.
			\end{aligned}
		\end{equation}

		Next, we prove the upper bound, which will be divided into two cases.
		
		{\bf Case 1: $m>0$.}
		Without loss of generality, here and in what follows, we always assume that  $v$ is parallel to the $x_N$-axis, i.e., $v=|v|e_{x_N}$.

		Let  $\phi\in \mathcal{S}_a\cap C^{\infty}_c(\R^N, \mathbb R)$ and introduce 
		\begin{equation*}
			\phi _{\lambda }(x):=e^{i\lambda v\cdot x}\phi (x)=e^{i\lambda |v|{x_N}}\phi (x),~~\text{ where } \lambda>0,
		\end{equation*}
		which, together with the fact that 
		$\int_{\mathbb{R}^N}\phi\frac{\partial\phi}{\partial {x_i}}dx=0$, 
		yields that
		\begin{equation*}
			\frac{i}{2}\int_{\mathbb{R}^N}\bar{\phi}_{\lambda }(v\cdot \nabla )\phi _{\lambda }dx=-\frac{\lambda |v|^2}{2} a~\text{ and }~\|\phi _{\lambda}\|_2^2=\|\phi\|_2^2=a.
		\end{equation*}
		Then we deduce that
		\begin{equation*}
			\begin{aligned}
				E(\phi_{\lambda})
				=&\frac{1}{2}\int_{\mathbb{R}^N}\bar{\phi}_{\lambda }(\sqrt{-\Delta+m^2})\phi _{\lambda }dx-\frac{1}{2}\lambda |v|^2a-\frac{N}{2N+2}\|\phi\|_{2+\frac2N}^{2+\frac2N}-\frac{\mu}{q+2} \|\phi\|_{q+2}^{q+2}.
			\end{aligned}
		\end{equation*}
		Moreover, according to the operator inequality
		$$\sqrt{-\Delta+m^2}\le \frac{1}{2\lambda } (-\Delta+m^2+\lambda^2),$$
		one can see  that
		\begin{equation*}
			\begin{aligned}
				&\frac{1}{2}\int_{\mathbb{R}^N}{\bar\phi}_{\lambda}(\sqrt{-\Delta+m^2})\phi _{\lambda }dx-\frac{1}{2}\lambda |v|^2a
				\leq\frac{1}{4\lambda }\int_{\mathbb{R}^N}\bar{\phi}_{\lambda }({-\Delta+m^2} +{\lambda }^2)\phi_{\lambda }dx-\frac{1}{2}\lambda |v|^2a\\
				&=\frac{1}{4\lambda}\left[\int_{\mathbb{R}^N}{\phi} (-\Delta)\phi dx +{\lambda }^2|v|^2a +\left(m^2 +{\lambda}^2\right)a\right]-\frac{1}{2}\lambda|v|^2a=:\frac{1}{4\lambda }\int_{\mathbb{R}^N}{\phi} (-\Delta)\phi dx+f(\lambda).
			\end{aligned}
		\end{equation*}
		The minimiser of $f$
		for $\lambda >0$ is $\lambda _*=\frac{m}{\sqrt{ 1-|v|^2}}$ and so
		\begin{equation}\label{eq2.12}
			\begin{aligned}
				E(\phi_{\lambda_*})
				&\leq  f(\lambda_*)+\frac{1}{4\lambda_* }\int_{\mathbb{R}^N}{\phi} (-\Delta)\phi dx-\frac{N}{2N+2}\|\phi\|_{2+\frac2N}^{2+\frac2N}-\frac{\mu}{q+2} \|\phi\|_{q+2}^{q+2}\\
				&=\frac{\sqrt{1-|v|^2}}{2}ma +\frac{\sqrt{ 1-|v|^2}}{4m} \|\nabla \phi\|_2^2 -\frac{N}{2N+2}\|\phi\|_{2+\frac2N}^{2+\frac2N}-\frac{\mu}{q+2} \|\phi\|_{q+2}^{q+2}\\
				&=:\frac{\sqrt{ 1-|v|^2}}{2}ma +\widetilde{E}_{q}(\phi).
			\end{aligned}
		\end{equation}
		Meanwhile, considering the test function  $\phi_{t}:=t^{\frac{N}{2}}\phi(t\cdot)$, for $t>0$ sufficiently small,
		there holds
		\begin{equation*}
			\begin{aligned}
				\widetilde{E}_{q}(\phi_{t})
				&=\frac{t^2\sqrt{1-|v|^2}}{4m}\|\nabla \phi\|_2^2
				-\frac{t N}{2N+2}\|\phi\|_{2+\frac2N}^{2+\frac2N}-\frac{\mu t^{\frac{Nq}{2}}}{q+2} \|\phi\|_{q+2}^{q+2}<0.
			\end{aligned}
		\end{equation*}
		Then it follows from \eqref{eq2.12} that
		\begin{equation*}
			e_{v}(a)< \frac{\sqrt{1-|v|^2}}{2}ma.
		\end{equation*}

		{\bf Case 2: $m=0$.} In this case, by taking $u_{t}:=t^{\frac{N}{2}}u(t \cdot)$ with $u\in \mathcal{S}_a$, we deduce that
		\begin{equation*}
			\begin{aligned}
				E(u_{t})
				=\frac{t}{2}T_v(u)-\frac{Nt}{2N+2}\|u\|_{2+\frac2N}^{2+\frac2N}-\frac{\mu t^{\frac{Nq}{2}}}{q+2} \|u\|_{q+2}^{q+2},
			\end{aligned}
		\end{equation*}
		and so, taking into account its behaviour for $t>0$ small enough,
		\begin{equation*}
			e(a)<0= \frac{\sqrt{1-|v|^2}}{2}ma.
		\end{equation*}
	\end{proof}

	Now, we establish a strict subadditivity inequality for $e_{v}(a)$, which plays a key role in ruling out {\em Dichotomy}. To do this, for any $a>0$, we define the  auxiliary constrained minimization problem
	\begin{equation}\label{eq2.21}
		{e}^{a}(1):=\inf_{\psi \in \mathcal{S}_1}{E}^{a} (\psi),
	\end{equation}
	where
	\begin{equation*}
		\begin{split}
			{E}^{a}(\psi):=&\frac{1}{2}\mathcal{T}_{m,v}(\psi)-\frac{Na^{\frac{1}{N}}}{2N+2} \|\psi\|_{2+\frac{2}{N}}^{2+\frac{2}{N}}-\frac{\mu a^{\frac{q}{2}}}{q+2} \|\psi\|_{q+2}^{q+2},
		\end{split}
	\end{equation*}
	and
	$$\mathcal{S}_1:=\left\{\psi \in H^{\frac{1}{2}}(\mathbb{R}^N):\|\psi\|_2^2=1\right\}.$$
	Arguing as in \eqref{eq2.18}, for any $\psi \in \mathcal{S}_1$, one can see that the problem \eqref{eq2.21} is well-defined. Meanwhile, we notice the following scaling behaviour
	\begin{equation}\label{eq2.23}
		e(a)=ae^{a}(1).
	\end{equation}
	Then, there holds
	\begin{lem}\label{lem2.2}
		Under the assumptions of Lemma \ref{lem2.1}, we have that $e(a)$ satisfies the following strict subadditivity inequality
		\begin{equation}\label{eq2.22}
			e(a)<e(\lambda )+e(a-\lambda), \text{ where } 0<\lambda<a.
		\end{equation}
		In addition,  $e(t)-\sqrt{1-|v|^2}mt/2$ is strictly decreasing
		and  $e(t)$ is continuous with respect to $t\in(0,a]$.
	\end{lem}
	\begin{proof}
		Let $t\in (0,a]$ and $\{\psi_k^t\} \subset\mathcal{S}_1$ be a minimizing sequence of $e^{t}(1)$. We claim that there exist some constants $C_1(t),C_2(t)>0$ such that
		\begin{equation}\label{eq2.24}
			\liminf\limits_{k\rightarrow\infty}\|\psi_k^t\|_{2+\frac{2}{N}}^{2+\frac{2}{N}}\ge C_1(t) ~~\text{ or }~~
			\liminf\limits_{k\rightarrow\infty}\|\psi_k^t\|_{q+2}^{q+2}\ge C_2(t).
		\end{equation}
		On the contrary, by passing to a subsequence, we have that
		$$
		\lim\limits_{k\rightarrow\infty}\|\psi_k^t\|_{2+\frac{2}{N}}^{2+\frac{2}{N}}=\lim\limits_{k\rightarrow\infty}\|\psi_k^t\|_{q+2}^{q+2}=0,
		$$
		which, combined with \eqref{APP.C},
		indicates that
		\begin{equation*}
			e(t)=te^{t}(1)=t\lim_{k \to \infty}{E}^{t}(\psi_k^t)\ge \frac{\sqrt{1-|v|^2}}{2}mt.
		\end{equation*}
		But, from Lemma \ref{lem2.1} we have that
		\begin{equation*}
			e(t)< \frac{\sqrt{1-|v|^2}}{2}mt,
		\end{equation*}
		which leads to a contradiction.  Thus (\ref{eq2.24}) holds. Choosing $0< t_1<t_2\leq a$ and taking $\{\psi_k^{t_1}\} \subset\mathcal{S}_1$ a minimizing sequence of $e^{t_1}(1)$, we deduce from  (\ref{eq2.24}) that
		\begin{equation*}
			\begin{aligned}
				e^{t_1}(1)&=\lim_{k \to \infty} {E}^{t_1}(\psi_k^{t_1})\\
				&= \lim_{k \to \infty}{E}^{t_2}(\psi_k^{t_1})+\lim_{k \to \infty}\frac{N}{2N+2}\left(t_2^{\frac{1}{N}}-t_1^{\frac{1}{N}}\right) \|\psi_k^{t_1}\|_{2+\frac{2}{N}}^{2+\frac{2}{N}}+\lim_{k \to \infty}\frac{\mu}{q+2}\left(t_2^{\frac{q}{2}}-t_1^{\frac{q}{2}}\right)\|\psi_k^{t_1}\|_{q+2}^{q+2}\\
				&> \lim_{k \to \infty}{E}^{t_2}(\psi_k^{t_1})\geq e^{t_2}(1),
			\end{aligned}
		\end{equation*}
		which implies $e^{t}(1)$ is strictly decreasing with respect to $t\in(0, a]$.
		In addition, by \eqref{eq2.17} and (\ref{eq2.23}), for each $t\in(0, a]$, this yields that
		\begin{equation}\label{eq2.30}
			e^{t}(1)-\frac{\sqrt{1-|v|^2}}{2}m=\frac{1}{t}\left(e(t)-\frac{\sqrt{1-|v|^2}}{2}mt\right)< 0.
		\end{equation}
		Therefore, setting $ 0<t_3<t_4\leq a$, it follows from (\ref{eq2.23}) and \eqref{eq2.30} that
		\begin{align*}
			e(t_3)-\frac{\sqrt{1-|v|^2}}{2}mt_3
			&=t_3\left(e^{t_3}(1)-\frac{\sqrt{1-|v|^2}}{2}m\right)
			> t_4\left(e^{t_3}(1)-\frac{\sqrt{1-|v|^2}}{2}m\right)\\
			&>
			t_4e^{t_4}(1)-\frac{\sqrt{1-|v|^2}}{2}mt_4=e(t_4)-\frac{\sqrt{1-|v|^2}}{2}mt_4,
		\end{align*}
		so $e(t)-\sqrt{1-|v|^2}mt/2$ is strictly decreasing with respect to $t\in (0, a]$.
		
		Next, we show that $e(a)$ satisfies the strict subadditivity inequality. From (\ref{eq2.23}) and the fact that $e^{t}(1)$ is strictly decreasing with respect to $t\in(0, a]$, we infer  that
		$$e(\theta M)=\theta Me^{\theta M}(1) <\theta Me^{ M}(1)=\theta e(M),$$
		where $a> M>0,~ \frac aM\geq\theta>1, ~$(i.e., $0< M<\theta M\leq a$). By \cite[Lemma II.1]{Lion}, this inequality leads to the strict subadditivity inequality  \eqref{eq2.22}.
		
		Finally, we show that $e(t)$ is continuous with respect to $t\in (0,a]$. Indeed, by using \eqref{eq2.23}, here  it is sufficient to prove that $e^{t}(1)$ is continuous with respect to $t\in (0,a]$.
		For any $t_5,t_6\in (0,a]$ with $t_5<t_6$, by the definition of $e^{t}(1)$, there exists  $\psi\in\mathcal{S}_1$ such that
		\begin{equation}\label{eq-inf}
			e^{t_6}(1)\le {E}^{t_6}(\psi)\le e^{t_6}(1)+t_6-t_5\le e^{t_6}(1)+a
			< \frac{\sqrt{1-|v|^2}}{2}m+a,
		\end{equation}
		where we have used \eqref{eq2.30}.
		Recalling \eqref{eq2.18}, one has 
		\begin{equation*}
			\begin{aligned}
				E^{t_6}(\psi)
				&\ge\frac{1}{2}\left[1-\left(\frac{t_6}{a^\ast_v}\right)^{\frac1N}\right]T_v(\psi)-\frac{\mu C_{v,N,q} t_6^{\frac{q}{2}}}{q+2} \left(T_v(\psi)\right)^{\frac{Nq}{2}}
				\\
				&\ge\frac{1}{2}\left[1-\left(\frac{a}{a^\ast_v}\right)^{\frac1N}\right]T_v(\psi)-\frac{\mu C_{v,N,q} a^{\frac{q}{2}}}{q+2} \left(T_v(\psi)\right)^{\frac{Nq}{2}},
			\end{aligned}
		\end{equation*}
		which, together with \eqref{eq-inf}, implies that  $\psi$ is uniformly bounded in $H^{\frac{1}{2}}(\R^N)$, with respect to $t_5$ and $t_6$. 
		\\
		According to the fact that $e^{t}(1)$ is strictly decreasing with respect to $t\in(0, a]$, by \eqref{eq-inf}  again, it then follows that
		\begin{equation*}
			\begin{aligned}
				0<e^{t_5}(1)-e^{t_6}(1)
				&\le {E}^{t_5}(\psi)-{E}^{t_6}(\psi)+t_6-t_5\\
				&=\frac{N}{2N+2}\left(t_6^{\frac{1}{N}}-t_5^{\frac{1}{N}}\right) \|\psi\|_{2+\frac{2}{N}}^{2+\frac{2}{N}}+\frac{\mu}{q+2} \left(t_6^{\frac{q}{2}}-t_5^{\frac{q}{2}}\right) \|\psi\|_{q+2}^{q+2}+t_6-t_5\\
				&\le  C_1\left(t_6^{\frac{1}{N}}-t_5^{\frac{1}{N}}\right) +C_2\left(t_6^{\frac{q}{2}}-t_5^{\frac{q}{2}}\right)+t_6-t_5,
			\end{aligned}
		\end{equation*}
		and
		the proof is completed.
	\end{proof}

	Taking any minimizing sequence $\{\psi_k\} \subset\mathcal{S}_a$ so that $\lim\limits_{k \to \infty} E(\psi_k)=e(a)$,
	by \eqref{eq10280938} and \eqref{eq2.18}, we deduce that $\{\psi_k\} $ is bounded in $H^{\frac{1}{2}}(\mathbb{R}^N)$.
	
	Next, we apply the concentration-compactness principle to prove the existence of minimizers to problem \eqref{eq1.3}. Since the proof is standard, here we omit the details for simplicity. Repeating similar arguments as Step 4 of \cite[Lemma 3.1]{LZW} (see \cite[Theorem 1.1]{HYZ24}), by using Lemma \ref{lem1.3}, Lemma \ref{lem2.1} and Lemma \ref{lem2.2}, the desired conclusion follows.

	\subsection{Proof of Theorem \ref{Thm1}-(\ref{E2})}
	\
	\\
	We begin by establishing the following precise energy estimates on $e(a)$.

	\begin{lem}\label{prop1.2}
		Let  $a^*_v\ge a>0$, $m>0$, and $\mu<0$,  then we have that
		$$\frac{ma}{2}\sqrt{\left[1-\left(\frac{a}{a^\ast_v}\right)^{\frac1N}\right](1-|v|)} \sqrt{1+|v|+\left(\frac{a}{a^\ast_v}\right)^{\frac1N}(1-|v|)}\le e(a)\le \min_{\tau>0}h(\tau),$$ 
		where $h$ is defined in \eqref{h}.
	\end{lem}

	\begin{proof}
		To obtain the upper bound, we define
		\begin{equation}\label{eq2.3}
			Q_v^\tau:=\tau^{\frac{N}{2}}Q_v(\tau \cdot ) \quad \text{ and } \quad
			\varphi_\tau:=\left(\frac{a}{a^*_v}\right)^{\frac12}
			Q_v^\tau,
		\end{equation}
		where $Q_v\in H^{\frac{1}{2} }(\mathbb{R}^N)$ is a solution of equation \eqref{eq1.6} and $\tau>0$.
		
		First observe that, by Lemma \ref{lem2.7}, $Q_v\in L^\tau(\R^N)$, for all $\tau \ge 1$, we have that $\hat Q_v\in L^2(\mathbb R^N) \cap L^\infty(\mathbb R^N)$ and so
		\begin{equation*}
			\int_{\mathbb R^N} \frac{|\hat Q_v(k)|^2}{|k|} dk 
			=\int_{B_R} \frac{|\hat Q_v(k)|^2}{|k|} dk 
			+\int_{B_R^c} \frac{|\hat Q_v(k)|^2}{|k|} dk
			\leq
			C\int_{B_R} \frac{1}{|k|} dk 
			+ \frac{1}{R} \int_{B_R^c} |\hat Q_v(k)|^2 dk <+\infty.
		\end{equation*}
		Clearly, by Plancherel's theorem, we have that
		\begin{equation}\label{eq3.8}
			\begin{aligned}
				\int_{\mathbb{R}^N}\bar{Q}_v^\tau(\sqrt{-\Delta+m^2}-\sqrt{-\Delta})Q_v^\tau dx&=\int_{\mathbb{R}^N}|\hat{Q}_v(k)|^2(\sqrt{\tau^2|k|^2+m^2}-\tau|k|)dk\\
				&=\int_{\mathbb{R}^N}|\hat{Q}_v(k)|^2\frac{m^2}{\sqrt{\tau^2|k|^2+m^2}+\tau|k|}dk\\
				&\le\frac{m^2}{2\tau}\int_{\mathbb{R}^N}\frac{|\hat{Q}_v(k)|^2}{|k|}dk.
			\end{aligned}
		\end{equation}
		Moreover, using \eqref{eq1.9}, there holds
		\begin{equation}\label{eq3.9}
			T_v(Q_v)=\frac{N}{N+1} \|Q_v\|_{2+\frac2N}^{2+\frac2N}=N\|Q_v\|_2^2=Na^\ast_v,
		\end{equation}
		thus we deduce by \eqref{riscalamento} that
		\begin{equation}\label{eq3.7}
			\begin{aligned}
				T_v(Q_v^\tau)=\tau T_v(Q_v)=\tau\frac{N}{N+1} \|Q_v\|_{2+\frac2N}^{2+\frac2N}=\frac{N}{N+1} \|Q_v^\tau\|_{2+\frac2N}^{2+\frac2N}.
			\end{aligned}
		\end{equation}
		Noting that $\varphi_\tau\in \mathcal{S}_a$, then from \eqref{eq1.5}, \eqref{eq3.8}, \eqref{eq3.9} and  \eqref{eq3.7} it follows that
		\begin{equation}\label{eq3.10}
			\begin{aligned}
				e(a)&\le E(\varphi_\tau)
				=\frac{1}{2}\frac{a}{a^\ast_v}\mathcal{T}_{m,v}(Q_v^\tau)-\frac{1}{2}\frac{a}{a^\ast_v}T_v(Q_v^\tau)-\frac{N}{2N+2}\left(\frac{a}{a^\ast_v}\right)^{1+\frac{1}{N}}\|Q_v^\tau\|_{2+\frac2N}^{2+\frac2N}\\
				&~~~~+\frac{1}{2}\frac{a}{a^\ast_v}\frac{N}{N+1} \|Q_v^\tau\|_{2+\frac2N}^{2+\frac2N}-\frac{\mu}{q+2}\left(\frac{a}{a^\ast_v}\right)^{\frac{q+2}{2}}\|Q_v^\tau\|_{q+2}^{q+2}\\
				&=\frac{1}{2}\frac{a}{a^\ast_v}\int_{\mathbb{R}^N}\bar{Q}_v^\tau(\sqrt{-\Delta+m^2}-\sqrt{-\Delta})Q_v^\tau dx+\frac{1}{2}\frac{a}{a^\ast_v}\frac{N}{N+1}\left[1-\left(\frac{a}{a^\ast_v}\right)^{\frac{1}{N}} \right] \|Q_v^\tau\|_{2+\frac2N}^{2+\frac2N}\\
				&~~~~-\frac{\mu}{q+2}  \left(\frac{a}{a^\ast_v}\right)^{\frac{q+2}{2}}\|Q_v^\tau\|_{q+2}^{q+2}\\
				&\le \frac{am^2}{4a^\ast_v\tau}\int_{\mathbb{R}^N}\frac{|\hat{Q}_v(k)|^2}{|k|}dk+\frac{aN\tau}{2}\left[1-\left(\frac{a}{a^\ast_v}\right)^{\frac{1}{N}} \right]-\frac{\mu\tau^{\frac{Nq}{2}}}{q+2} \left(\frac{a}{a^\ast_v}\right)^{\frac{q+2}{2}}\|Q_v\|_{q+2}^{q+2}:=h(\tau).
			\end{aligned}
		\end{equation}
		A simple analysis shows that there exists a unique global minimiser $\tau^*$ of $h$ for $\tau>0$ and so
		\begin{equation*}
			e(a)\le h(\tau^\ast)=\min_{\tau>0}h(\tau).
		\end{equation*}
		
		In the following, we give the lower bound, which is divided into two cases.
		
		{\bf Case 1: $0<a<a^*_v$.} In this case, for any $u\in \mathcal{S}_a$, by \eqref{eq1.5}, it yields that
		\begin{equation}\label{eq3.12}
			\begin{aligned}
				E(u)
				&\ge\frac{1}{2}\mathcal{T}_{m,v}(u)-\frac{1}{2}\left(\frac{a}{a^\ast_v}\right)^{\frac1N}T_v(u)\\
				&=\frac{1}{2}\int_{\mathbb{R}^N}|\hat{u}(k)|^2\left[\sqrt{|k|^2+m^2}-v\cdot k-\left(\frac{a}{a^\ast_v}\right)^{\frac1N}(|k|-v\cdot k)\right]dk\\
				&\ge\frac{1}{2}\int_{\mathbb{R}^N}|\hat{u}(k)|^2\left\{\sqrt{|k|^2+m^2}-\left(\frac{a}{a^\ast_v}\right)^{\frac1N}|k|
				-\left[1-\left(\frac{a}{a^\ast_v}\right)^{\frac1N}\right]|v||k|\right\}dk\\
				&=:\frac{1}{2}\int_{\mathbb{R}^N}|\hat{u}(k)|^2g(|k|)dk.
			\end{aligned}
		\end{equation}
		Taking the minimum of $g(|k|)$ over $|k|$,
		it follows that
		\begin{equation*}
			\begin{aligned}
				g(|k|)&\ge  m\sqrt{\left[1-\left(\frac{a}{a^\ast_v}\right)^{\frac1N}\right](1-|v|)} \sqrt{1+|v|+\left(\frac{a}{a^\ast_v}\right)^{\frac1N}(1-|v|)},
			\end{aligned}
		\end{equation*}
		which, combined with \eqref{eq3.12}, gives that
		\begin{equation}\label{eq3.101}
			e(a)\ge \frac{1}{2}\int_{\mathbb{R}^N}|\hat{u}(k)|^2g(|k|)dk\ge \frac{ma}{2}\sqrt{\left[1-\left(\frac{a}{a^\ast_v}\right)^{\frac1N}\right](1-|v|)} \sqrt{1+|v|+\left(\frac{a}{a^\ast_v}\right)^{\frac1N}(1-|v|)}.
		\end{equation}
		
		{\bf Case 2: $a=a^*_v$.} In this case, 
		proceeding as in \eqref{eq3.12}, we have that
		\begin{equation*}
			\begin{aligned}
				E(u)
				\ge\frac{1}{2}\mathcal{T}_{m,v}(u)-\frac{1}{2}T_v(u)
				=\frac{1}{2}\int_{\mathbb{R}^N}\bar{u}(\sqrt{-\Delta+m^2}-\sqrt{-\Delta})u dx\ge0,
			\end{aligned}
		\end{equation*}
		which means that
		$$e(a)\ge 0= \frac{ma}{2}\sqrt{\left[1-\left(\frac{a}{a^\ast_v}\right)^{\frac1N}\right](1-|v|)} \sqrt{1+|v|+\left(\frac{a}{a^\ast_v}\right)^{\frac1N}(1-|v|)}.$$
		
		Hence, the proof is complete.
	\end{proof}

	In particular, 
	if $m>0$ and $\mu_m^*<\mu<0$, where $\mu_m^*$ is defined in \eqref{mu_1}, we further have the following upper bound.

	\begin{lem}\label{lem2.3}
		Assume that $a^*_v\ge a>0$, $m>0$, $\mu_m^*<\mu<0,$ then there holds
		\begin{equation}\label{eq2.40}
			e(a)< \frac{\sqrt{1-|v|^2}}{2}ma.
		\end{equation}
		
	\end{lem}
	
	\begin{proof}
		Repeating the similar steps in Case 1 of the proof of Lemma \ref{lem2.1}, as in \eqref{eq2.12}, we deduce that for any  function $\phi\in \mathcal{S}_a\cap C^{\infty}_c(\R^N, \mathbb R)$ 
		\begin{equation}\label{eq2.38}
			\begin{aligned}
				E(\phi_{\lambda^*})
				\leq &\frac{\sqrt{1-|v|^2}}{2}ma +\frac{\sqrt{ 1-|v|^2}}{4m} \|\nabla\phi\|_2^2-\frac{N}{2N+2}\|\phi\|_{2+\frac2N}^{2+\frac2N}
				-\frac{\mu}{q+2} \|\phi\|_{q+2}^{q+2}\\
				=:&\frac{\sqrt{1-|v|^2}}{2}ma +\tilde{E}_{q}(\phi).
			\end{aligned}
		\end{equation}
		In addition, defining $\phi_{t}:=t^{\frac{N}{2}}\phi(t \cdot)$ with $t>0$, there holds
		\begin{equation}\label{eq2.39}
			\begin{aligned}
				\tilde{E}_{q}(\phi_{t})
				&=\frac{\sqrt{1-|v|^2}}{4m}{t}^2 \|\nabla\phi\|_2^2 -\frac{Nt}{2N+2}\|\phi\|_{2+\frac2N}^{2+\frac2N}-\frac{\mu{t}^{\frac{Nq}{2}}}{q+2} \|\phi\|_{q+2}^{q+2}
				=:{t}^{\frac{Nq}{2}}\left(g_3(t)-\frac{\mu}{q+2}\|\phi\|_{q+2}^{q+2}\right).
			\end{aligned}
		\end{equation}
		Let $\varepsilon>0$ be such that $\mu_m^*<\mu_m^*+\varepsilon<\mu<0$. 
		For such $\varepsilon$ there exists $\phi_\varepsilon\in \mathcal{S}_a\cap C^{\infty}_c(\R^N, \mathbb R)$ such that
		\[
		\mu_m^*
		\leq
		-m^{\frac{2-Nq}{2}} C(q,N,v) \frac{\|\phi_\varepsilon\|_{2+\frac2N}^{\frac{(N+1)(4-Nq)}N}}{\|\nabla \phi_\varepsilon\|_2^{2-Nq}\|\phi_\varepsilon\|_{q+2}^{q+2}}
		\leq \mu_m^*+\varepsilon.
		\]
		Then, applying \eqref{eqming3} in Lemma \ref{lemA.7} to \eqref{eq2.39}, we get that
		\begin{align*}
			\tilde{E}_{q}((\phi_\varepsilon)_{t_3^*})
			&=
			\frac{\|\phi_\varepsilon\|_{q+2}^{q+2}}{q+2}(t_3^*)^{\frac{Nq}{2}}
			\left[-m^{\frac{2-Nq}{2}} C(q,N,v)
			\frac{\|\phi_\varepsilon\|_{2+\frac2N}^{\frac{(N+1)(4-Nq)}N}}{\|\nabla \phi_\varepsilon\|_2^{2-Nq}\|\phi_\varepsilon\|_{q+2}^{q+2}}
			-\mu\right]\\
			&\leq
			\frac{\|\phi_\varepsilon\|_{q+2}^{q+2}}{q+2}(t_3^*)^{\frac{Nq}{2}}
			(\mu_m^*+\varepsilon-\mu)
			<0.    
		\end{align*}

		Thus, combining  \eqref{eq2.38}, we derive \eqref{eq2.40}.
	\end{proof}

	Taking in account Lemma \ref{prop1.2} and Lemma \ref{lem2.3}, to conclude the proof of Theorem \ref{Thm1}-\eqref{E2}, we have just to prove the existence of minimizer.
	We start with the following result.
	\begin{lem}\label{lem2.4}
		Suppose that $a^*_v\ge a>0$, $m>0$ and $\mu_m^*<\mu<0$.
		Then we have the strict subadditivity inequality
		\begin{equation}\label{eq2.41}
			e(a)<e(\lambda )+e(a-\lambda ), \text{ where } 0<\lambda<a.
		\end{equation}
		
	\end{lem}
	
	\begin{proof}
		For each $\rho>0$, we let
		$$\tilde{e}(\rho):=e(\rho)-\frac{\sqrt{1-|v|^2}}{2}m\rho.$$
		Clearly, the inequality \eqref{eq2.41} is equivalent to
		\begin{equation}\label{eq3.61}
			\tilde{e}(a)<\tilde{e}(\lambda )+\tilde{e}(a-\lambda ), \text{ where } 0<\lambda<a.
		\end{equation}
		So our next goal is to prove \eqref{eq3.61}.
		
		Observe that, for a generic $\theta>1,\rho\in (0,a), u\in S_\rho$,
		since $0<q<2/N$, by \eqref{APP.C}, we have that
		\begin{equation}\label{mle}
			\frac{{\left(\theta^{\frac{q+2}{2}}-\theta\right)}\mathcal{T}_{m,v}(u)+\frac{N\left(\theta^{1+\frac1N}-\theta^{\frac{q+2}{2}}\right)}{N+1}
				\|u\|_{2+\frac2N}^{2+\frac2N}}{\sqrt{1-|v|^2}\rho\left(\theta^{\frac{q+2}{2}}-\theta\right)}
			\ge \frac{\mathcal{T}_{m,v}(u)}{\sqrt{1-|v|^2}\rho}
			\ge
			m. 
		\end{equation}

		Moreover, for any $\epsilon>0$, there exists some $u \in \mathcal{S}_\rho$ with $\rho>0$, such that
		$$E(u)\le e(\rho)+\epsilon.$$
		Hence, for all $\theta>1$ and $m>0$, keeping in mind \eqref{mle}, we obtain
		\begin{equation}\label{eq3.1}
			\begin{aligned}
				\tilde{e}(\theta \rho)
				&\le E(\theta^{\frac12}u)-\frac{\sqrt{1-|v|^2}}{2}m\theta \rho\\
				&=\frac{\theta}{2}\mathcal{T}_{m,v}(u)-\frac{\sqrt{1-|v|^2}}{2}m\theta \rho-\frac{N\theta^{1+\frac1N}}{2N+2}\|u\|_{2+\frac2N}^{2+\frac2N}-\frac{\mu \theta^{\frac{q+2}{2}}}{q+2} \|u\|_{q+2}^{q+2}\\
				&=\theta^{\frac{q+2}{2}}\left(E(u)-\frac{\sqrt{1-|v|^2}}{2}m\rho\right)+\frac{\sqrt{1-|v|^2}}{2}m\rho\left(\theta^{\frac{q+2}{2}}-\theta\right)\\
				&~~~~+\frac{1}{2}\left(\theta-\theta^{\frac{q+2}{2}}\right)\mathcal{T}_{m,v}(u)+\frac{N}{2N+2}\left(\theta^{\frac{q+2}{2}}-\theta^{1+\frac1N}\right)\|u\|_{2+\frac2N}^{2+\frac2N}\\
				&\le\theta^{\frac{q+2}{2}}\left(E(u)-\frac{\sqrt{1-|v|^2}}{2}m\rho\right)\le \theta^{\frac{q+2}{2}}\left(e(\rho)+\epsilon-\frac{\sqrt{1-|v|^2}}{2}m\rho\right)\le \theta^{\frac{q+2}{2}}[\tilde{e}(\rho)+\epsilon].
			\end{aligned}
		\end{equation}

		Now, for every $0<\lambda<a$, we claim that
		\begin{equation}\label{eq3.62}
			\tilde{e}(a)< \frac{a}{\lambda}\tilde{e}(\lambda).
		\end{equation}
		Indeed, if $\tilde{e}(\lambda)>0$, then it is obvious that \eqref{eq3.62} holds since, by \eqref{eq2.40}, $\tilde{e}(a)<0$. Otherwise, if $\tilde{e}(\lambda)\le0$, substituting $\theta=\frac{a}{\lambda}$ and $\rho=\lambda$ into \eqref{eq3.1}, and letting $\epsilon<(\theta^{-\frac{q}{2}}-1)\tilde{e}(\rho)$, then it follows that
		$$\tilde{e}(a)=\tilde{e}(\theta \rho)\le \theta^{\frac{q+2}{2}}[\tilde{e}(\rho)+\epsilon]<\theta\tilde{e}(\rho)=\frac{a}{\lambda}\tilde{e}(\lambda).$$
		Furthermore, replacing $\lambda$ with $a-\lambda$ in \eqref{eq3.62}, then this yields that
		\begin{equation}\label{eq3.63}
			\tilde{e}(a)< \frac{a}{a-\lambda}\tilde{e}(a-\lambda).
		\end{equation}
		Combining \eqref{eq3.62} and \eqref{eq3.63}, one can see that
		\begin{equation*}
			\tilde{e}(a)=\frac{\lambda}{a}\tilde{e}(a)+\frac{a-\lambda}{a}\tilde{e}(a)< \tilde{e}(\lambda)+\tilde{e}(a-\lambda)
		\end{equation*}
		and we conclude.
	\end{proof}

	Studying the behaviour of minimizing sequences, we have
	\begin{lem}\label{lem2.6}
		Under the assumptions of  Theorem \ref{Thm1}-\eqref{E2}, any minimizing sequence for $e(a)$ is 
		bounded in $H^{\frac{1}{2}}(\mathbb{R}^N)$.
		
	\end{lem}
	
	\begin{proof} 
		Observe that the case $0<a<a^\ast_v$ is completely different from the case $a=a^\ast_v$ taking into account the Gagliardo-Nirenberg inequality \eqref{eq1.5}, so we divide the proof into two cases.
		
		\textbf{Case 1: $0<a< a^*_v$.} In this case, by \eqref{eq10280938} and \eqref{eq2.18}, we attain the desired result.

		\textbf{Case 2: $a= a^\ast_v$.} Let $\{\psi_n\}\subset\mathcal{S}_{a^*_v}$ be a minimizing sequence for $e(a^\ast_v)$, such that
		\begin{equation*}
			e(a^\ast_v)\le E(\psi_n)\le e(a^\ast_v)+\frac{1}{n}.
		\end{equation*}
		Then we deduce from \eqref{eq1.5} and \eqref{eq2.5} that
		\begin{equation}\label{eq3.87}
			e(a^\ast_v)+\frac{1}{n}
			\ge  E(\psi_n)
			\ge\frac{1}{2}T_v(\psi_n)-\frac{N}{2N+2}\|\psi_n\|_{2+\frac2N}^{2+\frac2N}
			-\frac{\mu}{q+2} \|\psi_n\|_{q+2}^{q+2}
			\ge -\frac{\mu}{q+2} \|\psi_n\|_{q+2}^{q+2}\ge 0.
		\end{equation}
		Analogously,
		\begin{equation}\label{eq3.88}
			\begin{aligned}
				0\le\frac{1}{2}T_v(\psi_n)-\frac{N}{2N+2}\|\psi_n\|_{2+\frac2N}^{2+\frac2N}\le e(a^\ast_v)+\frac{1}{n}.
			\end{aligned}
		\end{equation}
		
		Assume by contradiction that $\{\psi_n\}$ is unbounded in $H^{\frac{1}{2}}(\mathbb{R}^N)$.
		Then, by \eqref{eq10280938},
		\begin{equation}\label{eq3.91}
			\eps_n:=(T_v(\psi_n))^{-1}\to 0 \text{ as } n\to \infty.
		\end{equation}
		Furthermore, set
		\begin{equation}\label{eq3.92}
			\tilde{\psi}_n:=\eps_n^{\frac{N}{2}}\psi_n(\eps_n \cdot).
		\end{equation}
		From \eqref{eq3.87} and \eqref{eq3.88} it follows  that
		\begin{equation}\label{eq3.89}
			0\le -\frac{\mu}{q+2} \|\tilde{\psi}_n\|_{q+2}^{q+2}\le \eps_n^{\frac{Nq}{2}}e(a^\ast_v)+\frac{1}{n}\eps_n^{\frac{Nq}{2}}\to 0, \text{ as } n\to \infty,
		\end{equation}
		and, using \eqref{riscalamento},
		\begin{equation}\label{eq3.90}
			\begin{aligned}
				0&\le\frac{1}{2}T_v(\tilde{\psi}_n)-\frac{N}{2N+2}\|\tilde{\psi}_n\|_{2+\frac2N}^{2+\frac2N}\le \eps_ne(a^\ast_v)+\frac{1}{n}\eps_n\to 0, \text{ as } n\to \infty.
			\end{aligned}
		\end{equation}
		Moreover, by \eqref{riscalamento}, \eqref{eq3.91}, and \eqref{eq3.92}, we infer that
		\begin{equation*}
			T_v(\tilde{\psi}_n)=\eps_nT_v(\psi_n)=1,
		\end{equation*}
		which, together with \eqref{eq10280938}, yields that $\{\tilde{\psi}_n\}$ is bounded in $H^{\frac{1}{2}}(\mathbb{R}^N)$. Hence, using \eqref{eq3.89} and the interpolation inequality, we further derive that
		\begin{equation*}
			\|\tilde{\psi}_n\|_{2+\frac2N}^{2+\frac2N}\to 0, \quad \text{ as } n\to \infty,
		\end{equation*}
		which leads to a contradiction with \eqref{eq3.90}.
	\end{proof}

	With the help of the above conclusions, we now prove the existence result of Theorem  \ref{Thm1}-\eqref{E2} by applying the concentration-compactness principle \cite{Lion} (see also \cite{LW22}).
	Firstly, for any minimizing sequences $\{\psi_k\}\subset\mathcal{S}_a$ of $e(a)$, we introduce the following L\'{e}vy concentration function
	\begin{equation*}
		Q_k(r):=\sup_{y\in\R^N}\int_{B_r(y)}|\psi_k(x)|^2dx.
	\end{equation*}
	Since $\{Q_k\}$ is  sequence of monotone
	and uniformly bounded functions, by the Helly's selection theorem,  we can find a convergent subsequence, denoted again by $\{Q_k\}$, such that there is a non-decreasing function $Q(r)$ satisfying
	$$\lim_{k\to \infty}Q_k(r)=Q(r), \text{ for all } r>0.$$
	Noting that $0\le Q_k(r)\le a$, there exists $\beta\in [0,a]$ such that
	\begin{equation}\label{eq3.73}
		\lim_{r\to +\infty}Q(r)=\beta.
	\end{equation}

	Then, we divide the proof into two cases as follows.
	
	{\bf Case 1: $\beta=0$.} By \eqref{eq3.73}, we have
	$$\lim_{k \to \infty}\sup _{y\in \mathbb{R}^N}\int_{B_R(y)}|\psi_{k} |^2dx =0, \mbox{ for all }R>0, $$
	which, together with Lemma \ref{lem1.3}, indicates that
	\begin{equation}\label{eq3.77}
		\lim_{k \to \infty}\|\psi_k\|_p^p =0,\quad\text{for all }2< p<{2N}/{(N-1)}.
	\end{equation}
	Using \eqref{APP.C}, from \eqref{eq2.40} and \eqref{eq3.77}, we conclude  that
	\begin{equation*}\
		\begin{aligned}
			\frac{\sqrt{1-|v|^2}}{2}ma
			>e(a)=\lim\limits_{k \to \infty} E(\psi_k)
			=\lim\limits_{k \to \infty} \frac{1}{2}\mathcal{T}_{m,v}(\psi_k)
			\geq \frac{\sqrt{1-|v|^2}}{2}ma,
		\end{aligned}
	\end{equation*}
	which is absurd. Hence, the case $\beta=0$ does not hold.
	
	{\bf Case 2: $\beta\neq0$.} In this case, for $R>0$ large enough, \eqref{eq3.73} gives that $\frac{\beta}{2}<Q(R)<\frac{3\beta}{2}$. By passing to a subsequence if necessary, we have that
	\begin{equation*}
		\frac{\beta}{2}\le\lim_{k \to \infty}\sup _{y\in \mathbb{R}^N}\int_{B_R(y)}|\psi_{k} |^2dx\le\frac{3\beta}{2}.
	\end{equation*}
	Then, there exists some $\{y_k\}\subset\R^N$ such that
	\begin{equation}\label{eq3.80}
		\frac{\beta}{2}\le\lim_{k \to \infty}\int_{B_R(0)}|\psi_{k}(x-y_k)|^2dx\le\frac{3\beta}{2}.
	\end{equation}
	Moreover, noting that  $\{\psi_k(\cdot-y_k)\}$ is bounded
	by  Lemma \ref{lem2.6}
	in $H^{\frac{1}{2}}(\mathbb{R}^N)$, up to the subsequence, then there exists $\varphi \in H^{\frac{1}{2}}(\mathbb{R}^N)$ such that, as $k\to \infty$,
	$$\psi_k(\cdot-y_k)\rightharpoonup\varphi \text{ weakly in } H^{\frac{1}{2}}(\mathbb{R}^N),$$
	$$\psi_k(\cdot-y_k)\to\varphi \text{ in } L_{\rm loc}^{2}(\mathbb{R}^N),$$
	$$\psi_k(\cdot-y_k)\to\varphi \text{ a.e. in } \mathbb{R}^N,$$
	which, together with \eqref{eq3.80}, implies that
	$\|\varphi\|_2^2\ge \frac{\beta}{2}>0.$
	
	Set $\nu_k:=\psi_k(\cdot-y_k)-\varphi$. By Brezis-Lieb Lemma,
	we have that
	\begin{equation}\label{eq3.81}
		E(\psi_k)=E(\nu_k)+ E(\varphi)+o_k(1),
	\end{equation}
	and
	\begin{equation}\label{eq3.82}
		a=\|\psi_k\|_2^2=\|\nu_k\|_2^2+\|\varphi\|_2^2+o_k(1).
	\end{equation}
	Let $\beta_1:=\|\varphi\|_2^2$ and $\delta:=a-\beta_1$.  Then \eqref{eq3.82} yields that
	$$\lim_{k \to \infty}\|\nu_k\|_2^2=\delta\ge 0.$$
	If $\delta >0$, taking $\tilde{\nu}_k:=b_k\nu_k$ and $b_k:=\frac{\sqrt{\delta}}{\|\nu_k\|_2}$, then
	$\|\tilde{\nu}_k\|_2^2=\delta$ and $\lim\limits_{k\to \infty}b_k=1.$
	For $k$ large enough, we have that
	\begin{equation}\label{eq3.83}
		\begin{aligned}
			e(\delta)\le E(\tilde{\nu}_k)
			&=E(\nu_k)+o_k(1).
		\end{aligned}
	\end{equation}
	Hence, it follows from \eqref{eq3.81} and \eqref{eq3.83} that
	\begin{equation*}
		e(a)=E(\psi_k)+o_k(1)
		= E(\varphi)+E(\nu_k)+o_k(1)= E(\varphi)+E(\tilde{\nu}_k)+o_k(1)\ge e(\beta_1)+e(\delta),
	\end{equation*}
	which is impossible by using Lemma \ref{lem2.4}. Thus, we get $\delta =0$, that is,
	$$\lim_{k\to \infty}\|\psi_k\|_2^2=\|\varphi\|_2^2=a.$$
	Furthermore, by interpolation inequality, there holds
	\begin{equation*}
		\psi_k(\cdot-y_k)\rightarrow\varphi \text{~in}~ L^\kappa(\mathbb{R}^N),  ~\forall 2\le \kappa <{2N}/{(N-1)}.
	\end{equation*}
	By weakly lower semi-continuity, 
	we derive 
	\begin{equation*}
		\begin{aligned}
			e(a)
			=\lim\limits_{k \to \infty} E(\psi_k)
			\ge E(\varphi)\ge e(a),
		\end{aligned}
	\end{equation*}
	so $\varphi$ is a minimizer and then a solution.

	\subsection{Proof of Theorem \ref{Thm1}-(\ref{E3})}
	\
	\\
	Recall that in this case $\mu=0$. Our goal is a consequence of the following results.
	\begin{lem}\label{lem1.4}
		Let $a^*_v\ge a>0$ and $m>0$, then we have that \eqref{cacata} and \eqref{nacacata} hold.
		Moreover, if $a^*_v=a$ and $m\ge0$, then $e(a^*_v)=0$.
	\end{lem}
	
	\begin{proof}
		The first part follows arguing as in Lemma \ref{prop1.2} and Lemma \ref{lem2.3}.
		Moreover, for $a=a_v^*$, if $m=0$, recalling \eqref{eq2.18} and \eqref{eq3.10}, we also have that
		$e(a^*_v)=0$. If, instead, $m>0$ the conclusion is achieved letting $\tau\to+\infty$ in \eqref{eq3.10}.
	\end{proof}
	
	\begin{lem}\label{lem1.5}
		Fixed $a^*_v\ge a>0$ and $m>0$, there holds the strict subadditivity inequality
		\begin{equation}\label{eq4.7}
			e(a)<e(\lambda )+e(a-\lambda ), \text{ where } 0<\lambda<a.
		\end{equation}
		Moreover, the function $e(t)-\frac{\sqrt{1-|v|^2}}{2}mt$ is strictly decreasing and $e(t)$ is continuous,  with respect to $t\in(0,a]$.
	\end{lem}
	
	\begin{proof}
		Following the arguments of Lemma \ref{lem2.2}, we need to give the proof only for the continuity of $e(t)$ in $t=a^*_v$. Recalling \eqref{minh} and \eqref{cacata}, we infer that $$\lim_{a\to a^*_v}e(a)=0$$ and so, by Lemma \ref{lem1.4}, we conclude.
	\end{proof}

	We consider two scenarios.
	
	{\bf Case \eqref{E3i}: $a^*_v>a>0$ and $m>0$.} 
	As in the proof of Theorem \ref{Thm1}-\eqref{E1}, then it follows from Lemmas \ref{lem1.4}--\ref{lem1.5} that there is a minimizer to problem \eqref{eq1.3}. Thus, the proof of Theorem \ref{Thm1}-\eqref{E3}-\eqref{E3i} is finished.

	{\bf Case \eqref{E3ii}: $a=a^*_v$ and $m\ge0$.} By Lemma \ref{lem1.4}, we need to consider only the case $m=0$. Clearly, one can see that $Q_v$ is a minimizer of problem \eqref{eq1.3} by
	Lemma \ref{lem1.2} and 
	Proposition \ref{prop1.1}.
	Therefore, the proof of Theorem \ref{Thm1}-\eqref{E3}-\eqref{E3ii} is complete.

	\subsection{Proof of Theorem \ref{Thm1}-(\ref{NE4})}
	\
	\\
	The proof is divided into five distinct cases as follows.
	
	{\bf Case \eqref{eq:NE4i}: $a=a^*_v$,  $m\ge0$ and $\mu>0$.} In this case, letting $u_t:=t^{\frac{N}{2}}Q_v(t \cdot)$ with $t>0$, where $Q_v\in H^{\frac{1}{2} }(\mathbb{R}^N)$ is a solution of equation \eqref{eq1.6},  we have that $\|u_t\|_2^2=\|Q_v\|_2^2=a^\ast_v=a$.
	Therefore, from \eqref{riscalamento}, \eqref{eq1.9}, and \eqref{eq2.5}, it follows  that
	\begin{equation*}
		\begin{aligned}
			E(u_t)
			&\le\frac{1}{2}\int_{\mathbb{R}^N}\bar{u}_t(\sqrt{-\Delta}+m+iv\cdot \nabla )u_t dx-\frac{N}{2N+2}\|u_t\|_{2+\frac2N}^{2+\frac2N}-\frac{\mu}{q+2} \|u_t\|_{q+2}^{q+2}\\
			&=\frac{ma^\ast_v}{2}+\frac{t}{2}T_v(Q_v)-\frac{t}{2}\frac{N}{N+1} \|Q_v\|_{2+\frac{2}{N}}^{2+\frac{2}{N}}-\frac{\mu}{q+2} t^{\frac{Nq}{2}}\|Q_v\|_{q+2}^{q+2}\\
			&=\frac{ma^\ast_v}{2}-\frac{\mu}{q+2} t^{\frac{Nq}{2}}\|Q_v\|_{q+2}^{q+2}
			\to -\infty\quad \text{ as }t\to +\infty,
		\end{aligned}
	\end{equation*}
	which implies that $e(a)=-\infty$ and so the non-existence of minimizers for problem \eqref{eq1.3}, proving \eqref{eq:NE4i}.

	{\bf Case \eqref{eq:NE4iii}: $a>a^*_v$, $m\ge0$ and  $\mu\in\R$.} In this case, recalling the definition of $\varphi_\tau$ given by \eqref{eq2.3},  then we deduce from \eqref{eq1.9} and \eqref{eq2.5} that
	\begin{equation*}
		\begin{aligned}
			E(\varphi_\tau)
			&\le \frac{ma}{2}+\frac{\tau}{2}\frac{a}{a^\ast_v}T_v(Q_v)-\frac{N\tau}{2N+2}\left(\frac{a}{a^\ast_v}\right)^{\frac{N+1}{N}} \|Q_v\|_{2+\frac{2}{N}}^{2+\frac{2}{N}}-\frac{\mu}{q+2} \left(\frac{a}{a^\ast_v}\right)^{\frac{q+2}{2}}\tau^{\frac{Nq}{2}}\|Q_v\|_{q+2}^{q+2}\\
			&=\frac{ma}{2}+\frac{N\tau a}{2}\left[1-\left(\frac{a}{a^\ast_v}\right)^{\frac{1}{N}}\right]-\frac{\mu}{q+2} \left(\frac{a}{a^\ast_v}\right)^{\frac{q+2}{2}}\tau^{\frac{Nq}{2}}\|Q_v\|_{q+2}^{q+2}
			\to -\infty, \text{ as }\tau\to +\infty,
		\end{aligned}
	\end{equation*}
	which implies that $e(a)=-\infty$ and so the problem \eqref{eq1.3} has no minimizers, showing \eqref{eq:NE4iii}.

	{\bf Case \eqref{eq:NE4ii}: $a^*_v\ge a>0$, $m=0$ and $\mu<0$.} In this case, if there exists a solution $u$ to equation \eqref{eq1.2} with $\| u\|_2^2 =a$, then by the Pohozaev identity (see
	\cite[Lemma 2.2]{ZL22})
	\begin{equation*}
		T_v(u)=\mu\frac{Nq}{q+2}\|u\|_{q+2}^{q+2}+\frac{N}{N+1}\|u\|_{2+\frac{2}{N}}^{2+\frac{2}{N}},
	\end{equation*}
	we deduce from \eqref{eq1.5} that
	\begin{equation*}
		\begin{aligned}
			0>\mu\frac{Nq}{q+2}\|u\|_{q+2}^{q+2}
			&=T_v(u)-\frac{N}{N+1}\|u\|_{2+\frac{2}{N}}^{2+\frac{2}{N}}\ge\left[1-\left(\frac{a}{a^\ast_v}\right)^{\frac1N}\right]T_v(u)\ge 0,
		\end{aligned}
	\end{equation*}
	which is a contradiction. Hence, problem \eqref{eq1.3} has no minimizers.
	
	In what follows, we focus on establishing the energy estimates of $e(a)$. On the one hand, for any $u\in \mathcal{S}_a$,
	it follows  from \eqref{eq1.5} that
	\begin{equation*}
		\begin{aligned}
			E(u)
			&\ge\frac{1}{2}T_v(u)-\frac{N}{2N+2}\|u\|_{2+\frac2N}^{2+\frac2N}\ge\frac{1}{2}\left[1-\left(\frac{a}{a^\ast_v}\right)^{\frac1N}\right]T_v(u)\ge0,
		\end{aligned}
	\end{equation*}
	which implies that
	\begin{equation}\label{eq2.2}
		e(a)\ge0.
	\end{equation}
	On the other hand, choosing the  test function $\varphi_\tau\in \mathcal{S}_a$  given in \eqref{eq2.3}, then we infer from \eqref{eq1.9} that
	\begin{equation}\label{eq2.4}
		\begin{aligned}
			e(a)
			&\le \frac{1}{2}T_v(\varphi_\tau)-\frac{N}{2N+2}\|\varphi_\tau\|_{2+\frac2N}^{2+\frac2N}-\frac{\mu}{q+2} \|\varphi_\tau\|_{q+2}^{q+2}\\
			&=\frac{\tau}{2}\frac{a}{a^\ast_v}T_v( Q_v)-\frac{N\tau}{2N+2}\left(\frac{a}{a^\ast_v}\right)^{\frac{N+1}{N}} \|Q_v\|_{2+\frac{2}{N}}^{2+\frac{2}{N}}-\frac{\mu}{q+2} \left(\frac{a}{a^\ast_v}\right)^{\frac{q+2}{2}}\tau^{\frac{Nq}{2}}\|Q_v\|_{q+2}^{q+2}\\
			&=\frac{N\tau a}{2}\left[1-\left(\frac{a}{a^\ast_v}\right)^{\frac{1}{N}}\right]-\frac{\mu}{q+2} \left(\frac{a}{a^\ast_v}\right)^{\frac{q+2}{2}}\tau^{\frac{Nq}{2}}\|Q_v\|_{q+2}^{q+2}
			\to 0 \quad \text{ as }\tau\to 0^+.
		\end{aligned}
	\end{equation}
	Then, combining \eqref{eq2.2} and \eqref{eq2.4}, this gives $e(a)=0$, which completes the proof of \eqref{eq:NE4ii}.

	{\bf Case \eqref{eq:NE4iv}: $a=a^\ast_v$, $m>0$, and $\mu=0$.} We know that $e(a^*_v)=0$ by Theorem \ref{Thm1}-\eqref{E3}-\eqref{E3ii}.
	
	Suppose by contradiction that there exists a minimizer $u\in\mathcal{S}_a$ for $e(a)$. Then, from \eqref{eq1.5}, we derive that
	\begin{equation*}
		\mathcal{T}_{m,v}(u)=\frac{N}{N+1}\|u\|_{2+\frac2N}^{2+\frac2N}\le 
		T_v(u),
	\end{equation*}
	which is impossible since $m>0$. So the non-existence of minimizers for  $e(a)$ is proved, giving \eqref{eq:NE4iv}.
	
	{\bf Case \eqref{eq:NE4v}: $a^\ast_v>a>0$, $m=0$, and $\mu=0$.} In this case, by a similar argument as \eqref{eq2.2} and \eqref{eq2.4}, it is easy to conclude that $e(a)=0$.
	Furthermore, we argue by contradiction and assume that there exists a minimizer $u\in\mathcal{S}_a$ for $e(a)$. By \eqref{eq1.5}, then it follows that
	\begin{equation*}
		\begin{aligned}
			0=e(a)
			&=\frac{1}{2}T_v(u)-\frac{N}{2N+2}\|u\|_{2+\frac2N}^{2+\frac2N}\ge \frac{1}{2}\left[1-\left(\frac{a}{a^\ast_v}\right)^{\frac{1}{N}}\right]T_v(u)> 0,
		\end{aligned}
	\end{equation*}
	which is a contradiction. Thus, there exists no minimizers for $e(a)$.

	\section{Blow-up behaviour as \texorpdfstring{$a\nearrow a^*_v$}{}}\label{se4}
	In this section, we investigate the asymptotic behaviour of minimizers to problem \eqref{eq1.3}  as $a\nearrow a^*_v$. We start with the following estimates.
	\begin{lem}\label{lem6.1}
		Assume that $a^*_v>a>0$, $m\ge0$ and $\mu>0$. Let $u_a$ be a minimizer of $e(a)$. Then, as  $a\nearrow a^\ast_v$,
		we have that
		\begin{equation}\label{eq01191008}
			e(a)\sim -\left[1-\left(\frac{a}{a^\ast_v}\right)^{\frac1N}\right]^{-\frac{Nq}{2-Nq}},
		\end{equation}
		\begin{equation}\label{eq4.80}
			\|u_a\|_{q+2}^{q+2}\sim\left[1-\left(\frac{a}{a^\ast_v}\right)^{\frac1N}\right]^{-\frac{Nq}{2-Nq}},
		\end{equation}
		and
		\begin{equation}\label{eq4.95}
			T_v(u_a)\sim \|u_a\|_{2+\frac{2}{N}}^{2+\frac{2}{N}}\sim \left[1-\left(\frac{a}{a^\ast_v}\right)^{\frac1N}\right]^{-\frac{2}{2-Nq}}.
		\end{equation}
	\end{lem}
	
	\begin{proof} The proof is organized into four steps.
		
		\textbf{Step 1: The estimates of $e(a)$.} Clearly, the lower bound follows from \eqref{eq4.46}, namely, there holds
		\begin{equation}\label{pippo}
			e(a)\gtrsim-\left[1-\left(\frac{a}{a^\ast_v}\right)^{\frac1N}\right]^{-\frac{Nq}{2-Nq}}
			a^{\frac{q+2-Nq}{2-Nq}}
		\end{equation}
		On the other hand, applying the same trial function $\varphi_\tau\in \mathcal{S}_a$  as that of  \eqref{eq2.3}, 
		we deduce from \eqref{eq2.5} that
		\begin{equation*}
			\begin{aligned}
				e(a)&\le E(\varphi_\tau)
				\le\frac{1}{2}T_v(\varphi_\tau)+\frac{ma}{2}-\frac{N}{2N+2}\|\varphi_\tau\|_{2+\frac2N}^{2+\frac2N}-\frac{\mu}{q+2} \|\varphi_\tau\|_{q+2}^{q+2}\\
				&=\frac{N\tau a}{2}\left[1-\left(\frac{a}{a^\ast_v}\right)^{\frac{1}{N}}\right]-\frac{\mu}{q+2} \left(\frac{a}{a^\ast_v}\right)^{\frac{q+2}{2}}\tau^{\frac{Nq}{2}}\|Q_v\|_{q+2}^{q+2}+\frac{ma}{2},
			\end{aligned}
		\end{equation*}
		which, together with Lemma \ref{lemA.7}
		applied to $g_1$, implies that
		\begin{equation}\label{eq4.48}
			\begin{aligned}
				e(a)
				&\le\frac{Nq-2}{2}\left(\frac{Nq}{2}\right)^{\frac{Nq}{2-Nq}}\left\{\frac{N a}{2}\left[1-\left(\frac{a}{a^\ast_v}\right)^{\frac{1}{N}}\right]\right\}^{-\frac{Nq}{2-Nq}}\left[\frac{\mu}{q+2} \left(\frac{a}{a^\ast_v}\right)^{\frac{q+2}{2}}\|Q_v\|_{q+2}^{q+2}\right]^{\frac{2}{2-Nq}}+\frac{ma}{2}\\
				&\lesssim-\left[1-\left(\frac{a}{a^\ast_v}\right)^{\frac1N}\right]^{-\frac{Nq}{2-Nq}}
				a^{\frac{q+2-Nq}{2-Nq}} +\frac{ma}{2}.
			\end{aligned}
		\end{equation}
		Now, for $a$ close to $a_v^*$, since $q<2/N$, by \eqref{pippo} and \eqref{eq4.48}, the argument is thereby concluded.

		\textbf{Step 2: The estimates of $\|u_a\|_{q+2}^{q+2}$.} By using \eqref{eq10280938}, \eqref{eq1.5} and \eqref{eq2.5}, we note that
		\begin{equation*}
			\begin{aligned}
				e(a)
				&\ge \frac{1}{2}\left[1-\left(\frac{a}{a^\ast_v}\right)^{\frac 1N}\right]T_v(u_a)-\frac{\mu}{q+2} \|u_a\|_{q+2}^{q+2}\\
				&\ge \frac{1}{2}\left[1-\left(\frac{a}{a^\ast_v}\right)^{\frac 1N}\right](1-|v|)\|u_a\|_{\dot{H}^{\frac{1}{2} }(\R^N)}^2
				-\frac{\mu}{q+2}\|u_a\|_{q+2}^{q+2}\ge-\frac{\mu}{q+2}\|u_a\|_{q+2}^{q+2},
			\end{aligned}
		\end{equation*}
		which, together with \eqref{eq4.48}, implies that
		\begin{equation}\label{eq4.74}
			\|u_a\|_{q+2}^{q+2}\gtrsim\left[1-\left(\frac{a}{a^\ast_v}\right)^{\frac1N}\right]^{-\frac{Nq}{2-Nq}}
			a^{\frac{q+2-Nq}{2-Nq}} -\frac{ma}{2}.
		\end{equation}
		Consequently, we derive the lower bound.

		To see the upper bound, by the Young's inequality and \eqref{eq2.19}, for any $\eps>0$, it follows that
		\begin{equation*}
			\begin{aligned}
				\|u_a\|_{q+2}^{q+2}
				&\le C_{v,N,q}\left(T_v(u_a)\right)^{\frac{Nq}{2} }\|u_a\|_2^{{q+2}-{Nq}}\le \eps T_v(u_a)+C\eps^{-\frac{Nq}{2-Nq}}a^{\frac{q+2-Nq}{2-Nq}},
			\end{aligned}
		\end{equation*}
		which, combined with $e_{v}(a)<0$ (see \eqref{eq4.48} for $a$ close to $a_v^*$), \eqref{eq2.5}, \eqref{astarv}, \eqref{eq1.5}, implies that
		\begin{equation*}
			\begin{aligned}
				-\frac{\mu}{q+2}\|u_a\|_{q+2}^{q+2}
				&\ge e(a)-\frac{\mu}{q+2}\|u_a\|_{q+2}^{q+2}
				=\frac{1}{2}\mathcal{T}_{m,v}(u_a)-\frac{N}{2N+2} \|u_a\|_{2+\frac{2}{N}}^{2+\frac{2}{N}}-\frac{2\mu}{q+2}\|u_a\|_{q+2}^{q+2}\\
				&\ge \frac{1}{2}\left[1-\left(\frac{a}{a^\ast_v}\right)^{\frac 1N}-\frac{4\mu\eps}{q+2}\right]T_v(u_a)-C\eps^{-\frac{Nq}{2-Nq}}a^{\frac{q+2-Nq}{2-Nq}}.
			\end{aligned}
		\end{equation*}
		Taking $\eps=\frac{q+2}{8\mu}\left[1-\left(\frac{a}{a^\ast_v}\right)^{\frac 1N}\right]$, then we get
		\begin{equation}\label{eq4.79}
			\begin{aligned}
				-\frac{\mu}{q+2}\|u_a\|_{q+2}^{q+2}
				&\ge \frac{1}{4}\left[1-\left(\frac{a}{a^\ast_v}\right)^{\frac 1N}\right]T_v(u_a)
				-C\left[1-\left(\frac{a}{a^\ast_v}\right)^{\frac 1N}\right]^{-\frac{Nq}{2-Nq}}a^{\frac{q+2-Nq}{2-Nq}}\\
				&\gtrsim-\left[1-\left(\frac{a}{a^\ast_v}\right)^{\frac 1N}\right]^{-\frac{Nq}{2-Nq}}a^{\frac{q+2-Nq}{2-Nq}}.
			\end{aligned}
		\end{equation}
		Thus, combining \eqref{eq4.74} and  \eqref{eq4.79}, this yields that \eqref{eq4.80},
		for $a$ close to $a_v^*$.
		
		\textbf{Step 3: The estimates of $T_v(u_a)$.}
		By \eqref{eq2.19} and \eqref{eq4.74}, we have  that
		\begin{equation*}
			\begin{aligned}
				&C_{v,N,q}\left(T_v(u_a)\right)^{\frac{Nq}{2} }a^{\frac{q+2}{2}-\frac{Nq}{2}}\ge\|u_a\|_{q+2}^{q+2}\gtrsim\left[1-\left(\frac{a}{a^\ast_v}\right)^{\frac1N}\right]^{-\frac{Nq}{2-Nq}},
			\end{aligned}
		\end{equation*}
		which gives that
		\begin{equation*}
			T_v(u_a)\gtrsim\left[1-\left(\frac{a}{a^\ast_v}\right)^{\frac1N}\right]^{-\frac{2}{2-Nq}}.
		\end{equation*}
		
		On the other hand, we set
		$$\varepsilon_a:=\left[1-\left(\frac{a}{a^\ast_v}\right)^{\frac1N}\right]^{\frac{2}{2-Nq}} \quad \text{and} \quad \psi_a:=\varepsilon_a^{\frac{N}{2}}u_a(\varepsilon_a\cdot).$$
		Noting that by \eqref{riscalamento}
		\begin{equation*}
			T_v(\psi_a) =\left[1-\left(\frac{a}{a^\ast_v}\right)^{\frac1N}\right]^{\frac{2}{2-Nq}}T_v(u_a),
		\end{equation*}
		thus, to obtain the upper bound, it suffices to prove
		\begin{equation}\label{puppu}
			T_v(\psi_a) \le C, \text{ as } a\nearrow a^\ast_v.
		\end{equation}
		Arguing by contradiction, let us assume that
		$T_v(\psi_a) \to +\infty$,  as  $a\nearrow a^\ast_v$.
		Define
		$$\eta_a:=\left(T_v(\psi_a)\right)^{-1} \quad \text{ and } \quad \tilde{\psi}_a:=\eta_a^{\frac{N}{2}}\psi_a(\eta_a \cdot),$$
		Then, there holds
		\begin{equation}\label{eq4.83}
			\eta_a\to 0,   \text{ as } a\nearrow a^\ast_v,
		\end{equation}
		and
		\begin{equation}\label{eq4.84}
			T_v(\tilde{\psi}_a)=\eta_aT_v(\psi_a)=1,
		\end{equation}
		which, combined with \eqref{eq10280938}, indicates that $\{\tilde{\psi}_{a}\}$ is bounded in $H^{\frac{1}{2}}(\mathbb{R}^N)$. Therefore, by interpolation inequality and Sobolev embedding theorem, from \eqref{eq4.80} and \eqref{eq4.83}, it follows that
		\begin{equation}\label{eq4.85}
			\begin{aligned}
				\|\tilde{\psi}_a\|_{2+\frac2N}^{2+\frac2N}
				&\le \|\tilde{\psi}_a\|_{\frac{2N}{N-1}}^{(1-\delta)(2+\frac2N)}\|\tilde{\psi}_a\|_{q+2}^{\delta(2+\frac2N)}
				\le C \left[(\eta_a\varepsilon_a)^{\frac{Nq}{2}}\|u_a\|_{q+2}^{q+2}\right]^{\frac{\delta(2+\frac2N)}{q+2}}
				\lesssim \eta_a^{\frac{Nq\delta}{2(q+2)}\left(2+\frac2N\right)} 
				\to 0,
			\end{aligned}
		\end{equation}
		as $ a\nearrow a^\ast_v$, where $\frac{1}{2+\frac{2}{N}}=\frac{\delta}{q+2}+\frac{1-\delta}{\frac{2N}{N-1}}$.
		Moreover, since $u_a$ is a minimizer of $e(a)$, applying \eqref{eq1.5}, \eqref{eq01191008}, \eqref{eq4.80} and recalling the definition of $\varepsilon_a$,  we further obtain that
		\begin{equation*}
			\begin{aligned}
				0&\le\frac{1}{2}T_v(\tilde{\psi}_a)-\frac{N}{2N+2}\|\tilde{\psi}_a\|_{2+\frac2N}^{2+\frac2N}
				=\varepsilon_a\eta_a\left(\frac{1}{2}T_v(u_a)-\frac{N}{2N+2}\|u_a\|_{2+\frac2N}^{2+\frac2N}\right)\\
				&\le\eta_a\left(\varepsilon_ae(a)+\frac{\mu}{q+2}\varepsilon_a\|u_a\|_{q+2}^{q+2}\right)
				\lesssim \eta_a\left[1-\left(\frac{a}{a^\ast_v}\right)^{\frac 1N}\right]\to 0,\text{ as }  a\nearrow a^\ast_v,
			\end{aligned}
		\end{equation*}
		which, together with \eqref{eq4.84}, means that
		\begin{equation*}
			\lim_{a\nearrow a^\ast_v}\|\tilde{\psi}_a\|_{2+\frac2N}^{2+\frac2N}=\frac{N+1}{N}.
		\end{equation*}
		This contradicts \eqref{eq4.85}. So we have \eqref{puppu} obtaining 
		\begin{equation}\label{eq4.89}
			T_v(u_a)\sim\left[1-\left(\frac{a}{a^\ast_v}\right)^{\frac1N}\right]^{-\frac{2}{2-Nq}},  \text{ as } a\nearrow a^\ast_v.
		\end{equation}
		
		\textbf{Step 4: The estimates of $\|u_a\|^{2+\frac{2}{N}}_{2+\frac{2}{N}}$.}
		Observe that
		\begin{equation*}
			\begin{aligned}
				\frac{N}{2N+2} \|u_a\|_{2+\frac{2}{N}}^{2+\frac{2}{N}}
				&=-e(a)+\frac{1}{2}\mathcal{T}_{m,v}(u_a)-\frac{\mu}{q+2}\|u_a\|_{q+2}^{q+2}\ge -e(a)+\frac{1}{2}T_v(u_a)-\frac{\mu}{q+2}\|u_a\|_{q+2}^{q+2}.
			\end{aligned}
		\end{equation*}
		This together with \eqref{eq01191008}, \eqref{eq4.80}, and \eqref{eq4.89},  indicates that
		\begin{equation*}
			\begin{aligned}
				&\left[1-\left(\frac{a}{a^\ast_v}\right)^{\frac1N}\right]^{\frac{2}{2-Nq}}\frac{N}{2N+2}\|u_a\|_{2+\frac{2}{N}}^{2+\frac{2}{N}}\\
				&\ge\left[1-\left(\frac{a}{a^\ast_v}\right)^{\frac1N}\right]^{\frac{2}{2-Nq}}\left( -e(a)+\frac{1}{2}T_v(u_a)-\frac{\mu}{q+2}\|u_a\|_{q+2}^{q+2}\right)\gtrsim1,  \text{ as } a\nearrow a^\ast_v.
			\end{aligned}
		\end{equation*}
		On the other hand, it follows from \eqref{eq1.5} and \eqref{eq4.89} that
		\begin{equation*}
			\begin{aligned}
				\|u_a\|_{2+\frac{2}{N}}^{2+\frac{2}{N}}
				&
				\lesssim\left[1-\left(\frac{a}{a^\ast_v}\right)^{\frac1N}\right]^{-\frac{2}{2-Nq}},  \text{ as } a\nearrow a^\ast_v.
			\end{aligned}
		\end{equation*}
		Therefore, the desired result is established.
	\end{proof}

	Based on the above energy estimates, we now prove Theorem \ref{Thm3}, that is, the blow-up behaviour of the minimizers to problem (\ref{eq1.3}) as $a\nearrow a^\ast_v$.
	
	\begin{proof}[Proof of Theorem \ref{Thm3}] 
		In the following, for simplicity, we write $a\nearrow a^\ast_v$ instead of $a_n\nearrow a^\ast_v$, eventually up to a subsequence.
		Define
		$$\eps_a:=\left[1-\left(\frac{a}{a^\ast_v}\right)^{\frac1N}\right]^{\frac{2}{2-Nq}} \quad \text{and} \quad \tilde{\omega}_a:=\eps_a^{\frac{N}{2}}u_a(\eps_a\cdot).$$
		We have that $\|\tilde{\omega}_a\|_2^2=\|u_a\|_2^2=a$ and $\eps_a\to 0$ as $a\nearrow a^\ast_v$. Moreover, we also infer from \eqref{eq4.95} that, as $a\nearrow a^\ast_v$,
		\begin{equation}\label{eq01151612}
			\|\tilde{\omega}_a\|_{2+\frac{2}{N}}^{2+\frac{2}{N}}=\eps_a \|u_a\|_{2+\frac{2}{N}}^{2+\frac{2}{N}}\sim 1,
		\end{equation}
		and, by \eqref{riscalamento},
		\begin{equation}\label{eq01151628}
			T_v(\tilde{\omega}_a)=\eps_aT_v(u_a)\sim 1.
		\end{equation}
		Hence, the sequence $\{\tilde{\omega}_a\}$ is bounded in $H^{\frac12}(\R^N)$ 
		in virtue of \eqref{eq10280938}. 
		By Lemma \ref{lem1.3} and \eqref{eq01151612}, there exists a sequence $\{y_{a}\}\subset\mathbb{R}^N$ such that
		\begin{equation}\label{eq01151652}
			\lim\limits_{a\nearrow a^\ast_v}\int_{B_{1}(y_{a})}|\tilde{{\omega}}_{a} |^2dx>0.
		\end{equation}
		
		Set
		\begin{equation}\label{eq01151659}
			\omega_{a}:=\tilde{{\omega}}_{a} (\cdot+y_{a})=\eps_{a}^{\frac{N}{2} }u_{a}\left(\eps_{a}(\cdot+y_{a})\right).
		\end{equation}
		We derive from \eqref{eq01151652} that
		\begin{equation}\label{eq01151660}
			\lim\limits_{a\nearrow a^\ast_v}\int_{B_{1}(0)}|{{\omega}}_{a} |^2dx>0.
		\end{equation}
		Since, by \eqref{eq01151628} and \eqref{eq10280938},
		$\{\omega_{a}\}$ is bounded in $H^{\frac{1}{2}}(\mathbb{R}^N)$, going if necessary to a subsequence, there exists some $\omega_*\in H^{\frac{1}{2}}(\mathbb{R}^N)\setminus \{0\}$ such that
		$\omega_{a}\overset{}\rightharpoonup\omega_*$ weakly in $H^{\frac{1}{2}}(\mathbb{R}^N)$, as $a\nearrow a^\ast_v$.
		From (\ref{eq01151660}) and Fatou's lemma, then we infer that
		\begin{equation}\label{eq01151728}
			\begin{aligned}
				0<\lim\limits_{a\nearrow a^\ast_v}\int_{B_1(0)}|{{\omega}}_{a} |^2dx=\int_{B_1(0)}|\omega_*|^2dx
				&\leq\|{\omega}_*\|_2^{2}
				\leq\lim\limits_{a\nearrow a^\ast_v}\|{\omega}_{a}\|_2^{2}=a^\ast_v.
			\end{aligned}
		\end{equation}
		On the other hand, since $u_a$ is a minimizer of $e(a)$, it solves the following Euler-Lagrange
		equation
		\begin{equation}\label{eq01150823}
			(\sqrt{-\Delta+m^2})u_a+i(v\cdot \nabla )u_a =\lambda_a u_a+\mu |u_a|^qu_a+|u_a|^{\frac2N}u_a,
		\end{equation}
		and so the Lagrange multiplier $\lambda_a\in \R$ satisfies
		\begin{equation*}
			a\lambda_a=2e(a)- \frac{q\mu}{q+2}\|u_a\|_{q+2}^{q+2}
			-\frac{1}{N+1} \|u_a\|_{2+\frac{2}{N}}^{2+\frac{2}{N}}.
		\end{equation*}
		Then, from Lemma \ref{lem6.1}, we deduce  that, up to a subsequence, there exists a positive constant $\gamma$ such that
		\begin{equation}\label{eq01151600}
			\eps_a\lambda_a\to -\gamma, \,\,\text{as}\,\, a\nearrow a^\ast_v.
		\end{equation}
		In addition, by \eqref{eq01151659} and \eqref{eq01150823}, we notice that ${\omega}_{a}$ satisfies
		\begin{equation}\label{eq01151729}
			\left(\sqrt{-\Delta+(\eps_{a}m)^2}+ iv\cdot \nabla \right)  \omega_{a}=
			\eps_{a}\lambda_a\omega_{a}+\mu\eps_{a}^{\frac{2-Nq}{2}}|\omega_{a}|^q\omega_{a}
			+|\omega_{a}|^{\frac{2}{N}}\omega_{a}.
		\end{equation}
		Observe that, for any  test function $\varphi$, we get
		\begin{align*}
			&    \left|\int_{\R^N}\bar \varphi \left(\sqrt{-\Delta+(\eps_{a}m)^2} \omega_a- \sqrt{-\Delta}\omega_*\right) dx\right|\\
			&\quad \le  \left|\int_{\R^N}\bar \varphi \left(\sqrt{-\Delta+(\eps_{a}m)^2}- \sqrt{-\Delta}\right)\omega_a dx\right|
			+ \left|\int_{\R^N}\bar \varphi\sqrt{-\Delta} (\omega_a-\omega_*) dx\right|\\
			&\quad=\left|\int_{\R^N}
			(\sqrt{|k|^2+(\eps_{a}m)^2}-|k|)
			\hat{\omega}_a\hat{{\varphi}} dk\right|
			+ \left|\int_{\R^N}\bar \varphi\sqrt{-\Delta} (\omega_a-\omega_*) dx\right|\\
			&\quad
			\le C \varepsilon_a
			\|\omega_a\|_2 \|\varphi\|_2
			+ \left|\int_{\R^N}(\bar{\omega}_a-\bar{\omega}_*) \sqrt{-\Delta} \varphi dx\right|\to 0, \quad \text{ as } a\nearrow a^\ast_v.
		\end{align*}
		This, combined with \eqref{eq01151600} and the fact that  $\omega_{a}\overset{}\rightharpoonup\omega_*$
		weakly in $ H^{\frac{1}{2}}(\mathbb{R}^N)$, implies that $\omega_*$ satisfies
		\begin{equation}\label{eq01151730}
			(\sqrt{-\Delta} +iv\cdot \nabla )\omega_*= |\omega_*|^{\frac{2}{N} } \omega_*-\gamma\omega_*.
		\end{equation}
		So, by the Pohozaev identity
		\eqref{eq1.9}, 
		\begin{equation}\label{eq01151734}
			T_v(\omega_*)=\frac{N}{N+1}\|\omega_*\|_{2+\frac2N}^{2+\frac2N}=N\gamma\|\omega_*\|_2^2.
		\end{equation}
		Putting $\omega_*$ into (\ref{eq1.5}), it follows from  (\ref{eq01151728}) and \eqref{eq01151734}  that
		\begin{equation*}
			\left({a^\ast_v}\right)^{\frac1N}
			=\frac{\|\omega_*\|_{2+\frac2N}^{2+\frac2N}}{\frac{N+1}{N\left(a^\ast_v\right)^{\frac1N}}
				T_v(\omega_*)}
			\leq\|\omega_*\|_2^{\frac2N}
			\leq\left({a^\ast_v}\right)^{\frac1N},
		\end{equation*}
		which says that $\omega_*$ optimizes the Gagliardo-Nirenberg inequality (\ref{eq1.5}) and $\omega_{a}$ converges to $\omega_*$
		strongly in $L^2({\mathbb{R}^N})$ as $a\nearrow a^\ast_v$. According to the interpolation inequality and Sobolev embedding theorem,
		we further conclude that
		$$\omega_{a}\overset{}\rightarrow\omega_* \text{~in~} L^\kappa(\mathbb{R}^N) \text{ as } a\nearrow a^\ast_v ,~ \forall 2\le \kappa<{2N}/{(N-1)}.$$
		Thus, by (\ref{eq01151729}) and (\ref{eq01151730}), we obtain that
		\begin{equation}\label{eq01151748}
			\lim_{a\nearrow a^\ast_v}\int_{\mathbb{R}^N}\bar{\omega}_{a}\left(\sqrt{-\Delta+(\eps_{a}m)^2 }+ iv\cdot \nabla \right)  \omega_{a}dx
			=T_v(\omega_*).
		\end{equation}
		Moreover, 
		\begin{align*}
			0\le \int_{\mathbb{R}^N}\bar{\omega}_{a}\left(\sqrt{-\Delta+(\eps_{a}m)^2 }-\sqrt{-\Delta } \right)  \omega_{a}dx
			&=\int_{\mathbb{R}^N}\left(\sqrt{{|\xi |}^2+(\eps_{a}m)^2 }-{|\xi|}\right)|
			\mathcal{F}\left[{\omega}_{a}\right]|^2d\xi\\
			&\leq\eps_{a}m\int_{\mathbb{R}^N}\left|\mathcal{F}\left[{\omega}_{a}\right ]\right|^2d\xi
			=\eps_{a}m\|{\omega}_{a}\|_2^2=\eps_{a}m a\to 0, 
		\end{align*}
		as $a\nearrow a^\ast_v$, which, together with \eqref{eq01151748}, yields that
		\begin{equation}\label{eq4.99}
			\lim_{a\nearrow a^\ast_v}T_v(\omega_a)
			=T_v(\omega_*),
		\end{equation}
		namely
		\begin{equation}\label{eq4.100}
			\omega_{a}\overset{}\rightarrow \omega_* \text{~in} ~ H^{\frac{1}{2}}(\mathbb{R}^N), \text{ as } a\nearrow a^\ast_v.
		\end{equation}
		Let
		\begin{equation}\label{eq3.67}
			G_0:=\gamma^{-\frac N2}\omega_*\left(\gamma^{-1}\cdot\right).
		\end{equation}
		Clearly, $G_0$ optimizes the Gagliardo-Nirenberg inequality (\ref{eq1.5}) and satisfies the equation (\ref{eq1.6}). Meanwhile, there holds
		\begin{equation*}
			\eps_{a}^{\frac{N}{2} }u_{a}\left(\eps_{a}(\cdot+y_{a})\right)=\omega_{a}\overset{}\rightarrow \omega_*=\gamma^{\frac N2}{G_0\left(\gamma \cdot\right)}
			\text{~in~} H^{\frac{1}{2}}(\mathbb{R}^N), \text{ as }a\nearrow a^\ast_v.
		\end{equation*}

		It remains to determine $\gamma$.  Recalling the definitions of $\eps_a$ and  $\omega_a$, we deduce from \eqref{eq1.9}, \eqref{eq1.5},  \eqref{eq2.5}, \eqref{eq4.99}, and \eqref{eq3.67} that
		\begin{equation}\label{eq4.98}
			\begin{aligned}
				\liminf_{a\nearrow a^\ast_v}
				\eps_a^\frac{Nq}{2}\left(\frac{1}{2}\mathcal{T}_{m,v}(u_a)-\frac{N}{2N+2} \|u_a\|_{2+\frac{2}{N}}^{2+\frac{2}{N}}\right)
				&\ge \frac{1}{2}\liminf_{a\nearrow a^\ast_v}\eps_a T_v(u_a)
				= \frac{1}{2}\liminf_{a\nearrow a^\ast_v}T_v(\omega_a)\\
				&=\frac{1}{2}T_v(\omega_*)
				=\frac{1}{2}\gamma T_v(G_0)=\frac{1}{2}\gamma Na^\ast_v.
			\end{aligned}
		\end{equation}
		Moreover, by \eqref{eq4.100} and \eqref{eq3.67}, we also have that
		\begin{equation}\label{eq3.65}
			\lim_{a\nearrow a^\ast_v}\eps_a^\frac{Nq}{2}\|u_a\|_{q+2}^{q+2}
			=\lim_{a\nearrow a^\ast_v}\|\omega_a\|_{q+2}^{q+2}
			=\|\omega_*\|_{q+2}^{q+2}
			=\gamma^{\frac{Nq}{2}}\|G_0\|_{q+2}^{q+2}.
		\end{equation}
		Thus, it follows from \eqref{eq4.98} and \eqref{eq3.65}, together with \eqref{eqming1} in Lemma \ref{lemA.7}, that
		\begin{equation}\label{eq3.66}
			\begin{aligned}
				\liminf_{a\nearrow a^\ast_v}\eps_a^\frac{Nq}{2} e(a)
				&\ge \frac{Na^\ast_v}{2}\gamma-\frac{\mu}{q+2}\gamma^{\frac{Nq}{2}}\|G_0\|_{q+2}^{q+2}\\
				&\ge \frac{Nq-2}{2}\left(\frac{Nq}{2}\right)^{\frac{Nq}{2-Nq}}{\left(\frac{Na^\ast_v}{2}\right)}^{-\frac{Nq}{2-Nq}}
				{\left(\frac{\mu}{q+2}\|G_0\|_{q+2}^{q+2}\right)}^{\frac{2}{2-Nq}}.
			\end{aligned}
		\end{equation}
		On the other hand, we consider the following trial function
		$$\zeta_\tau:=\tau^{\frac{N}{2}}{\sqrt{\frac{a}{a^\ast_v}}}G_0(\tau \cdot),~~\tau>0.$$
		Clearly, $\|\zeta_\tau\|_2^2=a$. Applying \eqref{eq1.9} and \eqref{eq2.5},  we have that
		\begin{equation*}
			\begin{aligned}
				e(a)
				&\le\frac{1}{2}\mathcal{T}_{m,v}(\zeta_\tau)-\frac{N}{2N+2}\|\zeta_\tau\|_{2+\frac2N}^{2+\frac2N}-\frac{\mu}{q+2} \|\zeta_\tau\|_{q+2}^{q+2}\\
				&\le\frac{ma}{2}+\frac{1}{2}T_v(\zeta_\tau)-\frac{N}{2N+2}\|\zeta_\tau\|_{2+\frac2N}^{2+\frac2N}-\frac{\mu}{q+2} \|\zeta_\tau\|_{q+2}^{q+2}\\
				&\le\frac{ma}{2}+\frac{Na}{2}\left[1-\left(\frac{a}{a^\ast_v}\right)^{\frac{1}{N}}\right]\tau-\frac{\mu}{q+2} \left(\frac{a}{a^\ast_v}\right)^{\frac{q+2}{2}}\tau^{\frac{Nq}{2}}\|G_0\|_{q+2}^{q+2}.
			\end{aligned}
		\end{equation*}
		Then, by choosing $\tau=\eps_a^{-1}t$, for $t>0$, one can see that
		\begin{equation*}
			\begin{aligned}
				e(a)
				&\le\frac{ma}{2}+\frac{Na}{2}\eps_a^{-\frac{Nq}{2}}t-\frac{\mu}{q+2} \left(\frac{a}{a^\ast_v}\right)^{\frac{q+2}{2}}\eps_a^{-\frac{Nq}{2}}
				t^{\frac{Nq}{2}}\|G_0\|_{q+2}^{q+2},
			\end{aligned}
		\end{equation*}
		Hence, it follows that
		\begin{equation*}
			\begin{aligned}
				\limsup_{a\nearrow a^\ast_v}\eps_a^{\frac{Nq}{2}}e(a)
				\le\frac{Na^\ast_v}{2}t-\frac{\mu}{q+2}t^{\frac{Nq}{2}}\|G_0\|_{q+2}^{q+2},
			\end{aligned}
		\end{equation*}
		for any $t>0$, which, combined with \eqref{eqming1} in Lemma \ref{lemA.7}, gives that
		\begin{equation*}
			\begin{aligned}
				&\limsup_{a\nearrow a^\ast_v}\eps_a^{\frac{Nq}{2}}e(a)\le \frac{Nq-2}{2}\left(\frac{Nq}{2}\right)^{\frac{Nq}{2-Nq}}{\left(\frac{Na^\ast_v}{2}\right)}^{-\frac{Nq}{2-Nq}}
				{\left(\frac{\mu}{q+2}\|G_0\|_{q+2}^{q+2}\right)}^{\frac{2}{2-Nq}}.
			\end{aligned}
		\end{equation*}
		Combining this last inequality with \eqref{eq3.66}, we deduce that
		\[
		\frac{Na^\ast_v}{2}\gamma-\frac{\mu}{q+2}\gamma^{\frac{Nq}{2}}\|G_0\|_{q+2}^{q+2}
		=\frac{Nq-2}{2}\left(\frac{Nq}{2}\right)^{\frac{Nq}{2-Nq}}{\left(\frac{Na^\ast_v}{2}\right)}^{-\frac{Nq}{2-Nq}}
		{\left(\frac{\mu}{q+2}\|G_0\|_{q+2}^{q+2}\right)}^{\frac{2}{2-Nq}},
		\]
		namely $\gamma$ is the global minimizer of the function
		\[
		t>0\mapsto \frac{Na^\ast_v}{2}t-\frac{\mu}{q+2}t^{\frac{Nq}{2}}\|G_0\|_{q+2}^{q+2},
		\]
		which is achieved at
		\begin{equation*}
			\gamma=\left[\frac{q\mu\|G_0\|_{q+2}^{q+2}}{(q+2)a^\ast_v}\right]^{\frac{2}{2-Nq}}.
		\end{equation*}
		
		This completes the proof of Theorem \ref{Thm3}.
	\end{proof}

	\section{Asymptotic behaviour as \texorpdfstring{$m\to 0^+$}{}}
	In this section, we  investigate the asymptotic behaviour of $e_m(a)$ and minimizers to problem \eqref{eq1.3} as $m\to 0^+$, that is, Theorem \ref{Thm2}.
	
	\begin{proof}[Proof of Theorem \ref{Thm2}]
		Fix $\mu\le 0$. Observe that 
		\[
		\lim_{m\to 0^+}\min_{\tau>0}h(\tau)=0.
		\]
		and so, by \eqref{cacata}, 
		\[
		\lim_{m\to 0^+}e_m(a)=0.
		\]

		Moreover, if $a_v^*>a>0$ and, for $m>0$, we take $\mu_m^*<\mu\le 0$ and $u_m$ is a corresponding minimizer for $e_m(a)$, we have, by \eqref{eq1.5},
		\begin{equation}\label{eq3.5}
			\begin{aligned}
				e_m(a)
				\ge\frac{1}{2}T_v(u_m)-\frac{N}{2N+2}\|u_m\|_{2+\frac2N}^{2+\frac2N}\ge\frac{1}{2}\left[1-\left(\frac{a}{a^\ast_v}\right)^{\frac1N}\right]T_v(u_m).
			\end{aligned}
		\end{equation}
		If $m\to 0$, we deduce that $T_v(u_m)\to 0$ and so, by \eqref{eq10280938}, $u_m\to 0$ in $\dot{H}^\frac12(\R^N)$. Hence, by interpolation, we infer that  $u_m \to 0$ in $L^\kappa(\R^N)$, for all $2< \kappa\le{2N}/{(N-1)}$.
	\end{proof}

	\section{Limiting profiles as \texorpdfstring{$\mu\to 0^-$}{}}
	
	This section is focused on the behaviour of minimizers as $\mu\to 0^-$.
	To emphasize the dependence on the parameter $\mu$, we rewrite the constrained minimization problem \eqref{eq1.3} as
	\begin{equation}\label{eq4.2}
		e_{\mu}(a):=\inf_{\varphi\in \mathcal{S}_a} E_{\mu}(\varphi),
	\end{equation}
	where
	\begin{equation*}
		\begin{aligned}
			E_{\mu}(\varphi):=\frac{1}{2}\int_{\mathbb{R}^N}\bar{\varphi}(\sqrt{-\Delta+m^2}+iv\cdot \nabla)\varphi dx-\frac{N}{2N+2}\|\varphi\|_{2+\frac2N}^{2+\frac2N}-\frac{\mu}{q+2} \|\varphi\|_{q+2}^{q+2}.
		\end{aligned}
	\end{equation*}

	First, we deal with the case $a^\ast_v>a>0$, $m>0$ and $0>\mu>\mu_m^*$ (Theorem \ref{Thm1}-(\ref{E2}) guaranties the existence of a minimizer $\phi_\mu$). 
	\begin{proof}[Proof of Theorem \ref{Thm6}]
		Let $\phi_{\mu_k}$ be a minimizer of  $e_{\mu_k}(a)$, where $\{\mu_k\}\subset(\mu_m^*,0)$ satisfies $\lim\limits_{k\to \infty}\mu_k = 0^-$. By \eqref{eq4.2}, for any $\epsilon>0$, there exists some $\phi_\epsilon\in H^{\frac{1}{2}}(\mathbb{R}^N)$ with $\|\phi_\epsilon\|_2^2=a$, such that $E_{0}(\phi_\epsilon)\leq e_{0}(a)+\epsilon$.
		Then
		$$e_{\mu_k}(a)
		\leq E_{\mu_k}(\phi_\epsilon)
		=E_{0}(\phi_\epsilon) + o_k(1)
		\leq e_{0}(a)+\epsilon + o_k(1),$$
		which, using the arbitrariness of $\epsilon$ and Lemma \ref{lem1.4}, implies that
		\begin{equation}\label{eq4.104}
			e_{\mu_k}(a)+o_k(1)\leq e_{0}(a)<\frac{\sqrt{1-|v|^2}}{2}ma.
		\end{equation}
		Arguing as in \eqref{eq3.5} and by \eqref{eq4.104}, we have
		$$\frac{1}{2}\left[1-\left(\frac{a}{a^\ast_v}\right)^{\frac1N}\right]T_v(\phi_{\mu_k})\le e_{\mu_k}(a)\le \frac{\sqrt{1-|v|^2}}{2}ma+o_k(1),$$
		and so $\{\phi_{\mu_k}\}$ is bounded in $H^{\frac{1}{2}}(\mathbb{R}^N)$.
		
		On the other hand, we note that
		\begin{align*}
			e_{\mu_k}(a)
			=E_{\mu_k}(\phi_{\mu_k})
			=E_{0}(\phi_{\mu_k})+o_k(1)
			\ge  e_{0}(a)+o_k(1),
		\end{align*}
		which, combined with \eqref{eq4.104}, yields that
		\begin{equation*}
			\lim_{k\to \infty}e_{\mu_k}(a)=\lim_{k\to \infty}E_{0}(\phi_{\mu_k})= e_{0}(a).
		\end{equation*}
		That is, $\{\phi_{\mu_k}\}\subset\mathcal{S}_a$ is a bounded minimizing sequence for $e_{0}(a)$.
		
		Repeating similar steps as the proof of Theorem \ref{Thm1}-\eqref{E3}-\eqref{E3i} (see also \cite[Theorem 1.2]{HYZ24}) and going if necessary to a subsequence, there exists some $\phi_0\in H^{\frac{1}{2}}(\mathbb{R}^N)$  such that
		$$\phi_{\mu_k}\to\phi_0 \text{~strongly~in} ~ H^{\frac{1}{2}}(\mathbb{R}^N),\,\, \text{as}\,\, k\to \infty,$$
		where $\phi_0$ is a minimizer of $e_{0}(a)$.
	\end{proof}
	
	Now we focus on the proof of Theorem \ref{Thm5}. Hereafter, we denote with $\varphi_\mu$ a minimizer of $e_{\mu}(a^\ast_v)$ for the case  $m>0$ and $\mu_m^*<\mu<0$ (see Theorem \ref{Thm1}-(\ref{E2})). Clearly, it satisfies the  Euler-Lagrange equation
	\begin{equation*}
		(\sqrt{-\Delta+m^2}+iv\cdot \nabla )\varphi_\mu-|\varphi_\mu|^{\frac{2}{N}}\varphi_\mu-\mu|\varphi_\mu|^{q}\varphi_\mu =\lambda_\mu \varphi_\mu,
	\end{equation*}
	where $\lambda_\mu\in \R$ is a suitable Lagrange multiplier. Moreover, the following holds
	
	\begin{lem}\label{lem3.2}
		We have that $\|\varphi_\mu\|_{\dot{H}^\frac{1}{2}(\R^N)}\to +\infty$ as $\mu\to0^-$.
		
	\end{lem}
	
	\begin{proof}
		Suppose by contradiction that there exists a sequence $\{\mu_n\}$ with $\mu_n\to 0^-$, as $n\to \infty$, such that $\{\varphi_{\mu_n}\}$ is  bounded in $H^{\frac{1}{2}}(\mathbb{R}^N)$. Then, up to a subsequence, there exists $\varphi_0\in H^{\frac{1}{2}}(\mathbb{R}^N)$ such that
		\begin{equation*}
			\varphi_{\mu_n}\rightharpoonup\varphi_0 \text{~weakly~in} ~ H^{\frac{1}{2}}(\mathbb{R}^N),\,\, \text{as}\,\, n\to \infty.
		\end{equation*}
		Proceeding as in Step 4 in \cite[Lemma 3.1]{LZW}, where the concentration-compactness principle was used, we can obtain a contradiction. For the readers' convenience, here we give the proof, which is divided into three steps.
		
		{\bf Step 1: Vanishing does not occur.} If Vanishing occurs, recalling Lemma \ref{lem1.3} and 
		\eqref{APP.C}, it then follows 
		\begin{equation}\label{eq4.6}
			e_{\mu_n}(a^\ast_v)=E_{\mu_n}(\varphi_{\mu_n})
			=\frac{1}{2}\mathcal{T}_{m,v}(\varphi_{\mu_n})+o_n(1)
			\ge \frac{\sqrt{1-|v|^2}}{2}ma^\ast_v+o_n(1)
		\end{equation}

		Next, recalling Lemma \ref{lem1.4}, we show that $\lim\limits_{n\to\infty}e_{\mu_n}(a^\ast_v)=e_{0}(a^\ast_v)=0$.
		Indeed, on one hand, 
		since $\varphi_{\mu_n}\to 0$ in $L^\kappa(\R^N)$, for all $\kappa\in (2,2N/(N-1))$, one can see that
		\begin{equation}\label{eq3.99}
			e_{\mu_n}(a^\ast_v)=E_{\mu_n}(\varphi_{\mu_n})
			=E_{0}(\varphi_{\mu_n})+o_n(1)\ge e_{0}(a^\ast_v)+o_n(1)
			=o_n(1).
		\end{equation}
		On the other hand, by the definition of $e_{0}(a^\ast_v)$, for any $\sigma>0$, there exists some $\varphi_\sigma\in H^{\frac{1}{2}}(\mathbb{R}^3)$ with $\|\varphi_\sigma\|_2^2=a^\ast_v$, such that $E_{0}(\varphi_\sigma)\leq e_{0}(a^\ast_v)+\sigma =\sigma$.
		Then we deduce that
		\begin{equation}\label{eq3.99b}
			e_{\mu_n}(a^\ast_v)=E_{\mu_n}(\varphi_{\mu_n})
			\leq E_{\mu_n}(\varphi_\sigma)
			=E_{0}(\varphi_\sigma)+o_n(1)\leq \sigma+o_n(1).
		\end{equation}
		Therefore, combining \eqref{eq3.99}, \eqref{eq3.99b},  and using the arbitrariness of $\sigma$ we have that
		$$\lim_{n\to \infty}e_{\mu_n}(a^\ast_v)= e_{0}(a^\ast_v)=0<\frac{\sqrt{1-|v|^2}}{2}ma^\ast_v,$$
		which leads to a contradiction with \eqref{eq4.6}. So, Vanishing is ruled out.

		{\bf Step 2: Dichotomy cannot occur.} If Dichotomy occurs, then there exists $\lambda \in(0,a^\ast_v)$ such that for each 
		$\varepsilon>0$, there exist two bounded sequences $\{\varphi^1_{\mu_n}\}$ and $\{\varphi^2_{\mu_n}\}$, with $\operatorname{dist}(supp\varphi^1_{\mu_n},supp \varphi^2_{\mu_n})\to \infty$ as $n\to \infty$, satisfying
		\begin{equation}\label{eq4.101}
			\lambda-\varepsilon\le \|\varphi^1_{\mu_n}\|_2^{2}\le \lambda+\varepsilon, 
			~~
			(a^\ast_v-\lambda)-\varepsilon\le \|\varphi^2_{\mu_n}\|_2^{2}\le (a^\ast_v-\lambda)+\varepsilon,
		\end{equation}
		for $n$ large enough. In addition, up to a subsequence, it follows that, for sufficient large $n$,
		\begin{equation}\label{eq4.11}
			E_{\mu_n}(\varphi_{\mu_n})-E_{\mu_n}(\varphi^1_{\mu_n})-E_{\mu_n}(\varphi^2_{\mu_n})+o_n(1)\ge -C(\varepsilon)-\delta_1(\varepsilon)+\mu_n\delta_2(\varepsilon),
		\end{equation}
		where $C(\varepsilon),\delta_1(\varepsilon),\delta_2(\varepsilon)\to 0$ as $\varepsilon\to0^+$ (see 
		\cite[Lemma 2.4]{FJL07} for details).
		Thus, from \eqref{eq3.99b} and \eqref{eq4.11} we derive that for every $\sigma>0$,
		\begin{equation*}
			\begin{aligned}
				e_{0}(a^\ast_v)+\sigma
				&\ge e_{\mu_n}(a^\ast_v)+o_n(1)
				=E_{\mu_n}(\varphi_{\mu_n})+o_n(1)\\
				&\ge E_{\mu_n}(\varphi^1_{\mu_n})
				+E_{\mu_n}(\varphi^2_{\mu_n})
				-C(\varepsilon)
				-\delta_1(\varepsilon)
				+o_n(1)\\
				&=
				E_{0}(\varphi^1_{\mu_n})
				+E_{0}(\varphi^2_{\mu_n})
				-C(\varepsilon)
				-\delta_1(\varepsilon)
				+o_n(1)\\
				&\ge e_{0}(\|\varphi^1_{\mu_n}\|^2_2)
				+e_{0}(\|\varphi^2_{\mu_n}\|^2_2)
				-C(\varepsilon)
				-\delta_1(\varepsilon)
				+o_n(1).
			\end{aligned}
		\end{equation*}
		Since, by Lemma \ref{lem1.5}, the function $t\in(0,a_v^*]\mapsto e_{0}(t)-\sqrt{1-|v|^2}mt/2$ is strictly decreasing, for $n$ large enough, by using \eqref{eq4.101},  we further get that
		\begin{equation*}
			\begin{split}
				e_{0}(a^\ast_v) +\sigma
				&\ge e_{0}(\|\varphi^1_{\mu_n}\|^2_2)
				-\frac{\sqrt{1-|v|^2}}{2}m\|\varphi^1_{\mu_n}\|^2_2\\
				&\quad
				+\frac{\sqrt{1-|v|^2}}{2}m\|\varphi^1_{\mu_n}\|^2_2
				-\frac{\sqrt{1-|v|^2}}{2}m(\lambda-\varepsilon)
				+\frac{\sqrt{1-|v|^2}}{2}m(\lambda-\varepsilon)\\
				&\quad
				+e_{0}(\|\varphi^2_{\mu_n}\|^2_2)
				-\frac{\sqrt{1-|v|^2}}{2}m\|\varphi^2_{\mu_n}\|^2_2\\
				&\quad
				+\frac{\sqrt{1-|v|^2}}{2}m\|\varphi^2_{\mu_n}\|^2_2
				-\frac{\sqrt{1-|v|^2}}{2}m(a^\ast_v-\lambda-\varepsilon)
				+\frac{\sqrt{1-|v|^2}}{2}m(a^\ast_v-\lambda-\varepsilon)\\
				&\quad
				-C(\varepsilon)-\delta_1(\varepsilon)
				+o_n(1)\\
				&\ge
				e_{0}(\lambda+\varepsilon)-\frac{\sqrt{1-|v|^2}}{2}m(\lambda+\varepsilon)
				+e_{0}(a^\ast_v-\lambda+\varepsilon)-\frac{\sqrt{1-|v|^2}}{2}m(a^\ast_v-\lambda+\varepsilon)\\
				&\quad
				+\frac{\sqrt{1-|v|^2}}{2}m(\lambda-\varepsilon)
				+\frac{\sqrt{1-|v|^2}}{2}m(a^\ast_v-\lambda-\varepsilon)
				-C(\varepsilon)-\delta_1(\varepsilon)
				+o_n(1)\\
				&=
				e_{0}(\lambda+\varepsilon)
				+e_{0}(a^\ast_v-\lambda+\varepsilon)
				-2\sqrt{1-|v|^2}m\varepsilon
				-C(\varepsilon)-\delta_1(\varepsilon)
				+o_n(1).
			\end{split}
		\end{equation*}
		Passing to the limit as $n\to\infty$ and $\varepsilon\to 0^+$,  applying the continuity of $e_{0}(a)$ in $a$ given by Lemma \ref{lem1.5}, and due to the arbitrariness of $\sigma>0$ we deduce that
		\begin{equation*}
			\begin{aligned}
				e_{0}(a^\ast_v)\ge e_{0}(\lambda)+e_{0}(a^\ast_v-\lambda),
			\end{aligned}
		\end{equation*}
		holds for some $0<\lambda <a^\ast_v$. This contradicts \eqref{eq4.7}. Hence, Dichotomy is excluded.
		
		{\bf Step 3: A contradiction is derived by Compactness.} Using the concentration-compactness principle, one can see that
		\begin{equation*}
			\varphi_{\mu_n}\rightarrow\varphi_0 \text{~strongly~in} ~  L^\kappa(\mathbb{R}^N)\,\, \text{as}\,\, n\to \infty, ~\forall 2\le \kappa<{2N}/{(N-1)}.
		\end{equation*}
		By \eqref{eq3.99b} and lower semi-continuity of the norm again, one has that for every $\sigma>0$
		\begin{equation*}
			e_{0}(a^\ast_v)+\sigma
			\ge \frac{1}{2}\mathcal{T}_{m,v}(\varphi_{0})-\frac{N}{2N+2}\|\varphi_{0}\|_{2+\frac2N}^{2+\frac2N}=E_{0}(\varphi_{0})\ge e_{0}(a^\ast_v),
		\end{equation*}
		which indicates that $\varphi_{0}$ is a minimizer of $e_{0}(a^\ast_v)$. However, by Theorem \ref{Thm1}-\eqref{NE4}-\eqref{eq:NE4iv}, we know that $e_{0}(a^\ast_v)$ has no minimizer, which leads to a contradiction.
		
		Thus, the proof  is completed.
	\end{proof}

	Define
	\begin{equation}\label{deftaumu}
		\tau_{\mu}:=\left(T_v(\varphi_\mu)\right)^{-1}.
	\end{equation}
	Based on the above results, we have the following estimates and blow-up behaviour.
	
	\begin{lem}\label{lem3.1}
		Let $m>0$ and $0>\mu>\mu_m^*$. Then the following energy estimates hold:
		\begin{equation}\label{eq4.23}
			\tau_{\mu}\to 0, \text{ as } \mu\to0^-,
		\end{equation}
		\begin{equation}\label{eq4.5}
			0\le e_\mu(a^\ast_v)\lesssim (-\mu)^{\frac{2}{2+Nq}},
		\end{equation}
		\begin{equation}\label{eq4.32}
			\lim_{\mu\to0^-}\tau_{\mu}\|\varphi_\mu\|_{2+\frac2N}^{2+\frac2N}=\frac{N+1}{N},
		\end{equation}
		and
		\begin{equation}\label{eq4.34}
			\lim_{\mu\to0^-}\frac{-\mu}{q+2} \|\varphi_\mu\|_{q+2}^{q+2}=0.
		\end{equation}
		Moreover,if $\mu_n\to0^-$, there exist $\tilde{\vartheta}>0$ and $\{\bar{y}_{n}\}\subset\R^N$, such that
		$$\tau_{\mu_n}^{\frac{N}{2} }\varphi_{\mu_n}\left(\tau_{\mu_n}(\cdot+\bar{y}_{n})\right)\to{\tilde{\vartheta}}^{\frac N2}{\tilde{W}_0(\tilde{\vartheta} \cdot)}
		\quad \text{in } H^{\frac{1}{2}}(\mathbb{R}^N),$$
		where $\tilde{W}_0$ optimizes the inequality (\ref{eq1.5}) and satisfies equation (\ref{eq1.6}).
	\end{lem}
	
	\begin{proof} 
		The argument proceeds in two steps.
		
		{\bf Step 1: Energy estimates.} Clearly, by \eqref{eq10280938} and Lemma \ref{lem3.2}, it yields that \eqref{eq4.23} holds.
		
		By Lemma \ref{prop1.2}, we have that $0\le e_\mu(a^\ast_v)$.
		As for the upper bound in \eqref{eq4.5}, arguing as in the proof of Lemma \ref{prop1.2} (see \eqref{eq3.10}), we have that for $\tau>0$,
		\begin{equation*}
			\begin{aligned}
				e_{\mu}(a^\ast_v)
				&\le\frac{m^2}{4\tau}\int_{\mathbb{R}^N}\frac{|\hat{Q}_v(k)|^2}{|k|}dk-\frac{\mu}{q+2}\tau^{\frac{Nq}{2}}\|Q_v\|_{q+2}^{q+2}.
			\end{aligned}
		\end{equation*}
		Then,
		by \eqref{eq5.2} in Lemma \ref{lemA.7},
		\begin{equation*}
			\begin{aligned}
				e_{\mu}(a^\ast_v)
				&\le \frac{2+Nq}{2}\left(\frac{m^2}{2Nq}\int_{\mathbb{R}^N}\frac{|\hat{Q}_v(k)|^2}{|k|}dk\right)^{\frac{Nq}{2+Nq}}
				\left(\frac{-\mu}{q+2}\|Q_v\|_{q+2}^{q+2}\right)^{\frac{2}{2+Nq}}
			\end{aligned}
		\end{equation*}
		and so we get \eqref{eq4.5}.
		
		Furthermore, from \eqref{eq1.5}, \eqref{eq2.5}, and \eqref{eq4.5}, we deduce that
		\begin{equation}\label{eq4.25}
			\begin{aligned}
				(-\mu)^{\frac{2}{2+Nq}}
				&\gtrsim e_{\mu}(a^\ast_v)=E_{\mu}(\varphi_\mu)
				\ge\frac{1}{2}T_v(\varphi_\mu)-\frac{N}{2N+2}\|\varphi_\mu\|_{2+\frac2N}^{2+\frac2N}
				\ge 0,
			\end{aligned}
		\end{equation}
		which, combined with the definition of $\tau_\mu$,
		gives \eqref{eq4.32}.
		
		Finally, \eqref{eq4.34} follows by \eqref{eq4.25}, observing that
		\begin{equation}\label{eq4.50}
			(-\mu)^{\frac{2}{2+Nq}}\gtrsim e_{\mu}(a^\ast_v)
			\geq
			\frac{1}{2}{T}_{v}(\varphi_\mu)-\frac{N}{2N+2}\|\varphi_\mu\|_{2+\frac2N}^{2+\frac2N}-\frac{\mu}{q+2} \|\varphi_\mu\|_{q+2}^{q+2}\ge\frac{-\mu}{q+2} \|\varphi_\mu\|_{q+2}^{q+2}\ge 0.
		\end{equation}
		
		{\bf Step 2: Blow-up analysis.} Let $\tilde{\omega}_\mu:=\tau_{\mu}^{\frac{N}{2}}\varphi_\mu(\tau_{\mu}\cdot)$. 
		From \eqref{riscalamento} and \eqref{deftaumu} we deduce that $T_v(\tilde{\omega}_\mu)=1$
		and so, by \eqref{eq10280938},
		$\{\tilde{\omega}_{\mu}\}$ is bounded in $H^{\frac{1}{2}}(\mathbb{R}^N)$.
		In the following, for simplicity, we write $\mu\to 0^-$ in place of $\mu_n\to 0^-$, eventually up to a subsequence.
		By Lemma \ref{lem1.3} and since, from \eqref{eq4.32},
		$$
		\lim\limits_{\mu\to0^-}\|\tilde{\omega}_\mu\|_{2+\frac2N}^{2+\frac2N}=\frac{N+1}{N},
		$$ one can see that there exists a sequence $\{\bar{y}_{\mu}\}\subset\mathbb{R}^N$ such that
		\begin{equation*}
			\lim_{\mu\to0^-}\int_{B_{1}(\bar{y}_{\mu})}|\tilde{{\omega}}_{\mu} |^2dx>0.
		\end{equation*}
		Set
		\begin{equation}\label{eq4.29}
			\omega_{\mu}:=\tilde{{\omega}}_{\mu} (\cdot+\bar{y}_{\mu})=\tau_{\mu}^{\frac{N}{2} }\varphi_{\mu}\left(\tau_{\mu}(\cdot+\bar{y}_{\mu})\right).
		\end{equation}
		Arguing as the first part of the proof of Theorem \ref{Thm3} (see \eqref{eq01151660}--\eqref{eq4.100})
		we conclude that there exists some $0\not\equiv \omega_0\in H^{\frac{1}{2}}(\mathbb{R}^N)$ which optimizes  \eqref{eq1.5}, that solves
		\begin{equation*}
			(\sqrt{-\Delta} +iv\cdot \nabla )\omega_0-|\omega_0|^{\frac{2}{N} } \omega_0=-\tilde{\vartheta}\omega_0
		\end{equation*}
		for a certain $\tilde{\vartheta}>0$ and such that

		\begin{equation}\label{eq4.45}
			\omega_{\mu}\overset{}\rightarrow \omega_0 \quad\text{~in} ~ H^{\frac{1}{2}}(\mathbb{R}^N), \quad \text{ as } \mu\to0^-.
		\end{equation}
		Then
		\begin{equation*}
			\tilde{W}_0:=\tilde{\vartheta}^{-\frac N2}\omega_0\left(\tilde{\vartheta}^{-1}\cdot\right),
		\end{equation*}
		is an optimizer for the inequality \eqref{eq1.5}, satisfies equation \eqref{eq1.6}, and
		$$\tau_{\mu}^{\frac{N}{2} }\varphi_{\mu}\left(\tau_{\mu}(\cdot+\bar{y}_{\mu})\right)=\omega_{\mu}\overset{}\rightarrow {\omega}_0={\tilde{\vartheta}}^{\frac N2}{\tilde{W}_0(\tilde{\vartheta} \cdot)}
		\text{~in~} H^{\frac{1}{2}}(\mathbb{R}^N), \text{ as } \mu\to0^-.$$
	\end{proof}

	In order to provide the exact blow-up rate, we proceed with the following lemma, which plays a key role in the asymptotic analysis.
	
	\begin{lem}
		Let $m>0$. Then, as $\mu\to0^-$, we have that
		\begin{equation}\label{eq4.49}
			e_\mu(a^\ast_v)\sim (-\mu)^{\frac{2}{2+Nq}},
		\end{equation}
		\begin{equation}\label{eq4.58}
			\|\varphi_\mu\|_{q+2}^{q+2}\sim(-\mu)^{-\frac{Nq}{2+Nq}},
		\end{equation}
		and
		\begin{equation}\label{eq4.68}
			T_v(\varphi_\mu)\sim \|\varphi_\mu\|_{2+\frac2N}^{2+\frac2N} \sim(-\mu)^{-\frac{2}{2+Nq}}.
		\end{equation}
	\end{lem}
	\begin{proof} 
		Let us estimate $e_\mu(a^\ast_v)$. The upper bound  is given by  \eqref{eq4.5}.
		In order to get the lower bound, considering $\omega_\mu$ as in \eqref{eq4.29} and using \eqref{riscalamento} and \eqref{eq1.5}, we have that
		\begin{equation}\label{eq4.56}
			\begin{aligned}
				e_{\mu}(a^\ast_v)
				&=E_{\mu}(\varphi_\mu)
				=\frac{1}{2\tau_{\mu}}\mathcal{T}_{\tau_\mu m,v}(\omega_\mu)-\frac{N}{(2N+2)\tau_{\mu}}\|\omega_\mu\|_{2+\frac2N}^{2+\frac2N}-\frac{\mu}{(q+2)\tau_{\mu}^{\frac{Nq}{2}}} \|\omega_\mu\|_{q+2}^{q+2}\\
				&\ge \frac{1}{2\tau_{\mu}}\int_{\mathbb{R}^N}\bar{{\omega}}_\mu(\sqrt{-\Delta+(\tau_{\mu}m)^2}-\sqrt{-\Delta})\omega_\mu dx-\frac{\mu}{(q+2)\tau_{\mu}^{\frac{Nq}{2}}} \|\omega_\mu\|_{q+2}^{q+2}\\
				&= \frac{\tau_{\mu}}{2}\int_{\mathbb{R}^N}\bar{{\omega}}_\mu\left(\frac{m^2}{\sqrt{-\Delta+(\tau_{\mu}m)^2}+\sqrt{-\Delta}}\right)\omega_\mu dx-\frac{\mu}{(q+2)\tau_{\mu}^{\frac{Nq}{2}}} \|\omega_\mu\|_{q+2}^{q+2}\\
				&\ge \frac{\tau_{\mu}}{2}\int_{\mathbb{R}^N}\bar{{\omega}}_\mu\left(\frac{m^2}{2\sqrt{-\Delta+1}}\right)\omega_\mu dx-\frac{\mu}{(q+2)\tau_{\mu}^{\frac{Nq}{2}}} \|\omega_\mu\|_{q+2}^{q+2},
			\end{aligned}
		\end{equation}
		for $\tau_{\mu}$ small enough. Moreover, 
		since 
		$$\int_{\mathbb{R}^N}\bar{{\omega}}_\mu\left(\frac{m^2}{\sqrt{-\Delta+1}}\right)\omega_\mu dx
		\le
		m^2\|\omega_\mu\|_2^2=m^2a,
		$$
		by \eqref{eq4.45} and the dominated convergence theorem, there holds
		$$\lim_{\mu\to0^-}\int_{\mathbb{R}^N}\bar{{\omega}}_\mu\left(\frac{m^2}{\sqrt{-\Delta+1}}\right)\omega_\mu dx=\int_{\mathbb{R}^N}\bar{{\omega}}_0\left(\frac{m^2}{\sqrt{-\Delta+1}}\right){\omega}_0 dx:=M_1>0,$$
		and
		$$\lim_{\mu\to0^-}\|\omega_\mu\|_{q+2}^{q+2}=\|{\omega}_0\|_{q+2}^{q+2}:=M_2>0,$$
		which, together with \eqref{eq4.56} and using \eqref{eqming4}, yields that, for $\mu$ sufficiently small
		$$e_{\mu}(a^\ast_v)\ge \frac{\tau_{\mu}}{2}C_1-\frac{\mu}{(q+2)\tau_{\mu}^{\frac{Nq}{2}}}C_2\gtrsim(-\mu)^{\frac{2}{2+Nq}}$$
		and so \eqref{eq4.49}.

		Now let us estimate $\|\varphi_\mu\|_{q+2}^{q+2}$.
		Arguing as in the proof of Lemma \ref{lem3.1} (see \eqref{eq4.50}), we have 
		\begin{equation*}
			\|\varphi_\mu\|_{q+2}^{q+2}\lesssim(-\mu)^{-\frac{Nq}{2+Nq}}.
		\end{equation*}
		On the other hand, taking $\tilde{\mu}:=\theta\mu$ with $\theta>1$, we have that
		\begin{equation*}
			\begin{aligned}
				e_{\tilde{\mu}}(a^\ast_v)
				\le E_{\tilde{\mu}}(\varphi_\mu)
				&=E_{{\mu}}(\varphi_\mu)+\frac{{\mu}}{q+2} \|\varphi_\mu\|_{q+2}^{q+2}-\frac{\tilde{\mu}}{q+2} \|\varphi_\mu\|_{q+2}^{q+2}=e_{\mu}(a^\ast_v)-\frac{\theta\mu-\mu}{q+2} \|\varphi_\mu\|_{q+2}^{q+2},
			\end{aligned}
		\end{equation*}
		which, combined with
		\eqref{eq4.49}, yields that
		\begin{equation*}
			\begin{aligned}
				\|\varphi_\mu\|_{q+2}^{q+2}
				&\ge \frac{q+2}{-\mu(\theta-1)}\left(e_{\tilde{\mu}}(a^\ast_v)-e_{\mu}(a^\ast_v)\right)\ge \frac{C_1(-\tilde{\mu})^{\frac{2}{2+Nq}}-C_2(-\mu)^{\frac{2}{2+Nq}}}{-\mu(\theta-1)}
				=\frac{C_1{\theta}^{\frac{2}{2+Nq}}-C_2}{\theta-1}(-\mu)^{-\frac{Nq}{2+Nq}}
			\end{aligned}
		\end{equation*}
		for $\mu$ close to $0$. Then, choosing $\theta>1$ large enough, there holds
		\begin{equation}\label{eq4.55}
			\|\varphi_\mu\|_{q+2}^{q+2}\gtrsim(-\mu)^{-\frac{Nq}{2+Nq}},  \text{ as } \mu\to0^-.
		\end{equation}
		
		Finally, let us estimate $T_v(\varphi_\mu)$ and $\|\varphi_\mu\|_{2+\frac2N}^{2+\frac2N}$. Arguing as in the proof of Lemma \ref{lem3.1} (see \eqref{eq4.25}), it follows  that
		\begin{equation}\label{eq4.57}
			(-\mu)^{\frac{2}{2+Nq}}\gtrsim\frac{1}{2}T_v(\varphi_\mu)-\frac{N}{2N+2}\|\varphi_\mu\|_{2+\frac2N}^{2+\frac2N}\ge 0,
		\end{equation}
		which implies that
		\begin{equation*}
			\lim_{\mu\to0^-}T_v(\varphi_\mu)\sim \lim_{\mu\to0^-} \|\varphi_\mu\|_{2+\frac2N}^{2+\frac2N}.
		\end{equation*}
		Thus, it suffices to establish the estimates of $T_v(\varphi_\mu)$.
		
		Note that, by   \eqref{eq2.19} and \eqref{eq4.55}, we have
		\begin{equation*}
			T_v(\varphi_\mu)
			\gtrsim
			\|\varphi_\mu\|_{q+2}^{\frac{2(q+2)}{Nq}}
			\gtrsim
			(-\mu)^{-\frac{2}{2+Nq}},  \text{ as } \mu\to0^-.
		\end{equation*}
		
		To derive the upper bound, taking $\phi_\mu:=((-\mu)^{\frac{2}{2+Nq}})^{\frac{N}{2}}\varphi_\mu((-\mu)^{\frac{2}{2+Nq}}\cdot)$, 
		then, by \eqref{riscalamento},
		\begin{equation}\label{eq4.62}
			T_v(\phi_\mu) =(-\mu)^{\frac{2}{2+Nq}}T_v(\varphi_\mu).
		\end{equation}
		Therefore, it remains to prove
		$$T_v(\phi_\mu) \le C,  \text{ as } \mu\to0^-.$$
		Arguing by contradiction, 
		we take a sequence $\{\mu_n\}$ with $\mu_n\to 0^-$ such that 
		$$T_v(\phi_{\mu_n}) \to +\infty, \text{ as } n\to\infty.$$
		Define 
		$\tilde{\eta}_n:=\left(T_v(\phi_{\mu_n})\right)^{-1}$ and 
		$\tilde{\phi}_{n}:=\tilde{\eta}_n^{\frac{N}{2}}\phi_{\mu_n}(\tilde{\eta}_n \cdot)=(\tilde{\eta}_n(-\mu_n)^{\frac{2}{2+Nq}})^{\frac{N}{2}}\varphi_{\mu_n}(\tilde{\eta}_n(-\mu_n)^{\frac{2}{2+Nq}}\cdot)$. Then
		\begin{equation}\label{eq4.64}
			\tilde{\eta}_n\to 0,\text{ as } n\to\infty,
		\end{equation}
		and by \eqref{eq4.62},
		\begin{equation}\label{eq4.65}
			T_v(\tilde{\phi}_n)=1,
		\end{equation}
		which, combined with \eqref{eq10280938}, indicates that $\{\tilde{\phi}_{n}\}$ is bounded in $H^{\frac{1}{2}}(\mathbb{R}^N)$. Applying interpolation inequality and Sobolev embedding theorem, we conclude from \eqref{eq4.58} and \eqref{eq4.64} that
		\begin{equation}\label{eq4.67}
			\|\tilde{\phi}_n\|_{2+\frac2N}^{2+\frac2N}
			\le \|\tilde{\phi}_n\|_{\frac{2N}{N-1}}^{(1-\delta)(2+\frac2N)}\|\tilde{\phi}_n\|_{q+2}^{\delta(2+\frac2N)}
			\lesssim \tilde{\eta}_n^\frac{q(N+1)\delta}{q+2}\to 0,\text{ as } n\to\infty,
		\end{equation}
		where
		$$\frac{1}{2+\frac{2}{N}}=\frac{\delta}{q+2}+\frac{1-\delta}{\frac{2N}{N-1}}.$$
		On the other hand, by \eqref{eq4.57}  and \eqref{eq4.64}, as $n\to \infty$, we have that 
		\begin{equation*}
			\begin{aligned}
				0&\le\frac{1}{2}T_v(\tilde{\phi}_n)-\frac{N}{2N+2}\|\tilde{\phi}_n\|_{2+\frac2N}^{2+\frac2N}
				=\tilde{\eta}_n(-\mu_n)^{\frac{2}{2+Nq}}\left(\frac{1}{2}T_v(\varphi_{\mu_n})-\frac{N}{2N+2}\|\varphi_{\mu_n}\|_{2+\frac2N}^{2+\frac2N}\right)
				\lesssim \tilde{\eta}_n(-\mu_n)^{\frac{4}{2+Nq}}\to 0.
			\end{aligned}
		\end{equation*}
		This combined with \eqref{eq4.65} leads to
		\begin{equation*}
			\lim_{n\to\infty}\|\tilde{\phi}_n\|_{2+\frac2N}^{2+\frac2N}=\frac{N+1}{N},
		\end{equation*}
		which is in contradiction with  \eqref{eq4.67}.
	\end{proof}

	Using the results above, we can now prove Theorem \ref{Thm5}.
	
	\begin{proof}[Proof of Theorem \ref{Thm5}] 
		In the following, for simplicity, in the sequences we omit the dependence on $n$.
		Define
		\begin{equation*}
			{\zeta}_{\mu}:=\left((-\mu)^{\frac{2}{2+Nq}}\right)^{\frac{N}{2} }\varphi_{\mu}\left((-\mu)^{\frac{2}{2+Nq}}\cdot\right).
		\end{equation*}
		Arguing as in the proof of Step 2 in Lemma \ref{lem3.1} and using \eqref{eq4.68},
		we infer that there exist  $\vartheta_0>0$ and 
		$\{y_{\mu}\}\subset\R^N$, such that, as $\mu\to0^-$,
		\begin{equation}\label{eq4.72}
			\zeta_\mu(\cdot +y_\mu)\rightarrow\zeta_0:={\vartheta_0}^{\frac N2}{W_0(\vartheta_0 \cdot)}, \quad
			\text{~in~} H^{\frac{1}{2}}(\mathbb{R}^N) ,
		\end{equation}
		where $W_0$ optimizes the inequality (\ref{eq1.5}) and satisfies equation (\ref{eq1.6}).
		
		Next, we compute the value of $\vartheta_0$. 
		Arguing as in \eqref{eq4.56}, by \eqref{eq1.5} we derive that
		\begin{equation*}
			e_{\mu}(a^\ast_v)
			\ge
			\frac{(-\mu)^{\frac{2}{2+Nq}}}{2}\inte\bar{\zeta}_\mu\left(\frac{m^2}{\sqrt{-\Delta+\left((-\mu)^{\frac{2}{2+Nq}}m\right)^2}
				+\sqrt{-\Delta}}\right)\zeta_\mu dx+\frac{(-\mu)^{\frac{2}{2+Nq}}}{q+2} \|\zeta_\mu\|_{q+2} ^{q+2},
		\end{equation*}
		which, together with \eqref{eq4.72}, Fatou's lemma and \eqref{eq5.2} in Lemma \ref{lemA.7}, implies that
		\begin{equation*}
			\begin{aligned}
				&\liminf_{\mu\to0^-}(-\mu)^{-\frac{2}{2+Nq}}e_{\mu}(a^\ast_v)\\
				&\ge\frac{1}{2}\liminf_{\mu\to0^-}\inte\bar{\zeta}_\mu\left(\frac{m^2}{\sqrt{-\Delta+\left((-\mu)^{\frac{2}{2+Nq}}m\right)^2}
					+\sqrt{-\Delta}}\right)\zeta_\mu dx+\frac{1}{q+2} \liminf_{\mu\to0^-}\|\zeta_\mu\|_{q+2} ^{q+2}\\
				&\ge\frac{1}{2}\inte\bar{\zeta}_0\left(\frac{m^2}{2\sqrt{-\Delta}}\right)\zeta_0 dx+\frac{1}{q+2} \|\zeta_0\|_{q+2} ^{q+2}\\
				&=\frac{1}{4\vartheta_0}\inte\bar{W}_0\left(\frac{m^2}{\sqrt{-\Delta}}\right)W_0 dx+\frac{\vartheta_0^{\frac{Nq}{2}}}{q+2} \|W_0\|_{q+2} ^{q+2}.
			\end{aligned}
		\end{equation*}
		
		On the other hand, considering 
		$\varphi:=\left((-\mu)^{-\frac{2}{2+Nq}}\vartheta\right)^{\frac{N}{2}}W_{0}\left((-\mu)^{-\frac{2}{2+Nq}}\vartheta \cdot\right)$, for $\vartheta>0$ arbitrary, by \eqref{riscalamento},
		one can easily check that
		\begin{equation*}
			\begin{aligned}
				&\limsup_{\mu\to0^-}(-\mu)^{-\frac{2}{2+Nq}}e_{\mu}(a^\ast_v)
				\le
				\limsup_{\mu\to0^-}(-\mu)^{-\frac{2}{2+Nq}} E_\mu(\varphi)
				\\
				&=
				\limsup_{\mu\to0^-}\left[\frac{(-\mu)^{-\frac{4}{2+Nq}}\vartheta}{2}\int_{\mathbb{R}^N}\bar{W}_0
				\left(\sqrt{-\Delta+{\left((-\mu)^{\frac{2}{2+Nq}}\vartheta^{-1}m\right)}^2}+iv\cdot \nabla\right)W_0 dx\right.\\
				&~~~~\left.-\frac{N}{2N+2}(-\mu)^{-\frac{4}{2+Nq}}\vartheta\|W_0\|_{2+\frac2N}^{2+\frac2N}
				+\frac{\vartheta^{\frac{Nq}{2}}}{q+2} \|W_0\|_{q+2}^{q+2}\right]\\
				&= \frac{1}{2}\limsup_{\mu\to0^-}(-\mu)^{-\frac{4}{2+Nq}}\vartheta\int_{\mathbb{R}^N}\bar{W}_0
				\left(\sqrt{-\Delta+{\left((-\mu)^{\frac{2}{2+Nq}}\vartheta^{-1}m\right)}^2}-\sqrt{-\Delta}\right)W_0 dx
				+\frac{\vartheta^{\frac{Nq}{2}}}{q+2} \|W_0\|_{q+2}^{q+2}\\
				&=\frac{1}{2\vartheta}\limsup_{\mu\to0^-}\inte\bar{W}_0\left(\frac{m^2}{\sqrt{-\Delta+
						\left((-\mu)^{\frac{2}{2+Nq}}\vartheta^{-1}m\right)^2}+\sqrt{-\Delta}}\right)W_0 dx+\frac{\vartheta^{\frac{Nq}{2}}}{q+2} \|W_0\|_{q+2}^{q+2}\\
				&\le \frac{1}{4\vartheta}\inte\bar{W}_0\left(\frac{m^2}{\sqrt{-\Delta}}\right)W_0 dx+\frac{\vartheta^{\frac{Nq}{2}}}{q+2} \|W_0\|_{q+2}^{q+2}.
			\end{aligned}
		\end{equation*}
		Combining these last two estimates, we have that for any $\vartheta>0$
		\[
		\frac{1}{4\vartheta_0}\inte\bar{W}_0\left(\frac{m^2}{\sqrt{-\Delta}}\right)W_0 dx+\frac{\vartheta_0^{\frac{Nq}{2}}}{q+2} \|W_0\|_{q+2}^{q+2}\le
		\frac{1}{4\vartheta}\inte\bar{W}_0\left(\frac{m^2}{\sqrt{-\Delta}}\right)W_0 dx+\frac{\vartheta^{\frac{Nq}{2}}}{q+2} \|W_0\|_{q+2}^{q+2}
		\]
		and so, in particular,
		by \eqref{eq5.2} of Lemma \ref{lemA.7} and taking the minimum over $\vartheta>0$, we deduce that
		\begin{equation*}
			\vartheta_0=\left[\frac{(q+2)}{2Nq \|W_0\|_{q+2} ^{q+2}}\inte\bar{W}_0\left(\frac{m^2}{\sqrt{-\Delta}}\right)W_0 dx\right]^{\frac{2}{2+Nq}}.
		\end{equation*}
	\end{proof}
	
	\section{Asymptotic behaviour as \texorpdfstring{$|v|\to 0^+$}{}}\label{se7}

	In this subsection, our aim is to analyse the blow-up behaviour for the travelling solitary waves as $|v|\to0^+$ when $\mu=0$ and $m>0$. To this end, as stated in the introduction, we assume $v=(\beta,0,\ldots,0)\in\R^N$ with $0<\beta =|v|<1$ and rewrite the constrained minimization problem \eqref{eq1.3} as
	\begin{equation}\label{eq1.12}
		e_{\beta}(a):=\inf_{\varphi\in \mathcal{S}_a} E_{\beta}(\varphi ),
	\end{equation}
	where
	\begin{equation*}
		\begin{aligned}
			E_{\beta}(\varphi):=&\frac{1}{2}\int_{\mathbb{R}^N}\bar{{\varphi}}(\sqrt{-\Delta+m^2}+i\beta\partial_{x_1})\varphi dx-\frac{N}{2N+2}\|\varphi\|_{2+\frac2N}^{2+\frac2N},
		\end{aligned}
	\end{equation*}
	and $0<a<a^\ast_\beta:=\|Q_\beta\|_2^2$ with
	$Q_\beta:=Q_{(\beta,0,\ldots,0)}$ (see \eqref{astarv}). 
	Moreover, for simplicity, if $u\in H^{\frac{1}{2}}(\R^N)$,
	$$T_\beta (u):= \int_{\mathbb{R}^N}\bar{u}(\sqrt{-\Delta}+i\beta \partial_{x_1})u dx
	\quad\text{and}\quad
	\mathcal{T}_{m,\beta}(u):=\inte \bar{u}(\sqrt{-\Delta+m^2}+i\beta \partial_{x_1} )udx.$$

	We now derive an additional property of $Q_\beta$ and show that the function $\beta\in (0,1)\mapsto a_\beta^*$ is decreasing and Lipschitz continuous. 
	\begin{lem}\label{lem2.5}
		The following properties hold:
		\begin{enumerate}[label=(\roman*),ref=\roman*]
			\item \label{lem2.5i} For all $0<\beta<1$, we have that
			\begin{equation}\label{eq3.13}
				\inte\bar{Q}_\beta(i\partial_{x_1})Q_\beta dx\le 0.
			\end{equation}
			
			\item \label{lem2.5ii} For each $0<\beta_1<\beta_2<1$, there holds
			\begin{equation}\label{eq3.14}
				a^\ast_{\beta_2}\le a^\ast_{\beta_1}\le a^\ast,
			\end{equation}
			where $a^*$ is defined in \eqref{astar}.
			\item \label{lem2.5iii} For any $0<\epsilon<1$, if $0<\beta_1<\beta_2\le1-\epsilon$, then there exists a constant $M>0$ such that
			\begin{equation}\label{eq3.15}
				0\le  {(a^\ast_{\beta_1})}^{\frac1N}-{(a^\ast_{\beta_2})}^{\frac1N}=\|Q_{\beta_1}\|_2^{\frac{2}{N}}-\|Q_{\beta_2}\|_2^{\frac{2}{N}} \le M(\beta_2-\beta_1).
			\end{equation}
			
		\end{enumerate}
		
	\end{lem}
	
	\begin{proof}
		For functions $u:\R^N\to\mathbb{C}$, let $\tilde{u}(x):=u(-x)$.
		Then
		\begin{equation}\label{eq3.24}
			\|\tilde{Q}_\beta\|_2^2=\|{Q}_\beta\|_2^2=a^\ast_\beta,
			~~
			\|\tilde{Q}_\beta\|_{2+\frac2N}^{2+\frac2N}=\|{Q}_\beta\|_{2+\frac2N}^{2+\frac2N},
		\end{equation}
		and,  since $\widehat{\tilde{Q}_\beta}=\widetilde{\hat{Q}_\beta}$,
		\begin{equation}\label{eq3.19}
			\|\tilde{Q}_\beta\|_{\dot{H}^{\frac{1}{2}}(\mathbb{R}^N)}=\|Q_\beta\|_{\dot{H}^{\frac{1}{2}}(\mathbb{R}^N)},
		\end{equation}
		and
		\begin{equation}\label{eq3.20}
			\int_{\mathbb{R}^N}\bar{\tilde{Q}}_\beta(i\beta \partial_{x_1})\tilde{Q}_\beta dx
			=-\int_{\mathbb{R}^N}\bar{{Q}}_\beta(i\beta \partial_{x_1}){Q}_\beta dx.
		\end{equation}
		Moreover, by \eqref{eq1.9} in Proposition \ref{prop1.1}, for any $\beta\in[0,1)$ we have the following Pohozaev-type identity
		\begin{equation}\label{eq3.18}
			T_\beta(Q_\beta)
			=\frac{N}{N+1} \|Q_\beta\|_{2+\frac2N}^{2+\frac2N}=N\|Q_\beta\|_2^2,
		\end{equation}
		which yields that
		\begin{equation}\label{eq3.22}
			\mathcal{E}_{\beta}(Q_\beta)
			=0
			\quad\text{where}\quad
			\mathcal{E}_{\beta}(u)
			:=\frac{1}{2}T_\beta(u)
			-\frac{N}{2N+2}\|u\|_{2+\frac2N}^{2+\frac2N}.
		\end{equation}
		
		Let us prove \eqref{lem2.5i}. Assume by contradiction that
		\begin{equation}\label{eq3.25}
			\inte\bar{Q}_\beta(i\partial_{x_1})Q_\beta>0.
		\end{equation}
		Then, by \eqref{eq3.24}, \eqref{eq3.19}, \eqref{eq3.20}, \eqref{eq3.22} and \eqref{eq3.25}, we obtain that
		\begin{equation}\label{eq3.23}
			\begin{aligned}
				\mathcal{E}_{\beta}(\tilde{Q}_\beta)
				&=\frac{1}{2}\int_{\mathbb{R}^N}\bar{Q}_\beta(\sqrt{-\Delta}-i\beta\partial_{x_1}){Q}_\beta dx-\frac{N}{2N+2}\|{Q}_\beta\|_{2+\frac2N}^{2+\frac2N}\\
				&<\frac{1}{2}\int_{\mathbb{R}^N}\bar{Q}_\beta(\sqrt{-\Delta}+i\beta\partial_{x_1}){Q}_\beta dx-\frac{N}{2N+2}\|{Q}_\beta\|_{2+\frac2N}^{2+\frac2N}=\mathcal{E}_{\beta}(Q_\beta)=0.
			\end{aligned}
		\end{equation}
		On the other hand, we infer from \eqref{eq1.5} and \eqref{eq3.24} that $\mathcal{E}_{\beta}(\tilde{Q}_\beta)\ge0$,
		which leads to a contradiction.
		
		Now let us show \eqref{lem2.5ii}.
		If $0<\beta_1<\beta_2<1$, using \eqref{eq3.13} and \eqref{eq3.22}, we first deduce that
		\begin{equation}\label{eq3.21}
			\mathcal{E}_{\beta_2}(Q_{\beta_1})
			=\mathcal{E}_{\beta_1}(Q_{\beta_1})+\frac{1}{2}\int_{\mathbb{R}^N}\bar{Q}_{\beta_1}(i\beta_2\partial_{x_1}-i\beta_1\partial_{x_1})Q_{\beta_1}dx=\frac{\beta_2-\beta_1}{2}\int_{\mathbb{R}^N}\bar{Q}_{\beta_1}(i\partial_{x_1})Q_{\beta_1}dx\le0.
		\end{equation}
		Furthermore, applying \eqref{eq1.5}, one can see that
		\begin{equation*}
			\mathcal{E}_{\beta_2}(Q_{\beta_1})
			\ge\frac{1}{2}\left[1-\frac{{(a^\ast_{\beta_1})}^{\frac1N}}{{(a^\ast_{\beta_2})}^{\frac1N}}\right]T_{\beta_2}(Q_{\beta_1}),
		\end{equation*}
		which, combined with \eqref{eq10280938} and \eqref{eq3.21}, implies that $a^\ast_{\beta_1}\ge a^\ast_{\beta_2}$ and so, using also \eqref{eq1.16} we can conclude.
		
		Finally let us prove \eqref{lem2.5iii}. Let $\epsilon>0$ be such that $0<\beta_1<\beta_2\le1-\epsilon<1$. By \eqref{eq10280938} and \eqref{eq1.5}  we obtain
		\begin{equation}\label{eq3.27}
			\begin{aligned}
				\mathcal{E}_{\beta_1}(Q_{\beta_2})
				&\ge\frac{1}{2}
				\left[1-\frac{({a^\ast_{\beta_2})}^{\frac1N}}{{(a^\ast_{\beta_1})}^{\frac1N}}\right]T_{\beta_1}(Q_{\beta_2})
				\ge\frac{1-\beta_1}{2}
				\left[\frac{{(a^\ast_{\beta_1})}^{\frac1N}-{(a^\ast_{\beta_2})}^{\frac1N}}{{(a^\ast_{\beta_1})}^{\frac1N}}\right]\|{Q}_{\beta_2}\|_{\dot{H}^{\frac{1}{2}}(\R^N)}^2\\
				&\ge\frac{\epsilon}{2}
				\left[\frac{{(a^\ast_{\beta_1})}^{\frac1N}-{(a^\ast_{\beta_2})}^{\frac1N}}{{(a^\ast_{\beta_1})}^{\frac1N}}\right]\|{Q}_{\beta_2}\|_{\dot{H}^{\frac{1}{2}}(\R^N)}^2.
			\end{aligned}
		\end{equation}
		On the other hand, applying \eqref{eq10280938} and \eqref{eq3.22} again, we deduce that
		\begin{equation}\label{eq3.28}
			\begin{aligned}
				\mathcal{E}_{\beta_1}(Q_{\beta_2})
				&=\mathcal{E}_{\beta_2}(Q_{\beta_2})+
				\frac{\beta_1-\beta_2}{2}\int_{\mathbb{R}^N}\bar{Q}_{\beta_2}(i\partial_{x_1})Q_{\beta_2}dx
				\le\frac{\beta_2-\beta_1}{2}
				\|{Q}_{\beta_2}\|_{\dot{H}^{\frac{1}{2}}(\R^N)}^2.
			\end{aligned}
		\end{equation}
		Thus, combining \eqref{eq3.27} and \eqref{eq3.28}, and using (\ref{lem2.5ii}), we get
		$$0\le {(a^\ast_{\beta_1})}^{\frac1N}-{(a^\ast_{\beta_2})}^{\frac1N} \le\frac{{(a^\ast_{\beta_1})}^{\frac1N}}{\epsilon}(\beta_2-\beta_1)\le \frac{{(a^\ast)}^{\frac1N}}{\epsilon}(\beta_2-\beta_1),$$
		concluding the proof.
	\end{proof}
	
	\begin{lem}\label{convforte}
		$Q_\beta \to Q$ in $H^{\frac12}(\R^N)$, as $\beta \to 0^+$.
	\end{lem}
	
	\begin{proof}
		Observe that, using \eqref{eq3.18} and \eqref{eq1.16}, $Q_\beta$ is bounded in $H^{\frac12}(\R^N)$.
		By \eqref{eq1.9}, we deduce that $Q_\beta$ does not vanish as $\beta\to 0^+$, namely there exists $\{y_\beta\}$ in $\R^N$ 
		\begin{equation*}
			\lim\limits_{\beta\to 0^+}\int_{B_{1}(y_{\beta})}|Q_\beta|^2dx>0.
		\end{equation*}
		With an abuse of notation, we still denote by $Q_\beta$ its translation $Q_\beta(\cdot+y_\beta)$. So there exists $u_0\in H^\frac 12(\R^N)\setminus\{0\}$ such that, up to a subsequence, 
		\begin{equation}\label{wcQbeta}
			Q_{\beta}\rightharpoonup u_0 \text{~weakly~in} ~ H^{\frac{1}{2}}(\mathbb{R}^N), \,\, \text{as}\,\,\beta\to 0^+.
		\end{equation}
		Therefore
		\begin{equation*}
			\sqrt{-\Delta}u_0+u_0=|u_0|^{\frac{2}{N}} u_0, \qquad \text{in }\R^N.
		\end{equation*}
		
		By \eqref{astar}, \eqref{eq1.10}, \eqref{wcQbeta}, \eqref{eq3.114} and the fact that $Q$ is the unique ground state of \eqref{eqv0}, one has 
		\begin{equation*}
			\begin{aligned}
				a^\ast=\|Q\|_2^2\le\|{u}_0\|_2^2
				\le\lim_{\beta\to 0^+}\|{Q}_\beta\|_2^2=\lim_{\beta\to 0^+}a_\beta^*=a^\ast,
			\end{aligned}
		\end{equation*}
		which, combined with \eqref{eq1.10} and the uniqueness of ground state again, implies that 
		$u_0=Q$, up to a translation. 
		Then, by \eqref{eq3.18} we have that
		\begin{equation}\label{prima}
			T_\beta(Q_\beta) = N \|Q_\beta\|_2^2\to N \|Q\|_2^2=\|Q\|_{\dot{H}^\frac{1}{2}(\R^N)}^2
		\end{equation}
		and, by \eqref{eq10280938} and using the boundedness of $\{Q_\beta\}$,
		\begin{equation}\label{seconda}
			|T_\beta(Q_\beta)-\|Q_\beta\|_{\dot{H}^\frac{1}{2}(\R^N)}^2|
			\leq \beta \|Q_\beta\|_{\dot{H}^\frac{1}{2}(\R^N)}^2 \to 0
			\quad
			\text{as }\beta\to 0^+.
		\end{equation}
		Then, combining \eqref{prima} and \eqref{seconda}, we get
		\[
		\|Q_\beta\|_{\dot{H}^\frac{1}{2}(\R^N)}^2
		= \|Q_\beta\|_{\dot{H}^\frac{1}{2}(\R^N)}^2 -T_\beta(Q_\beta) + T_\beta(Q_\beta) \to \|Q\|_{\dot{H}^\frac{1}{2}(\R^N)}^2
		\]
		and we can conclude.
	\end{proof}

	Actually the following uniform control on Lebesgue norms for $Q_\beta$ holds.
	
	\begin{lem}\label{limitato}
		$Q_\beta$ are uniformly bounded in $L^{\tau}(\R^N)$, for $\tau \in[{\frac{2N}{N+2}},\frac{2N}{N-1}]$ and $\beta$ small enough.
	\end{lem}
	\begin{proof}
		By Lemma \ref{convforte}, it is sufficient to show the boundedness in $L^\frac{2N}{N+2}(\R^N)$.
		Using the notation of \cite[Lemma A.4]{BGLV19}
		\[
		Q_\beta=G_\beta*(|Q_\beta|^{2/N} Q_\beta)
		\]
		where $G_\beta$ is given by (A.6) therein. \\
		Arguing as in the proof of \cite[Lemma A.4]{BGLV19}, we obtain that, if $\beta$ is sufficiently small, e.g. $1-\beta^2\geq1/2$, 
		\[
		|G_\beta(x)|\leq C \int_0^\infty \frac{t}{((1-\beta^2)t^2+|x|^2)^\frac{N+1}{2}}dt
		\leq C\int_0^\infty \frac{t}{((t^2/2+|x|^2)^\frac{N+1}{2}}dt
		\leq \frac{C}{|x|^{N-1}}, \quad \text{for }x\neq 0,
		\]
		and 
		\[
		|G_\beta(x)|
		\leq \frac{C}{|x|^{N+1}},  \quad \text{for }|x|\geq 1,
		\]
		where $C$ does not depend on $\beta$. 
		Therefore $G_\beta$ is uniformly bounded in $L^1(\R^N)$.
		Then, by the Young inequality and Lemma \ref{convforte}, for $\beta$ sufficiently small, we get
		\[
		\|Q_\beta\|_{\frac{2N}{N+2}}\le \|G_\beta\|_{1}\|Q_\beta\|_{2}^{\frac{2}{N}+1}\le C.
		\]
	\end{proof}

	Based on the previous results, we now prove Theorem \ref{Thm4}, namely, the asymptotic behaviour of minimizers for problem \eqref{eq1.12} as $\beta\to0^+$.
	
	\begin{proof}[Proof of Theorem \ref{Thm4}]
		For $\beta\in(0,1)$, let
		\begin{equation}\label{eq3.60}
			a_\beta:=(1-\beta)^Na^\ast_\beta.
		\end{equation}
		Since $0<a_\beta<a^\ast_\beta$, then Theorem \ref{Thm1}-\eqref{E3}-\eqref{E3i} guarantees the existence of minimizers of $e_\beta(a_\beta)$.
		Thus, if $u_{\beta}$ is a minimizer of $e_{\beta}(a_\beta)$,
		then it satisfies the Euler-Lagrange equation
		\begin{equation}\label{eq3.29}
			(\sqrt{-\Delta+m^2}+i\beta\partial_{x_1})u_{\beta}-|u_{\beta}|^{\frac{2}{N}}u_{\beta}=\lambda_\beta u_{\beta} ,
		\end{equation}
		where $\lambda_\beta$ is a suitable Lagrange multiplier.
		Moreover, we also have that
		\begin{equation}\label{eq3.30}
			\lambda_\beta a_\beta=2e_{\beta}(a_\beta)-\frac{1}{N+1}\|u_{\beta}\|_{2+\frac2N}^{2+\frac2N}.
		\end{equation}
		Our proof proceeds in four steps.
		
		{\bf Step 1: Estimates of $e_{\beta}(a_{\beta})$.} As in \eqref{eq3.12}--\eqref{eq3.101},
		and using \eqref{eq1.16}, the lower bound is established, and so,  as $\beta\to 0^+$,
		\begin{equation}\label{eq3.102}
			e_{\beta}(a_{\beta})\ge \frac{m(1-\beta)^N a_\beta^*}{2} \sqrt{\beta(1-\beta)}\sqrt{2+\beta^2-\beta}\gtrsim\beta^{\frac12}.
		\end{equation}

		On the other hand, let we set
		$$\phi_\tau:=\tau^{\frac{N}{2}}{\sqrt{\frac{a_\beta}{a^\ast_\beta}}}Q_\beta(\tau \cdot), ~\tau>0.$$
		Applying \eqref{eq3.18}, similarly in \eqref{eq2.3}--\eqref{eq3.10}, we conclude that
		\begin{equation}\label{eq3.32}
			\begin{aligned}
				e_{\beta}(a_\beta)
				&\le\frac{a_\beta m^2}{4a^\ast_\beta\tau}\int_{\mathbb{R}^N}\frac{|\hat{Q}_\beta(k)|^2}{|k|}dk+\frac{a_\beta N\tau}{2}\left[1-\left(\frac{a_\beta}{a^\ast_\beta}\right)^{\frac{1}{N}} \right]\\
				&=(1-\beta)^N\left(\frac{m^2}{4\tau}\int_{\mathbb{R}^N}\frac{|\hat{Q}_\beta(k)|^2}{|k|}dk+\frac{a^\ast_\beta N\beta\tau}{2} \right)
				\le\frac{m^2}{4\tau}\int_{\mathbb{R}^N}\frac{|\hat{Q}_\beta(k)|^2}{|k|}dk+\frac{a^\ast_\beta N\beta\tau}{2}.
			\end{aligned}
		\end{equation}
		Moreover, by \cite[\S 5.10, formula (2)]{LL01}, namely,
		\begin{equation}\label{eq3.112}
			c_{2\tilde{\alpha}}\inte|k|^{-2\tilde{\alpha}}|\hat{f}(k)|^2dk=c_{N-2\tilde{\alpha}}\inte\inte\bar{f}(x)f(y)|x-y|^{2\tilde{\alpha}-N}dxdy,
		\end{equation}
		where $c_{\tilde{\alpha}}:=\pi^{-\frac{\tilde{\alpha}}{2}}\Gamma\left(\frac{\tilde{\alpha}}{2}\right)$ with $0<\tilde{\alpha}<\frac{N}{2}$, and   Hardy-Littlewood-Sobolev inequality, we deduce from   
		\eqref{eq3.14} and Lemma \ref{limitato}, that
		\begin{equation*}
			\int_{\mathbb{R}^N}\frac{|\hat{Q}_\beta(k)|^2}{|k|}dk\le C\inte\inte\bar{Q}_\beta(x)Q_\beta(y)|x-y|^{1-N}dxdy
			\le C\|Q_\beta\|_{\frac{2N}{N+2}}\|Q_\beta\|_{2}\le Ca^\ast,
		\end{equation*}
		with $C$ independent of $\beta$.
		Then, by taking the minimum over $\tau>0$ in \eqref{eq3.32}, it follows from 
		\eqref{eq5.2} in Lemma \ref{lemA.7} that, as $\beta\to 0^+$,
		\begin{equation}\label{eq3.110}
			e_{\beta}(a_\beta)\le \left(m^2\int_{\mathbb{R}^N}\frac{|\hat{Q}_\beta(k)|^2}{|k|}dk\right)^{\frac12}\left(\frac{a^\ast_\beta N\beta}{2}\right)^{\frac12}\lesssim\beta^{\frac12},
		\end{equation}
		which, together with \eqref{eq3.102}, indicates that
		\begin{equation}\label{eq3.103}
			e_{\beta}(a_\beta)\sim\beta^{\frac12}, \text{ as } \beta\to 0^+.
		\end{equation}
		
		{\bf Step 2: Estimates of $\|u_{\beta}\|_{\dot{H}^{\frac{1}{2}}(\R^N)}^2$ and $\|u_\beta\|_{2+\frac2N}^{2+\frac2N}$.}
		Since there exists $C>0$ such that, for every $k=(k_1,\ldots,k_N)\in\R^N$, $|k|-\beta k_1 \leq C |k|$, then
		$$\sqrt{-\Delta}+i\beta\partial_{x_1}\le C\sqrt{-\Delta}$$
		and so, using \eqref{eq1.5},
		\begin{equation}\label{eq3.105}
			\|u_\beta\|_{2+\frac2N}^{2+\frac{2}{N}}
			\le \frac{N+1}{N}\left(\frac{a_\beta}{{a^\ast_\beta}}\right)^{\frac{1}{N}}T_\beta(u_\beta)
			\le \frac{N+1}{N}(1-\beta) C \|u_{\beta}\|_{\dot{H}^{\frac{1}{2}}(\R^N)}^2.
		\end{equation}
		
		Firstly, by applying \eqref{eq10280938},  \eqref{eq1.5},  and \eqref{eq2.5}, a direct calculation shows that
		\begin{equation*}
			\begin{aligned}
				e_{\beta}(a_\beta)
				\ge\frac{1}{2}\left[1-\left(\frac{a_\beta}{{a^\ast_\beta}}\right)^{\frac{1}{N}}\right]T_\beta(u_\beta)
				\ge \frac{\beta(1-\beta)}{2}\|u_{\beta}\|_{\dot{H}^{\frac{1}{2}}(\R^N)}^2,
			\end{aligned}
		\end{equation*}
		which, together with \eqref{eq3.103}, yields that
		\begin{equation}\label{eq3.104}
			\|u_{\beta}\|_{\dot{H}^{\frac{1}{2}}(\R^N)}^2\lesssim\beta^{-\frac12}, \text{ as } \beta\to 0^+.
		\end{equation}
		
		Now, let us focus on the upper bound of $\|u_\beta\|_{2+\frac2N}^{2+\frac2N}$. To do this, we claim that, for every $\beta\in(0,1)$,
		\begin{equation}\label{eq3.33}
			\int_{\mathbb{R}^N}\bar{u}_\beta(i\beta\partial_{x_1})u_\beta dx\le 0.
		\end{equation}
		Indeed, if there exists $\beta\in(0,1)$ such that \eqref{eq3.33} is false, then, 
		defining as before the reflected function $\tilde{u}_\beta(x):=u_\beta(-x)$, by a similar argument as in
		\eqref{eq3.23}, we deduce that
		$$E_{\beta}(\tilde{u}_\beta)<E_{\beta}(u_\beta)=e_{\beta}(a_\beta),$$
		which leads to a contradiction since $E_{\beta}(\tilde{u}_\beta)\ge e_{\beta}(a_\beta)$.
		
		Observe that we consider the behaviour as $\beta\to 0^+$, thus one may assume that $0<\beta<\beta_0<1/2$. Then, from \eqref{eq3.14}, it follows  that
		\begin{align*}
			a_\beta-a_{\beta_0}
			&=(1-\beta)^Na^\ast_\beta-(1-\beta_0)^Na^\ast_{\beta_0}>(1-\beta_0)^N(a^\ast_\beta-a^\ast_{\beta_0})\ge0.
		\end{align*}
		Similarly, we also have that
		$a_{\beta_0}-a_{\frac12}\ge0$.\\
		Therefore, there holds
		\begin{equation}\label{eq3.34}
			a_\beta\ge a_{\beta_0}\ge a_{\frac12}.
		\end{equation}
		Moreover, by \eqref{eq3.14} and \eqref{eq3.15}, one can see that there exists $\tilde{M}>0$ such that
		\begin{equation}\label{eq3.35}
			\begin{aligned}
				(a_\beta)^{\frac1N}-(a_{\beta_0})^{\frac1N}
				&=(1-\beta)(a^\ast_\beta)^{\frac1N}-(1-\beta_0)(a^\ast_{\beta_0})^{\frac1N}\\
				&=(1-\beta)\left[(a^\ast_\beta)^{\frac1N}-(a^\ast_{\beta_0})^{\frac1N}\right]+(a^\ast_{\beta_0})^{\frac1N}(\beta_0-\beta)\\
				&\le(1-\beta)M(\beta_0-\beta)+(a^\ast)^{\frac1N}(\beta_0-\beta)
				\le\tilde{M}(\beta_0-\beta).
			\end{aligned}
		\end{equation}
		Then, from \eqref{eq3.33}, \eqref{eq3.34}, and \eqref{eq3.35}, we deduce  that
		\begin{equation*}
			\begin{aligned}
				e_{\beta_0}(a_{\beta_0})&\le E_{\beta_0}\left(\sqrt{\frac{a_{\beta_0}}{a_\beta}}u_\beta\right)
				=\frac{1}{2}\frac{a_{\beta_0}}{a_\beta}\mathcal{T}_{m,\beta_0}(u_\beta)-\left(\frac{a_{\beta_0}}{a_\beta}\right)^{1+\frac1N}\frac{N}{2N+2}\|u_\beta\|_{2+\frac2N}^{2+\frac2N}\\
				&=\frac{a_{\beta_0}}{a_\beta}\left\{e_{\beta}(a_{\beta})+\frac{\beta_0-\beta}{2}\int_{\mathbb{R}^N}\bar{u}_\beta(i\partial_{x_1})u_\beta dx+\left[1-\left(\frac{a_{\beta_0}}{a_\beta}\right)^{\frac1N}\right]\frac{N}{2N+2}\|u_\beta\|_{2+\frac2N}^{2+\frac2N}\right\}\\
				&\le\frac{a_{\beta_0}}{a_\beta}\left\{e_{\beta}(a_{\beta})+\frac{(a_\beta)^{\frac1N}-(a_{\beta_0})^{\frac1N}}{(a_\beta)^{\frac1N}}
				\frac{N}{2N+2}\|u_\beta\|_{2+\frac2N}^{2+\frac2N}\right\}\\
				&\le e_{\beta}(a_{\beta})+\frac{\tilde{M}(\beta_0-\beta)}{\left(a_\frac12\right)^{\frac1N}}
				\frac{N}{2N+2}\|u_\beta\|_{2+\frac2N}^{2+\frac2N}.
			\end{aligned}
		\end{equation*}
		Then, if $\tilde{\gamma}>0$ is large enough\footnote{The constant $\tilde{\gamma}>0$ satisfies $C_1{(1+\tilde{\gamma})}^{\frac12}-C_2>0$ where $C_1,C_2>0$ are the constants that appear considering \eqref{eq3.103}.},
		for $\beta_0:=(1+\tilde{\gamma})\beta$, applying \eqref{eq3.103}, we conclude that, for $\beta$ sufficiently small,
		\begin{equation*}
			\|u_\beta\|_{2+\frac2N}^{2+\frac2N}
			\ge\frac{2N+2}{N}\frac{\left(a_\frac12\right)^{\frac1N}}{\tilde{M}(\beta_0-\beta)}[e_{\beta_0}(a_{\beta_0})-e_{\beta}(a_{\beta})]
			\ge\frac{2N+2}{N}\frac{\left(a_\frac12\right)^{\frac1N}}{\tilde{M}}\frac{C_1{(1+\tilde{\gamma})}^{\frac12}{\beta}^{\frac12}-C_2{\beta}^{\frac12}}{\tilde{\gamma}\beta},
		\end{equation*}
		namely
		\begin{equation}\label{eq3.37}
			\|u_\beta\|_{2+\frac2N}^{2+\frac2N}\gtrsim\beta^{-\frac12}, \text{ as } \beta\to 0^+.
		\end{equation}
		Therefore, we derive from \eqref{eq3.105}, \eqref{eq3.104} and \eqref{eq3.37} that
		\begin{equation}\label{eq3.106}
			\|u_{\beta}\|_{\dot{H}^{\frac{1}{2}}(\R^N)}^2 \sim \|u_\beta\|_{2+\frac2N}^{2+\frac2N} \sim\beta^{-\frac12}, \text{ as } \beta\to 0^+,
		\end{equation}
		proving \eqref{qualunquecazzata}.
		
		{\bf Step 3: Blow-up analysis.} Let 
		$\tilde{\varphi}_\beta:=\beta^{\frac{N}{4}}u_\beta(\beta^{\frac12}\cdot)$. From \eqref{eq3.106} it follows  that
		\begin{equation*}
			\|\tilde{\varphi}_\beta \|_{\dot{H}^{\frac{1}{2}}(\R^N)}^2\sim \|\tilde{\varphi}_\beta\|_{2+\frac2N}^{2+\frac2N}\sim 1, \text{ as } \beta\to 0^+,
		\end{equation*}
		which, combined with the boundedness of $\{a_\beta\}$ by \eqref{eq1.16} and \eqref{eq3.60}, indicates that $\{\tilde{{\varphi}}_\beta\}$ is bounded in $H^{\frac{1}{2}}(\mathbb{R}^N)$. In the following, for simplicity, we write $\beta\to 0^+$ in place of $\beta_n\to 0^+$, eventually up to a subsequence. By
		Lemma \ref{lem1.3}, there exists a sequence $\{y_{\beta}\}\subset\mathbb{R}^N$ such that
		\begin{equation}\label{eq3.40}
			\lim\limits_{\beta\to 0^+}\int_{B_{1}(y_{\beta})}|\tilde{\varphi}_\beta|^2dx>0.
		\end{equation}
		Let
		\begin{equation}\label{eq3.41}
			\varphi_{\beta}:=\tilde{\varphi}_\beta(\cdot+y_{\beta})=\beta^{\frac{N}{4}}u_{\beta}({\beta}^{\frac12}(\cdot+y_{\beta})).
		\end{equation}
		Passing to a subsequence, and using \eqref{eq3.40}, there exists some $0\not\equiv\varphi_0\in H^{\frac{1}{2}}(\mathbb{R}^N)$ such that
		\begin{equation}\label{wcphi0}
			\varphi_{\beta}\rightharpoonup\varphi_0 \text{~weakly~in} ~ H^{\frac{1}{2}}(\mathbb{R}^N), \,\, \text{as}\,\,\beta\to 0^+.
		\end{equation}
		
		On the other hand, by \eqref{eq3.29} and \eqref{eq3.41}, we notice that $\varphi_{\beta}$ satisfies
		\begin{equation}\label{eq3.44}
			(\sqrt{-\Delta+\beta m^2}+i\beta\partial_{x_1})\varphi_{\beta}-|\varphi_{\beta}|^{\frac{2}{N}}\varphi_{\beta}=\beta^{\frac12}\lambda_\beta \varphi_{\beta}.
		\end{equation}
		Furthermore, using \eqref{eq1.16},  \eqref{eq3.60},
		\eqref{eq3.30}, \eqref{eq3.103}, and \eqref{eq3.106}, 
		up to a subsequence, it follows that there exists some constant $\eta>0$ such that
		\begin{equation}\label{eq3.45}
			\beta^{\frac12}\lambda_\beta \to -\eta,  \text{ as}\,\,\beta\to 0^+.
		\end{equation}
		In view of \eqref{eq3.44}, \eqref{eq3.45} and the fact that  $\varphi_{\beta}\rightharpoonup\varphi_0$
		weakly in $ H^{\frac{1}{2}}(\mathbb{R}^N)$, we thus have that $\varphi_0$ satisfies
		\begin{equation*}
			(\sqrt{-\Delta})\varphi_0-|\varphi_0|^{\frac{2}{N}} \varphi_0 =-\eta\varphi_0.
		\end{equation*}
		Set
		\begin{equation}\label{eq3.59}
			{Q}_0:=\eta^{-\frac{N}{2}}\varphi_0\left(\eta^{-1}\cdot\right).
		\end{equation}
		The function ${Q}_0$ satisfies equation \eqref{eqv0} and, by \eqref{eq1.9},
		\begin{equation}\label{eq3.109}
			\|Q_0\|_{\dot{H}^{\frac{1}{2}}(\R^N)}^2=\frac{N}{N+1} \|Q_0\|_{2+\frac2N}^{2+\frac2N}=N\|Q_0\|_2^2.
		\end{equation}
		Moreover,
		we infer from \eqref{astar}, \eqref{eq3.16}, \eqref{eq3.59}, \eqref{wcphi0}, \eqref{eq3.41}, \eqref{eq3.114}, and \eqref{eq3.60},
		that
		\begin{equation}\label{eq3.47}
			\begin{aligned}
				a^\ast=\|Q\|_2^2\le\|{Q}_0\|_2^2
				=\|{\varphi}_0\|_2^2
				\le\lim_{\beta\to 0^+}\|{\varphi}_\beta\|_2^2=\lim_{\beta\to 0^+}a_\beta=a^\ast,
			\end{aligned}
		\end{equation}
		which, by \eqref{eq1.10} and the uniqueness of the ground state for \eqref{eqv0}, implies that 
		$Q_0=Q$, up to a translation. Then,
		using the interpolation inequality and Sobolev embedding theorem,
		\begin{equation}\label{eq3.51}
			\varphi_{\beta}\rightarrow\varphi_0 \text{~in~} L^\kappa(\mathbb{R}^N) \text{ as } \beta\to 0^+ , \quad \text{for all } 2\le \kappa<{2N}/{(N-1)}.
		\end{equation}
		Hence, by \eqref{eq10280938}, \eqref{riscalamento}, and \eqref{eq3.41} again, we get that
		\begin{equation*}
			e_{\beta}(a_\beta)
			=\frac{1}{2}\beta^{-\frac12}\mathcal{T}_{\beta^\frac{1}{2}m,\beta}(\varphi_\beta)
			-\frac{N}{2N+2}\beta^{-\frac12}\|\varphi_\beta\|_{2+\frac2N}^{2+\frac2N}\ge\frac{1}{2}\beta^{-\frac12}(1-\beta)\|\varphi_\beta\|_{\dot{H}^{\frac{1}{2}}(\R^N)}^2-\frac{N}{2N+2}\beta^{-\frac12}\|\varphi_\beta\|_{2+\frac2N}^{2+\frac2N},
		\end{equation*}
		which, combined with \eqref{eq3.103}, \eqref{eq1.5}, \eqref{eq10280938}, \eqref{eq3.60}, \eqref{eq3.41}, \eqref{astarv}, allows us to get, as $\beta\to 0^+$,
		\begin{equation}\label{eq3.52}
			\begin{aligned}
				0\leftarrow \frac{\beta}{2}\|\varphi_\beta\|_{\dot{H}^{\frac{1}{2}}(\R^N)}^2+\beta^{\frac12}e_{\beta}(a_\beta)
				&\ge
				\frac{1}{2}\|\varphi_\beta\|_{\dot{H}^{\frac{1}{2}}(\R^N)}^2-\frac{N}{2N+2}\|\varphi_\beta\|_{2+\frac2N}^{2+\frac2N}\\
				&\ge
				\frac{1}{2}\|\varphi_\beta\|_{\dot{H}^{\frac{1}{2}}(\R^N)}^2
				-\frac{1}{2} T_{\beta}(\varphi_\beta) \left(\frac{\|\varphi_\beta\|_2^2}{\|Q_\beta\|_2^2}\right)^{\frac{1}{N}}
				\\
				&\ge \frac{1}{2}\|\varphi_\beta\|_{\dot{H}^{\frac{1}{2}}(\R^N)}^2
				\left[1-(1-\beta^2)
				\right]
				\ge 0. 
			\end{aligned}
		\end{equation}
		Furthermore, by 
		\eqref{eq3.59}, \eqref{eq3.109}, and \eqref{eq3.51}, we deduce that
		\begin{equation*}
			\|\varphi_0\|_{\dot{H}^{\frac{1}{2}}(\R^N)}^2
			=\eta \|Q_0\|_{\dot{H}^{\frac{1}{2}}(\R^N)}^2
			=\frac{N\eta}{N+1} \|Q_0\|_{2+\frac2N}^{2+\frac2N}
			=\frac{N}{N+1} \|\varphi_0\|_{2+\frac2N}^{2+\frac2N}=\frac{N}{N+1}\lim_{\beta\to 0^+}\|\varphi_\beta\|_{2+\frac2N}^{2+\frac2N},
		\end{equation*}
		which, together with \eqref{eq3.52}, yields that
		$\|\varphi_\beta\|_{\dot{H}^{\frac{1}{2}}(\R^N)}^2 \to \|\varphi_0\|_{\dot{H}^{\frac{1}{2}}(\R^N)}^2\text{ as}\,\,\beta\to 0^+.$
		Thus we have that, as $\beta\to 0^+$,
		\begin{equation}\label{eq3.107}
			\beta^{\frac{N}{4}}u_{\beta}({\beta}^{\frac12}(\cdot+y_{\beta}))=\varphi_\beta\rightarrow \varphi_0=\eta^{\frac{N}{2}}{Q}(\eta \cdot) \quad\text{~in} ~ H^{\frac{1}{2}}(\mathbb{R}^N).
		\end{equation}
		
		{\bf Step 4: The exact value on $\eta$.}
		Firstly, by \eqref{eq3.112}, Hardy-Littlewood-Sobolev inequality, Lemmas \ref{convforte} and  \ref{limitato}, we have that
		\begin{equation}\label{eq3.113}
			\begin{aligned}
				&\left|\int_{\mathbb{R}^N}\frac{|\hat{Q}_\beta(k)|^2}{|k|}dk-\int_{\mathbb{R}^N}\frac{|\hat{Q}(k)|^2}{|k|}dk\right|\\
				&\le C\left|\inte\inte\frac{\bar{Q}_\beta(x)(Q_\beta(y)-Q(y))}{|x-y|^{N-1}}dxdy\right|+C\left|\inte\inte\frac{(\bar{Q}_\beta(x)-\bar{Q}(x))Q(y)}{|x-y|^{N-1}}dxdy\right|\\
				&\le C\|Q_\beta\|_{\frac{2N}{N+2}}\|Q_\beta-Q\|_{2}+C\|Q\|_{\frac{2N}{N+2}}\|Q_\beta-Q\|_{2}\to 0, \qquad\text{as }\beta\to 0^+.
			\end{aligned}
		\end{equation}
		Therefore, it follows from \eqref{eq3.110} and \eqref{eq3.113} that
		\begin{equation}\label{limsup17}
			\begin{aligned}
				\limsup_{\beta\to 0^+}\beta^{-\frac{1}{2}}e_{\beta}(a_\beta)
				&\le \limsup_{\beta\to 0^+}\left(\frac{a^\ast_\beta N m^2}{2}\int_{\mathbb{R}^N}\frac{|\hat{Q}_\beta(k)|^2}{|k|}dk\right)^{\frac12}
				=\left(\frac{Na^\ast m^2}{2}\int_{\mathbb{R}^N}\frac{|\hat{Q}(k)|^2}{|k|}dk\right)^{\frac12}\\
				&=\left[\frac{Na^\ast m^2}{2}\int_{\mathbb{R}^N}\bar{{Q}}\left(\frac{1}{\sqrt{-\Delta}}\right){Q} dx\right]^{\frac12}.
			\end{aligned}
		\end{equation}

		On the other hand, from \eqref{eq1.5} and \eqref{eq3.60} we derive that
		\begin{equation*}
			\begin{aligned}
				e_{\beta}(a_\beta)
				&=\frac{1}{2}\beta^{-\frac12}\int_{\mathbb{R}^N}\bar{\varphi}_\beta(\sqrt{-\Delta+\beta m^2}-\sqrt{-\Delta})\varphi_\beta dx+\frac{1}{2}\beta^{-\frac12}T_\beta(\varphi_\beta)-\frac{N}{2N+2}\beta^{-\frac12}\|\varphi_\beta\|_{2+\frac2N}^{2+\frac2N}\\
				&\ge
				\frac{1}{2}\beta^{-\frac12}\int_{\mathbb{R}^N}\bar{\varphi}_\beta\left(\frac{\beta m^2}{\sqrt{-\Delta+\beta m^2}+\sqrt{-\Delta}}\right)\varphi_\beta dx+\frac{N}{2N+2}\beta^{-\frac12}\left[\left(\frac{a_\beta^\ast}{a_\beta}\right)^{\frac1N}-1\right]\|\varphi_\beta\|_{2+\frac2N}^{2+\frac2N}\\
				&=\frac{1}{2}\beta^{\frac12}\int_{\mathbb{R}^N}\bar{\varphi}_\beta\left(\frac{ m^2}{\sqrt{-\Delta+\beta m^2}+\sqrt{-\Delta}}\right)\varphi_\beta dx+\frac{N}{2N+2}\frac{\beta^{\frac12}}{1-\beta}\|\varphi_\beta\|_{2+\frac2N}^{2+\frac2N},
			\end{aligned}
		\end{equation*}
		which, together with \eqref{eq3.59}, \eqref{eq3.109}, \eqref{eq3.47},  \eqref{eq3.107}, Fatou's lemma, and Lemma \ref{lemA.7}, implies that
		\begin{equation*}
			\begin{aligned}
				\liminf_{\beta\to 0^+}\beta^{-\frac12}e_{\beta}(a_\beta)
				&\ge\frac{1}{2}\liminf_{\beta\to 0^+}\int_{\mathbb{R}^N}\bar{\varphi}_\beta\left(\frac{ m^2}{\sqrt{-\Delta+\beta m^2}+\sqrt{-\Delta}}\right)\varphi_\beta dx\\
				&\quad\quad
				+\frac{N}{2N+2}\liminf_{\beta\to 0^+}\left(\frac{1}{1-\beta}\|\varphi_\beta\|_{2+\frac2N}^{2+\frac2N}\right)\\
				&\ge\frac{1}{2}\int_{\mathbb{R}^N}\bar{\varphi}_0\left(\frac{ m^2}{2\sqrt{-\Delta}}\right)\varphi_0 dx+\frac{N}{2N+2}\|\varphi_0\|_{2+\frac2N}^{2+\frac2N}\\
				&=\frac{m^2}{4\eta}\int_{\mathbb{R}^N}\bar{{Q}}\left(\frac{1}{\sqrt{-\Delta}}\right){Q} dx+\frac{N\eta}{2N+2}\|{Q}\|_{2+\frac2N}^{2+\frac2N}\\
				&=\frac{m^2}{4\eta}\int_{\mathbb{R}^N}\bar{{Q}}\left(\frac{1}{\sqrt{-\Delta}}\right){Q} dx+\frac{Na^\ast\eta}{2}\\
				&\geq
				\left[\frac{ Na^\ast m^2}{2}\int_{\mathbb{R}^N}\bar{{Q}}\left(\frac{1}{\sqrt{-\Delta}}\right){Q} dx\right]^{\frac12}.
			\end{aligned}
		\end{equation*}
		Then, combining the previous inequality with \eqref{limsup17} we obtain
		$$\lim_{\beta\to 0^+}\beta^{-\frac12}e_{\beta}(a_\beta)
		=\frac{m^2}{4\eta}\int_{\mathbb{R}^N}\bar{{Q}}\left(\frac{1}{\sqrt{-\Delta}}\right){Q} dx+\frac{Na^\ast\eta}{2}
		=\left[\frac{Na^\ast m^2}{2}\int_{\mathbb{R}^N}\bar{{Q}}\left(\frac{1}{\sqrt{-\Delta}}\right){Q} dx\right]^{\frac12}$$
		and so, using the characterization of the minimizer of $g_2$ in Lemma \ref{lemA.7}, we conclude that
		$$\eta=\left[{\frac{m^2}{2Na^\ast}\int_{\mathbb{R}^N}\bar{{Q}}\left(\frac{1}{\sqrt{-\Delta}}\right){Q} dx}\right]^{\frac12}.$$

	\end{proof}

	\vskip 0.16truein
	\noindent {\bf Acknowledgements:}
	P.D., A.P., and G.S. are members of GNAMPA-INdAM and were supported by  PRIN PNRR, P2022YFAJH “Linear and Nonlinear PDEs: New directions and applications” 
	and GNAMPA-INdAM Project 2025 (CUP E5324001950001).
	P.D. and A.P. are also supported by GNAMPA-INdAM Project 2026 (CUP E53C25002010001) and by the
	Italian Ministry of University and Research under the Program Department of Excellence L. 232/2016
	(CUP D93C23000100001).
	G.S. was also supported by Capes, CNPq projects
	402514/2024-6 and 304244/2023-6,
	Fapesp projects 2022/16407-1 and 2022/16097-2.
	
	This work was partially completed during L.Y.’s visit to the Dipartimento di Meccanica, Matematica e Management, Politecnico di Bari, supported by China Scholarship Council (CSC). L.Y. appreciates the institution for its excellent academic environment and warm hospitality.
	
	\smallskip
	{\bf Data availability} Data sharing not applicable to this article as no datasets were generated or analysed during the current study.
	
	\smallskip
	{\bf Conflict of interest} The authors declare that they have no conflict of interest.


\begin{thebibliography}{40}
		\bibitem{A16} V. Ambrosio, {\em Ground states solutions for a non-linear equation involving a pseudo-relativistic Schr\"{o}dinger operator}, J. Math. Phys., 57 (2016), no. 5, 051502, 18 pp.
		
		\bibitem{BF25} J. Bellazzini, L. Forcella, {\em Mass-subcritical Half-Wave Equation with mixed nonlinearities: existence and non-existence of ground states}, 	arXiv:2504.07488.
		
		\bibitem{BGLV19} J. Bellazzini, V. Georgiev, E. Lenzmann, N. Visciglia, {\em On traveling solitary waves and absence of
			small data scattering for nonlinear half-wave equations}, Comm. Math. Phys., {\bf 372} (2019), no. 2, 713--732.
		
		\bibitem{BGV18} J. Bellazzini, V. Georgiev, N. Visciglia,  {\em Long time dynamics for semi-relativistic NLS and half wave in arbitrary dimension}, Math. Ann., 371 (2018), no. 1-2, 707--740.
		
		
		\bibitem{BP24} F. Bernini, P. d'Avenia, {\em On a fractional magnetic pseudorelativistic operator: properties and applications}, arXiv:2410.22426.
		
		
		
		\bibitem{CHS18} W. Choi, Y. Hong, J. Seok, {\em Optimal convergence rate and regularity of nonrelativistic limit for the nonlinear pseudo-relativistic equations}, J. Funct. Anal., 274 (2018), no. 3, 695--722.
		
		\bibitem{CS16} W. Choi, J. Seok, {\em Nonrelativistic limit of standing waves for pseudo-relativistic nonlinear Schr\"{o}dinger equations}, J. Math. Phys., 57 (2016), no. 2, 021510, 15 pp.
		
		
		\bibitem{CZN11} V. Coti Zelati, M. Nolasco, {\em Existence of ground states for nonlinear, pseudo-relativistic Schr\"{o}dinger equations}, Atti Accad. Naz. Lincei Rend. Lincei Mat. Appl., 2011, 22(1): 51--72.
		
		
		\bibitem{CZN13} V. Coti Zelati, M. Nolasco, {\em Ground states for pseudo-relativistic Hartree equations of critical type}, Rev. Mat. Iberoam. 29 (2013), no. 4, 1421--1436.
		
		
		\bibitem{DS}
		P. d’Avenia, M. Squassina, {\em Ground states for fractional magnetic operators}, ESAIM Control Optim. Calc. Var., 24 (2018), 1--24.
		
		\bibitem{DPV12} E. Di Nezza,  G. Palatucci, E. Valdinoci, {\em Hitchhiker's guide to the fractional Sobolev spaces}, Bull. Sci. Math., 136 (2012), no. 5, 521--573.
		
		
		\bibitem{ES07} A. Elgart, B. Schlein, {\em Mean field dynamics of boson stars}, Comm. Pure Appl. Math., 60 (2007), no. 4, 500--545.
		
		
		\bibitem{EHP}
		A. Esfahani, H. Hajaiej, A. Pomponio, {\em New insights into the solutions of a class of anisotropic nonlinear Schr\"odinger equations on the plane}, Milan J. Math. 93 (2025), 349--393.
		
		\bibitem{FF} M. M. Fall, V. Felli, {\em Unique continuation properties for relativistic Schr\"odinger operators with a singular potential}, Discrete Contin. Dyn. Syst. 35 (2015), 5827--5867.
		
		
		\bibitem{FQT12} P. Felmer, A. Quaas, J. G. Tan, {\em Positive solutions of the nonlinear Schr\"{o}dinger equation with the fractional Laplacian}, Proc. Roy. Soc. Edinburgh Sect., A 142 (2012), no. 6, 1237--1262.
		
		\bibitem{FLS} R.L. Frank, E. Lenzmann, L. Silvestre, {\em Uniqueness of radial solutions for the fractional Laplacian}. Commun. Pure Appl. Math {\bf 69} (2016), 1671--1726.
		
		
		\bibitem{FJL07} J. Fr\"{o}hlich, B. L. G. Jonsson, E. Lenzmann, {\em Boson stars as solitary waves}, Comm. Math. Phys., {\bf 274} (2007), no. 1, 1--30.
		
		
		\bibitem{FL07} J. Fr\"{o}hlich, E. Lenzmann, {\em Blowup for nonlinear wave equations describing boson stars}, Comm. Pure Appl. Math., 60 (2007), no. 11, 1691--1705.
		
		
		\bibitem{GL22} V. Georgiev, Y. Li,  {\em Nondispersive solutions to the mass critical half-wave equation in two dimensions}, Comm. Partial Differential Equations, 47 (2022), no. 1, 39--88.
		
		
		\bibitem{GZ17} Y. J. Guo, X. Y. Zeng, {\em Ground states of pseudo-relativistic boson stars under the critical stellar mass}, Ann. Inst. H. Poincar\'{e} C Anal. Non Lin\'{e}aire, 2017, 34(6): 1611--1632.
		
		\bibitem{GZ20}  Y. J. Guo, X. Y. Zeng, {\em The Lieb-Yau conjecture for ground states of pseudo-relativistic Boson stars}, J. Funct. Anal., 278 (2020), no. 12, 108510, 24 pp.
		
		\bibitem{HYZ24} Q. H. He, L. F. Yang, X. Y. Zeng, {\em Existence and limiting profiles of boosted ground states for the pseudo-relativistic Schr\"{o}dinger equation with focusing power type nonlinearity}, Potential Anal. 63 (2025), no. 1, 295--328.
		
		
		\bibitem{HS}
		Y. Hong, Y. Sire, {\em A new class of traveling solitons for cubic fractional nonlinear Schr\"odinger equations}, Nonlinearity 30, (2017), 1262.
		
		\bibitem{KLR13} J. Krieger, E. Lenzmann, P. Rapha\"{e}l, {\em Nondispersive solutions to the $L^2$-critical half-wave equation}, Arch. Ration. Mech. Anal., 209 (2013), no. 1, 61--129.
		
		\bibitem{L09} E. Lenzmann, {\em Uniqueness of ground states for pseudorelativistic Hartree equations}, Anal. PDE, 2 (2009), no. 1, 1--27.
		
		\bibitem{LL11} E. Lenzmann, M. Lewin, {\em On singularity formation for the $L^2$-critical Boson star equation}, Nonlinearity, 24 (2011), no. 12, 3515--3540.
		
		
		\bibitem{LZW} Y. Li, D. Zhao, Q. X. Wang, {\em Existence of the stable traveling wave for half-wave equation with $L^2$-critical combined nonlinearities},
		Appl. Anal., {\bf 101} (2022), no. 7, 2498--2510.
		
		\bibitem{LL01} E. H. Lieb, M. Loss,  {\em Analysis. Second edition}, Graduate Studies in Mathematics, 14. American Mathematical Society, Providence, RI, 2001.
		
		\bibitem{LS10} E. H. Lieb, R. Seiringer,  {\em The stability of matter in quantum mechanics}, Cambridge University Press, Cambridge, 2010.
		
		\bibitem{LT84} E. H. Lieb, W. E. Thirring, {\em Gravitational collapse in quantum mechanics with relativistic kinetic energy}, Ann. Physics, 155 (1984), no. 2, 494--512.
		
		\bibitem{LY87} E. H. Lieb, H.-T. Yau, {\em The Chandrasekhar theory of stellar collapse as the limit of quantum mechanics}, Comm. Math. Phys., 112 (1987), no. 1, 147--174.
		
		\bibitem{LY88} E. H. Lieb,  H.-T. Yau, {\em The stability and instability of relativistic matter}, Comm. Math. Phys. 118 (1988), no. 2, 177--213.
		
		\bibitem{Lion} P. L. Lions, {\em The concentration-compactness principle in the calculus of variations}, The locally compact case. I. Ann. Inst. H. Poincar\'e Anal. Non Lin\'eaire, {\bf 1} (1984), no. 2, 109--145.
		
		\bibitem{LW22} H. J. Luo, D. Wu, {\em Normalized ground states for general pseudo-relativistic Schr\"{o}dinger equations}, Appl. Anal., 101 (2022), no. 9, 3410--3431.
		
		
		
		
		
		
		
		
		\bibitem{P14} F. Pusateri, {\em Modified scattering for the boson star equation}, Comm. Math. Phys., 332 (2014), no. 3, 1203--1234.
		
		
		
		
		
		\bibitem{S19} S. Secchi, {\em A generalized pseudorelativistic Schr\"{o}dinger equation with supercritical growth}, Commun. Contemp. Math., 21 (2019), no. 8, 1850073, 21 pp.
		
		
		
		\bibitem{W21} Q. X. Wang, {\em A blow-up result for the travelling waves of the pseudo-relativistic Hartree equation with small velocity}, Math. Methods Appl. Sci., 44 (2021), no. 13, 10403--10415.
		
		
		
		
		
		
		\bibitem{ZL22} G. Q. Zhang, Y. W. Li, {\em Normalized ground state traveling solitary waves for the half-wave equations with combined nonlinearities},
		Z. Angew. Math. Phys. {\bf 73} (2022), no. 4, Paper No. 142, 27 pp.
		
		
		\bibitem{Z17} S. H. Zhu, {\em Existence of stable standing waves for the fractional Schr\"{o}dinger equations with combined nonlinearities}, J. Evol. Equ., 17 (2017), no. 3, 1003--1021.
		
		
		
		
		
	\end{thebibliography}
\end{document}